\documentclass[a4paper, 10pt, american]{amsart}

\usepackage{times,latexsym,amssymb}
\usepackage{amsmath,amsthm,bm}
\usepackage{color}
\usepackage{bm}
\usepackage[colorlinks,pdfpagelabels,pdfstartview = FitH,bookmarksopen
= true,bookmarksnumbered = true,linkcolor = blue,plainpages =
false,hypertexnames = false,citecolor = red,pagebackref=false]{hyperref}

\usepackage{amsbsy}
\usepackage{amstext}
\usepackage{amssymb}
\usepackage{esint}
\usepackage{stmaryrd}
\usepackage{graphicx}

\allowdisplaybreaks
\newtheorem{theorem}{Theorem}
\newtheorem{lemma}[theorem]{Lemma}
\newtheorem{definition}[theorem]{Definition}

\newtheorem{remark}[theorem]{Remark}
\numberwithin{theorem}{section}
\numberwithin{equation}{section}

\newcommand{\mint}{- \mskip-19,5mu \int}

\def\N{\mathbb{N}}
\def\R{\mathbb{R}}

\renewcommand{\d}{\:\! \mathrm{d}}
\newcommand{\dx}{\mathrm{d}x}

\newcommand{\dt}{\mathrm{d}t}

\newcommand{\lbbracket}{\boldsymbol{(}}
\newcommand{\rbbracket}{\boldsymbol{)}}
\newcommand{\dtau}{\mathrm{d}\tau}

\renewcommand{\epsilon}{\varepsilon}

\DeclareMathOperator{\Div}{div}

\DeclareMathOperator{\dist}{dist}

\renewcommand{\epsilon}{\varepsilon}
\newcommand{\eps}{\varepsilon}
\renewcommand{\rho}{\varrho}

\def\eqn#1$$#2$${\begin{equation}\label#1#2\end{equation}}

\newcommand{\power}[2]{\bm{#1^{\mbox{\unboldmath{\scriptsize$#2$}}}}}

\newcommand{\babs}[1]{\big|#1\big|}

\def\Xint#1{\mathchoice
    {\XXint\displaystyle\textstyle{#1}}%
    {\XXint\textstyle\scriptstyle{#1}}%
    {\XXint\scriptstyle\scriptscriptstyle{#1}}%
    {\XXint\scriptscriptstyle\scriptscriptstyle{#1}}%
    \!\int}
\def\XXint#1#2#3{\setbox0=\hbox{$#1{#2#3}{\int}$}
    \vcenter{\hbox{$#2#3$}}\kern-0.5\wd0}
\def\bint{\Xint-}
\def\dashint{\Xint{\raise4pt\hbox to7pt{\hrulefill}}}

\def\Xiint#1{\mathchoice
    {\XXiint\displaystyle\textstyle{#1}}%
    {\XXiint\textstyle\scriptstyle{#1}}%
    {\XXiint\scriptstyle\scriptscriptstyle{#1}}%
    {\XXiint\scriptscriptstyle\scriptscriptstyle{#1}}%
    \!\iint}
\def\XXiint#1#2#3{\setbox0=\hbox{$#1{#2#3}{\iint}$}
    \vcenter{\hbox{$#2#3$}}\kern-0.5\wd0}
\def\biint{\Xiint{-\!-}}

\renewcommand{\b}{\mathfrak{b}}
\renewcommand{\u}{\boldsymbol{u}}

\renewcommand{\a}{\boldsymbol{a}}
\newcommand{\g}{\boldsymbol{g}}
\newcommand{\A}{\mathbf{A}}
\newcommand{\ca}{\operatorname{cap}}

\newcommand{\G}{\mathbf{G}}


\author[K. Moring]{Kristian Moring}
\address{Kristian Moring\\
Department of Mathematics and Systems Analysis, Aalto University\\
P.~O.~Box 11100, FI-00076 Aalto, Finland}
\email{kristian.moring@aalto.fi}

\author[C. Scheven]{Christoph Scheven}
\address{Christoph Scheven\\ Fakult\"at f\"ur Mathematik, 
Universit\"at Duisburg-Essen\\45117 Essen, Germany}
\email{christoph.scheven@uni-due.de}

\author[S. Schwarzacher]{Sebastian Schwarzacher}
\address{Sebastian Schwarzacher\\Katedra matematick\'e analy\'zy, Matematicko-fyzik\'aln\'\i\
  fakulta Univerzity Karlovy, Sokolovsk\'a 83, 186 75 Praha 8, Czech Republic}
  \email{schwarz@karlin.mff.cuni.cz}

\author[T. Singer]{Thomas Singer}
\address{Thomas Singer\\
Department Mathematik, Friedrich-Alexander-Universit\"at  Er\-lang\-en-N\"urnberg\\
Cauerstrasse 11, 91058 Erlangen, Germany}
\email{thomas.singer@fau.de}

\begin{document}
\title[Global higher integrability of weak solutions of porous medium systems]{Global higher integrability of weak solutions of\\ porous medium systems}
\date{August 22, 2019}


\maketitle

\begin{abstract}
  We establish higher integrability up to the
  boundary for the gradient of solutions to
  porous medium type systems, whose model case is given by 
  \begin{equation*}
    \partial_t u-\Delta(|u|^{m-1}u)=\mathrm{div}\,F\,,
  \end{equation*}
  where $m>1$. More precisely, we prove that under suitable
  assumptions the spatial gradient
  $D(|u|^{m-1}u)$ of any weak solution is integrable to a larger power
  than the natural power $2$.
  Our analysis includes both the case of the lateral boundary and 
  the initial boundary. 
\end{abstract}

\maketitle

\section{Introduction}

We are concerned with the boundary regularity of solutions to
Cauchy-Dirichlet problems of the form
\begin{equation*}
  \left\{
  \begin{array}{ll}
    \partial_t
    u-\mathrm{div}\,\mathbf{A}(x,t,u,D(|u|^{m-1}u))=\mathrm{div}\,
    F&\mbox{in }\Omega_T:=\Omega\times(0,T),\\[1ex]
    u=g&\mbox{on }\partial_{\mathrm{par}}\Omega_T,
  \end{array}
  \right.
\end{equation*}
for vector-valued solutions $u:\Omega_T\to\R^N$, where $m>1$,
the domain $\Omega\subset\R^n$ is bounded with dimension $n\ge2$,
the dimension of the target space is $N\in\N$, 
and $\partial_{\mathrm{par}}\Omega_T$ denotes the parabolic
boundary of the space-time cylinder $\Omega_T$, where $T>0$.
We cover a large class of vector fields
$\mathbf{A}$ that we only require to satisfy growth and ellipticity
conditions corresponding to the model case
$\mathbf{A}(x,t,u,\zeta)=\zeta$ of the porous medium
system. The assumptions on the data are made precise in Section~\ref{sec:statement} below. Our
starting point are weak solutions, by which we mean in particular that
the spatial gradient satisfies $D(|u|^{m-1}u)\in L^2(\Omega_T)$. Our
goal is to establish the self-improving property of integrability up
to the boundary in the sense that
$D(|u|^{m-1}u)\in L^{2+\varepsilon}(\Omega_T)$ holds true for some
$\varepsilon>0$.

The question for higher integrability of solutions has a long history 
that starts with the classical work by 
Elcrat \& Meyers \cite{Meyers-Elcrat} on 
elliptic systems of $p$-Laplace type, which in turn is based on the
work of Gehring \cite{Gehring}. 
Since then, similar results have been established for a variety of
other elliptic problems, and the higher integrability of solutions has
proved to be a very useful tool for the derivation of further
regularity results. We refer to
\cite{Giaquinta-Modica:0, Giaquinta-Modica,Giaquinta:book,Giusti:book}
and the references therein. The question of higher integrability up to
the boundary for equations of $p$-Laplace type has been answered
positively by Kilpel\"ainen \& Koskela
\cite{Kilpelainen-Koskela}. They observed that the natural condition 
to impose on the regularity of the domain $\Omega\subset\R^n$ is the
property of uniform $p$-thickness of the complement
$\R^n\setminus\Omega$, see \cite[Rem. 3.3]{Kilpelainen-Koskela}.  

The first higher integrability result for a parabolic problem is due
to Giaquinta \& Struwe \cite{Giaquinta-Struwe}, who treated the
quasilinear case. However, it turned out that the techniques of Elcrat
\& Meyers could not directly be extended to the case of the parabolic
$p$-Laplace system due to the anisotropic scaling behaviour of this
system. This problem was solved by Kinnunen \& Lewis in
\cite{Kinnunen-Lewis:1} for weak solutions to $p$-Laplace type systems.
The much more intricate case of very weak solutions was settled by the
same authors in \cite{Kinnunen-Lewis:very-weak}.
Their approach relies on the idea of intrinsic cylinders by
DiBenedetto, see \cite{DbF2,DiBe,DBGV-book}. The heuristic idea is to
compensate for the inhomogeneity of the parabolic $p$-Laplace operator
$\partial_tu-\mathrm{div}\,(|Du|^{p-2}Du)$
by working with cylinders that depend on the size of $|Du|$. More
precisely, for a parameter $\lambda>0$ that is in some sense comparable
to $|Du|$, the idea by DiBenedetto is to consider 
cylinders of the form $Q_r^{(\lambda)}(x_o,t_o)=B_r(x_o)\times
(t_o-\lambda^{2-p}r^2,t_o+\lambda^{2-p}r^2)$.

The boundary version of the
higher integrability result for the parabolic $p$-Laplacian has been
established by Parviainen \cite{Parviainen,Parviainen-singular}, see
also B\"ogelein \& Parviainen \cite{Boegelein:1,Boegelein-Parviainen} for the
higher order case. The required regularity of the boundary is the same
as in the case of the elliptic $p$-Laplacian, i.e. the complement of
the domain is assumed to be uniformly $p$-thick. Finally, we note that
Adimurthi \& Byun \cite{AdimurthiByun:2018} proved global higher
integrability even for very weak solutions of parabolic $p$-Laplace
equations. 

Even after the case of the parabolic $p$-Laplace equation had been
quite well understood, the corresponding question for porous medium type
equations stayed open for a long time. This case turned out to pose
additional challenges, which stem from the fact that the differential
operator $\partial_t u-\Delta u^m=\partial_t
u-m\mathrm{div}(u^{m-1}Du)$ can degenerate depending on the size of
$u$, and not on the size of the gradient as for 
the parabolic $p$-Laplace. This type of degeneracy makes it much more
involved to derive gradient estimates, because both the size of the
solution and of the gradient have to be taken into account.
In particular, it is natural to work
with intrinsic cylinders of the type
\begin{equation}\label{intrinsic-cylinder-PME}
  Q_\rho^{(\theta)}(x_o,t_o)
  =
  B_\rho(x_o)\times
  \big(t_o-\theta^{1-m}\rho^{\frac{m+1}m},
       t_o+\theta^{1-m}\rho^{\frac{m+1}m}\big),
\end{equation}
where $\theta^m$ corresponds to $\frac1\rho u^m$. The construction of a
family of such intrinsic cylinders that is suitable for the derivation
of gradient estimates has first been established by Gianazza and the
third author in \cite{Gianazza-Schwarzacher},
using an idea from \cite{Schwarzacher}. The article
\cite{Gianazza-Schwarzacher} contains the first result
on higher integrability of the gradient for porous medium type
equations and opened the path to further results in this direction.
The higher integrability result was already extended to systems in 
\cite{Boegelein-Duzaar-Korte-Scheven}, to singular
porous medium equations and systems, i.e. the case $m<1$, in
\cite{Gianazza-Schwarzacher-singular,Boegelein-Duzaar-Scheven:2018},
and to a doubly nonlinear system in \cite{BDKS:2018}. All of the
mentioned results are restricted to the interior case. 
The present article is devoted to the question whether the higher
integrability of the gradient can be extended up to the boundary.
As to be expected from the
$p$-Laplacian case, we have to assume that the complement of the
domain is uniformly $2$-thick. However, it turns out that we need a
further assumption on the domain in the case of the
porous medium equation.
The additional problem stems from the fact that the degeneracy of the
porous medium equation depends on the values of the solution itself
rather than on the gradient. This means that close to the boundary,
the degeneracy also depends on the value of the boundary values. In
order to rebalance this nonlinearity with the help of intrinsic
cylinders of the type \eqref{intrinsic-cylinder-PME}, we need to
estimate the difference of the boundary values and the constant
$\theta$ by means of a suitable Poincar\'e inequality on cylinders
centred on the boundary. In order to obtain such an 
inequality for arbitrary boundary data, we have to restrict ourselves
to Sobolev extension domains.
The exact assumptions will be
given in the following section.

\medskip
 
\noindent
{\bf Acknowledgments.} T.~Singer has been supported by the DFG-Project  SI 2464/1-1 ``Highly nonlinear evolutionary problems". K.~Moring has been supported by the Magnus Ehrnrooth foundation. 

\subsection{Statement of the result}
\label{sec:statement}

We consider Cauchy-Dirichlet problems of the form
\begin{align}
\label{equation:PME}
\left\{
\begin{array}{cl}
\partial_t u- \Div \A(x,t,u, D\u^m) =\Div F &\text{ in } \Omega_T, \vspace{1mm}\\
u=g & \text{ on } \partial_{\mathrm{par}}\Omega_T,
\end{array}
\right. 
\end{align}
with $u\colon \Omega_T \to \R^N$, where
$\A:\Omega_T\times\R^N\times\R^{Nn}\to\R^{Nn}$ is a Carath\'eodory
function satisfying
\begin{align}
\label{assumption:A}
\begin{aligned}
\left\{
\begin{array}{c}
\A(x,t,u,\zeta)\cdot \zeta \geq \nu |\zeta|^2 \\
|\A(x,t,u, \zeta)| \leq L|\zeta|
\end{array}
\right.
\end{aligned}
\end{align}
for a.e.\ $(x,t)\in \Omega_T$ and any $(u,\zeta)\in \R^n \times \R^{Nn}$. Note that for $u\in \R^N$ we used the short hand notation
$$
\u^\alpha = |u|^{\alpha-1}u
$$
for $\alpha >0$, where we interpret $\u^\alpha$ as zero if $u$ is zero. 
For the inhomogeneity $F\colon \Omega_T \to \R^{Nn}$ we assume that
\begin{equation}\label{assumption:F}
F\in L^{2+\eps}(\Omega_T,\R^{Nn}),
\end{equation}
and for the boundary datum $g \colon \Omega_T\to \R^N$ we suppose that 
\begin{align}
  \label{assumption:g}
  \left\{
\begin{aligned}
 \g^m\in L^{2+\eps}&\big( 0,T;W^{1,2+\eps}(\Omega,\R^N)\big), ~g\in C^0\big([0,T],L^{m+1}(\Omega,\R^N)\big),\\ 
&~\text{ and }~  \partial_t\power{g}{m}\in L^{\frac{m(2+\eps)}{2m-1}}(\Omega_T,\R^N) 
\end{aligned}\right.
\end{align}
for some $\eps>0$.

We consider weak solutions in the following sense.

\begin{definition}
\label{global_solution}
A measurable map $u\colon \Omega_T \to \R^N$ in the class 
$$
\u^m \in L^2\big(0,T;W^{1,2}(\Omega,\R^N)\big) ~\text{ with }~ u \in C^0\big([0,T],L^{m+1}(\Omega,\R^N)\big)   
$$
is called a global weak solution to the Cauchy-Dirichlet problem \eqref{equation:PME} if 
\begin{align}
 \label{weak:PME}
\iint_{\Omega_T} \big[ u \cdot \partial_t \varphi -\A(x,t,u,D\u^m) \cdot D \varphi \big] \, \d x \d t= \iint_{\Omega_T} F \cdot D\varphi\, \d x\d t
\end{align}
holds true for every test-function $\varphi\in C^\infty_0(\Omega_T,\R^N)$ and, moreover
\begin{align*}
(\u^m-\g^m)(\cdot,t) \in W^{1,2}_0(\Omega,\R^N)\quad \text{ for almost every } t \in (0,T) 
\end{align*}
and 
\begin{align}
\label{assumption:initial_datum}
\frac 1 h \int_0^h \int_\Omega \big|\u^{\frac{m+1}{2}}(x,t)-\g^{\frac{m+1}{2}}(x,0)\big|^{2} \d x\d t \to 0 ~~\text{ as } h \downarrow 0
\end{align}
for a given function $g$ satisfying \eqref{assumption:g}.
\end{definition}

In order to state our assumptions on the boundary of the domain, we
recall the following two definitions. The first one is already
familiar from corresponding results for $p$-Laplace equations.

\begin{definition}\label{def:p-thick}
A set $E \subset \R^n$ is uniformly $p$-thick if there exist constants $\mu,\rho_o>0$ such that 
$$
\ca_p(E\cap \overline B_\rho(x_o),B_{2\rho}(x_o)) \geq \mu \ca_p(\overline B_\rho(x_o),B_{2\rho}(x_o))
$$
for all $x_o\in E$ and for all $0<\rho<\rho_o$.
\end{definition}

For the treatment of the porous medium equation, we rely on a suitable
Poincar\'e inequality for the boundary values, see Lemma
\ref{lem:poincare_g}. In order to achieve our main result for arbitrary
boundary values, we need to assume that $\Omega$ is a Sobolev
extension domain in the following sense.

\begin{definition}\label{def:sobolev-extension}
  A domain $\Omega\subset\R^n$ is called a $W^{1,p}$-extension domain
  if there exists a linear operator $E:W^{1,p}(\Omega)\to
  W^{1,p}(\R^n)$ such that $Eu(x)=u(x)$ for a.e. $x\in\Omega$ and
  \begin{equation}\label{bounded-extension}
    \|Eu\|_{W^{1,p}(\R^n)}\le c_E\|u\|_{W^{1,p}(\Omega)}
  \end{equation}
  for any $u\in W^{1,p}(\Omega)$ and a constant $c_E\in\R_{\ge 0}$. 
\end{definition}

In \cite{HaKoTu} it was shown that every $W^{1,p}$-extension domain satisfies the measure density condition, i.e.\ there exists $\alpha>0$ such that for all $x_o \in \overline \Omega$ and $0 < \rho \le 1$
\begin{align}
\label{measure_density}
|\Omega \cap B_\rho(x_o)|\geq \alpha |B_\rho(x_o)|
\end{align} 
holds true.

This allows us to formulate the main result of our paper. In order to
state the local estimate, we consider parabolic cylinders of the form
$$
  Q_R(x_o,t_o):=B_R(x_o)\times(t_o-R^{\frac{m+1}m},t_o+R^{\frac{m+1}m}).
$$ 

\begin{theorem}\label{thm:main}
  Let $m> 1$.
  There exist constants $\eps_o\in(0,1]$ and $c\ge1$ so that the
  following holds. Assume that for some $\eps\in(0,\eps_o]$, the
  assumptions \eqref{assumption:A},
  \eqref{assumption:F}, and
  \eqref{assumption:g} are in force and that
  $\Omega\subset\R^n$ is a
  bounded $W^{1,2+\eps}$-extension domain for which the complement
  $\R^n\setminus\Omega$ is uniformly $2$-thick.
  Then any
  global weak solution $u$ to the Cauchy-Dirichlet problem \eqref{equation:PME} in the sense of
  Definition \ref{global_solution} satisfies
  \begin{equation*}
    D\power um\in L^{2+\eps}\big(\Omega_T,\R^{Nn}\big).
  \end{equation*}
    Moreover, for any parabolic cylinder $Q_{2R}(z_o)\subset\R^n\times(-T,T)$ with $z_o \in \Omega_T\cup\partial_{\mathrm{par}}\Omega_T$ we have 
\begin{align} \label{local-estimate}
&\iint_{Q_R \cap \Omega_T}\babs{D\u^m}^{2+\eps} \, \d x \d t \\\nonumber
&\qquad\leq c  \Bigg( 1 + \biint_{Q_{2R}\cap\Omega_T}
 \frac{|\u^m - \g^m|^2 }{R^2}\,\d x\d t
 \Bigg)^\frac{\eps m}{m+1}
 \iint_{Q_{2R} \cap \Omega_T} \babs{D\u^m}^2 \, \d x\d t \\\nonumber
 &\qquad\qquad+
 c  \Bigg(\biint_{Q_{2R}\cap\Omega_T} G_R^{2+\eps}  \, \d x \d
 t\Bigg)^{\frac{2\eps m}{(2+\eps)(m+1)}}
 \iint_{Q_{2R} \cap \Omega_T} \babs{D\u^m}^2 \, \d x\d t \\\nonumber
 &\qquad\qquad+
 c \iint_{Q_{2R}\cap\Omega_T} G_R^{2+\eps} \, \d x \d t,
\end{align}
where we abbreviated
\begin{equation*}
  G_R^2:=|\partial_t \g^m|^{\frac{2m}{2m-1}}+|D\g^m|^2+\frac{|g|^{2m}}{R^2}+|F|^{2}.
\end{equation*}
The constant $\eps_o$ depends
  at most on $m,n,N,\nu,L,\mu,\rho_o$, and $\alpha$, and $c$ depends
  on the same data and additionally on $c_E$. Here, the parameters
  $\mu,\rho_o$ are introduced in Definition~\ref{lem:p-thick} with
  $p=2$, $c_E$ is the constant from
  Definition~\ref{def:sobolev-extension} with $p=2+\eps$ and $\alpha$
  is given by \eqref{measure_density}.
\end{theorem}

  \begin{remark}
    A close inspection of the proof shows that the constants
    in the preceding theorem actually depend continuously on $m>1$ and remain
    bounded when $m\downarrow1$. 
  \end{remark}
\subsection{Technical novelties and plan of the paper}

It has been observed by Gianazza and the third author in
\cite{Gianazza-Schwarzacher} that higher integrability in the interior
of the domain can be derived by working with cylinders
$Q_\rho^{(\theta)}(z_o)$ that are intrinsic in the sense
\begin{equation}\label{scaling-interior}
  \biint_{Q_\rho^{(\theta)}(z_o)}\frac{|u|^{2m}}{\rho^2}\d x\d t
  \approx\theta^{2m}.
\end{equation}
A coupling of this type is necessary in order to deal with the
degeneracy of the porous medium equation. 
This already becomes apparent in
the Caccioppoli type inequality, which is the first step towards any higher
integrability result. The interior version of this inequality is stated in
Lemma \ref{lem:initialcacciop} below. The time
derivative in the porous medium equation leads to an integral that is
comparable to 
\begin{align}\label{time-term}
  \biint_{Q_\rho^{(\theta)}(z_o)}\theta^{m-1}\frac{\big|\power{u}{\frac{m+1}2}-\power{a}{\frac{m+1}2}\big|^2}{\rho^{\frac{m+1}m}}\d  x\d t, 
\end{align}
where we choose the constant $a$ according to
$\power{a}{m}:=\biint_{Q_\rho^{(\theta)}(z_o)}\power{u}{m}\,\d x\,\d t$,
while the diffusion term results in an integral of the form
\begin{align}\label{diff-term}
  \biint_{Q_\rho^{(\theta)}(z_o)}\frac{\big|\power{u}{m}-\power{a}{m}\big|^2}{\rho^2}\d
  x\d t
  \approx
  \biint_{Q_\rho^{(\theta)}(z_o)}\frac{(|u|+|a|)^{m-1}\big|\power{u}{\frac{m+1}2}-\power{a}{\frac{m+1}2}\big|^2}{\rho^2}\d  x\d t.
\end{align}
The occurence of these two integrals in the Caccioppoli type
inequality is a natural consequence of the inhomogeneity of the porous
medium equation.
Heuristically, on a cylinder that satisfies an
intrinsic coupling of the type \eqref{scaling-interior},
the two integrals \eqref{time-term} and
\eqref{diff-term} are comparable, which makes it possible to deal
with the inhomogeneous form of the Caccioppoli inequality.
More precisely, for the estimate of \eqref{diff-term} by a Sobolev-Poincar\'e type
inequality, it is sufficient to work
with cylinders that are
\emph{sub-intrinsic} in the sense that the integral in
\eqref{scaling-interior} is only bounded from above by $\theta^{2m}$.
However, in order to estimate \eqref{time-term} by \eqref{diff-term},
it is necessary to bound $\theta$ from above. To this end, in
\cite{Gianazza-Schwarzacher}, Gianazza and the third author 
distinguished between the degenerate case, in which $\theta$ can be
bounded by an integral of the spatial derivative, and the
non-degenerate case, in which an intrinsic coupling of the type
\eqref{scaling-interior} can be achieved. A key step in their proof is the
construction of a suitable system of sub-intrinsic cylinders on which
either the degenerate or the non-degenerate case applies. The
combination of the Caccioppoli and the Sobolev-Poincar\'e inequalities
then leads to a reverse H\"older inequality on these cylinders, and a Vitali
type covering argument yields the desired higher integrability result
in the interior.

In the Caccioppoli inequality close to the lateral boundary, it is
more natural to subtract the boundary values from the solution rather
than the mean value. As a consequence, 
the suitable choice of the scaling parameter $\theta$ has to depend 
on the boundary values as well.
In the boundary situation, we thus work
with cylinders that satisfy a coupling of the type
\begin{equation}\label{scaling-boundary}
    \biint_{Q_\rho^{(\theta)}(z_o)}
    2\frac{|\power um-\power gm|^{2}+|g|^{2m}}{\rho^2}\d x\d t
    \approx\theta^{2m}.
\end{equation}
Both of the coupling conditions \eqref{scaling-interior} and \eqref{scaling-boundary} have to be taken into account for
the construction of a system of sub-intrinsic cylinders as in
\cite{Gianazza-Schwarzacher}. In fact, when considering a point $z_o$
close to the lateral boundary, it is not clear a priori if the
mentioned construction yields a cylinder for which the doubled
cylinder $Q_{2\rho}^{(\theta)}(z_o)$ touches
the boundary or not. This is the reason why both the interior scaling
\eqref{scaling-interior} and the boundary scaling
\eqref{scaling-boundary} enter in the construction of the cylinders, cf. Section~\ref{construction_cylinder}.    
As a matter of course, the derivation of the desired reverse H\"older
inequalities on these cylinders requires a much more
extensive case-by-case analysis than in the interior case.

At the initial boundary, we use an extension argument in order to
avoid the occurrence of a third type of coupling condition. More
precisely, we extend the solution by the reflected boundary values,
cf. \eqref{def:hatu} below. Then we use a scaling as in
\eqref{scaling-interior} with $u$ replaced by its extension. This
enables us to treat the initial boundary case with a 
coupling condition analogous to the interior.

This article is organized as follows. In the preliminary Section~\ref{sec:preliminaries},
we collect some technical tools that will be crucial for the proof. In
Section~\ref{sec:caccioppoli} , we derive suitable Caccioppoli type estimates and Section~\ref{sec:poincare} is devoted to Sobolev-Poincar\'e type inequalities for the
solutions. Both estimates are combined in Section~\ref{sec:reverse-holder} to establish
reverse H\"older type inequalities on sub-intrinsic cylinders. Each of
the three last-mentioned sections is subdivided into one subsection
that is concerned with the case of the lateral boundary and another
one that deals with the initial boundary. Moreover, for the derivation
of the reverse H\"older inequality, we have to consider two different
types of coupling conditions for the sub-intrinsic cylinders that can
be understood as the non-degenerate case
(see \eqref{intrinsic_coupling} for the lateral boundary and
\eqref{initial_intrinsic_coupling} for the initial boundary)
and the degenerate case (cf. \eqref{sub-intrinsic_coupling} and
\eqref{initial_intrinsic_coupling_2}, respectively).
The final Section~\ref{sec:hi-int} contains the construction of a suitable system
of cylinders, which can be shown to satisfy one of the mentioned
coupling conditions that lead to a reverse H\"older inequality. By a
Vitali type covering argument, the reverse H\"older estimates on the
cylinders can be extended to estimates on the super-level sets. Then,
a standard Fubini type argument yields the result.

\section{Preliminaries}
\label{sec:preliminaries}

\subsection{Notation}
For $z_o=(x_o,t_o)\in \Omega_T$ we set
$$
Q_\rho^{(\theta)}(z_o) :=B_\rho(x_o)\times \Lambda_\rho^{(\theta)}(t_o),
$$
where $B_\rho(x_o)$ denotes the open ball with radius $\rho>0$ and center
$x_o$ and 
$$
\Lambda_\rho^{(\theta)}(t_o) :=(t_o-\theta^{1-m}\rho^{\frac{m+1}{m}},t_o+\theta^{1-m}\rho^{\frac{m+1}{m}}).
$$
In the case $\theta=1$ we use the shorter notation
$Q_\rho(z_o):=Q_\rho^{(1)}(z_o)$.
From the definition of the cylinders it becomes clear that the
parabolic dimension associated to our problem is
\begin{equation*}
  d:=n+1+\tfrac1m.
\end{equation*}
Moreover, we will use the notations
$$
Q_{\rho,s}(z_o) := B_\rho(x_o) \times (t_o-s,t_o+s)
$$
as well as 
$$
Q_{\rho,+}^{(\theta)}(z_o):= Q_\rho^{(\theta)}(z_o) \cap \{t>0\}
\qquad\mbox{and}\qquad
Q_{\rho,-}^{(\theta)}(z_o):= Q_\rho^{(\theta)}(z_o) \cap \{t<0\}.
$$
For the mean value of a function $f\in L^1(A)$ over a set
$A\subset\R^k$ of finite positive measure we write
$(f)_A:=\bint_Af\,\dx$, 
and for a function $v\in L^1(\Omega_T)$, we abbreviate moreover
\begin{equation*}
  (v)_{z_o;\rho}^{(\theta)}:=\biint_{Q_\rho^{(\theta)}(z_o)}v\,\d x\d t
  \qquad\mbox{and}\qquad
  (v)_{x_o;\rho}(t):=\bint_{B_\rho(x_o)}v(x,t)\,\d x,
\end{equation*}
where $t\in[0,T]$. 
Finally, we define the boundary term as
\begin{align*}
\b[\u^m,\a^m]:= \tfrac{m}{m+1} \big(|a|^{m+1}-|u|^{m+1} \big) -u \cdot \big(\a^m-\u^m \big).
\end{align*}

\subsection{Auxiliary tools}

In order to prove energy estimates we have to use a mollification in time. For this purpose we define for $v \in L^1(\Omega_T,\R^N)$ the mollification
$$
 \llbracket v\rrbracket_h (x,t):=\tfrac 1 h\int_0^t e^{\frac{s-t}{h}}v(x,s)\d s.
$$
For the basic properties of the mollification $\llbracket \cdot \rrbracket_h$ we refer to~\cite[Lemma 2.2]{Kinnunen-Lindqvist} and~\cite[Appendix B]{BDM:pq}.


The next three Lemmas are helpful to estimate certain boundary terms,
and can be found in \cite[Lemmas 2.2, 2.3, 2.7]{Boegelein-Duzaar-Korte-Scheven}. 

\begin{lemma}
\label{lem:estimates}
Let $\alpha>1$. There exists a constant $c=c(\alpha)$ such that for any $u,a\in \R^N$ the following holds true:
\begin{itemize}
\item[(i)] $|u-a|^\alpha \leq c |\u^\alpha-\a^\alpha|$
\item[(ii)] $\frac{1}{c} \babs{\a^{\alpha} - \boldsymbol{b}^{\alpha}} \leq \left[|a|^{\alpha-1} + |b|^{\alpha-1} \right] | a - b | \leq c \babs{\a^{\alpha} - \boldsymbol{b}^{\alpha}}$
\end{itemize}

\end{lemma}

\begin{lemma}
  \label{lem:b}
  Let $m\ge 1$. 
There exists a constant $c=c(m)$ such that for every $u,a \in \R^N$ we have
\begin{itemize}\item[(i)] $\tfrac 1 c \big|\u^{\frac{m+1}{2}}-\a^{\frac{m+1}{2}} \big|^2 \leq \b[\u^m,\a^m] \leq c\big|\u^{\frac{m+1}{2}}-\a^{\frac{m+1}{2}} \big|^2 $
\item[(ii)] $\b[\u^m,\a^m] \leq c |\u^m-\a^m|^{\frac{m+1}{m}}$
\item[(iii)] $\frac 1 c |\u^m-\a^m|^2 \leq \left[|u|^{m-1}+|a|^{m-1} \right]\b[\u^m,\a^m]\leq  c |\u^m-\a^m|^2$
\end{itemize}
\end{lemma}

\begin{lemma}
\label{lem:boundaryave}
There exists a constant $c = c(m)$ such that for any bounded $A \in \R^n$, any $u \in L^{m+1}(A, \R^N)$, and any $a \in \R^N$ there holds
\begin{align*}
\bint_A \b[u,(u)_A]\, \dx \leq c \bint_A \b[u,a]\, \dx.
\end{align*}

\end{lemma}

The proof of the following lemma can be found in \cite[Lemma~3.5]{BDKS:2018}, see also \cite[Lemma
6.2]{Diening-Kaplicky-Schwarzacher} for an earlier version in a
special case.
\begin{lemma}
\label{lem:uavetoa}
Let $p \geq 1$ and $\alpha \geq \frac1p$. Then there exists a constant $c = c(\alpha,p)$ such that for any bounded sets of positive measure satisfying $A \subset B \subset \R^k$, $k \in \N$ and any $u \in L^{\alpha p}(B,\R^N)$ and constant $a \in \R^N$ there holds 
\begin{align*}
\bint_B \babs{\u^{\alpha} - \lbbracket \u \rbbracket_{\bf{A}}^{\alpha}}^p\, \dx \leq \frac{c|B|}{|A|} \bint_B \babs{\u^{\alpha} - \a^{\alpha}}^p\, \dx
\end{align*}
\end{lemma}
Finally, we state a well-known absorption Lemma, that can be found in \cite[Lemma 6.1]{Giusti:book} for instance.
\begin{lemma}
\label{lem:iteration} 
Let $0<\vartheta<1$, $A,C\geq 0$ and $\alpha,\beta>0$. Then, there
exists a constant $c=c(\beta,\vartheta)$ such that there holds: For any $0<r<\rho$ and any nonnegative bounded function $\phi\colon [r,\rho] \to \R_{\geq 0}$ satisfying
$$
\phi(t) \leq \vartheta \phi(s) +A(s^\alpha-t^\alpha)^{-\beta}+C \quad \text{ for all } r\leq t <s \leq \rho,
$$
we have
$$
\phi(r) \leq c \left[A(\rho^\alpha-r^\alpha)^{-\beta}+C \right].
$$
\end{lemma}

\subsection{Variational \texorpdfstring{$p$}{}-capacity}
\label{sec:capacity}

Let $1<p<\infty$ and $D\subset \R^n$ be an open set. The variational $p$-capacity of a compact set $C\subset D$ is defined by
$$
\ca_p(C,D) =\inf_f \int_{D}|Df|^p \d x,
$$
where the infimum is taken over all functions $f \in C^\infty_0(D)$
such that $f\equiv 1$ in $C$. In order to define the variational
$p$-capacity of an open set $U\subset E$, we are taking the supremum
over the capacities of compact sets contained in $U$. The variational $p$-capacity for an arbitrary set $E$ is defined by taking the infimum over the capacities of the open sets containing $E$.
The capacity of a ball is
\begin{align}
\label{cap:ball}
\ca_p(\overline B_\rho(x_o),B_{2\rho}(x_o)) = c \rho^{n-p}.
\end{align}
For more details we refer to  \cite[Ch.\ 4]{Evans-Gariepy} or \cite[Ch.\ 2]{Heinonen-Kilpelainen-Martio}.

At this point we introduce the uniform capacity density condition,
which is essential for proving a boundary version of a
Sobolev-Poincar\'e type inequality, where we note that this condition
is essentially sharp in the context of higher integrability. For the elliptic setting we see \cite{Kilpelainen-Koskela}, whereas the equations of parabolic $p$-Laplacian type were treated in \cite{Kinnunen-Parviainen}.

  We recall the definition of uniform $p$-thickness introduced in
  Definition~\ref{def:p-thick}. The following consequences of this
  property are well-known, see e.g. \cite[Lemma 3.8]{Parviainen}.

\begin{lemma}
\label{lem:est_cap}
Let $\Omega \subset \R^n$ be a bounded open set and assume that $\R^n\setminus \Omega$ is uniformly $p$-thick. Choose $y\in \Omega$ such that $B_{4\rho/3}(y)\setminus \Omega \neq \emptyset$. Then there exists a constant $\tilde \mu =\tilde \mu(n,\mu,\rho_o,p)>0$ such that 
$$
\ca_p\big(\overline B_{2\rho}(y)\setminus \Omega,B_{4\rho}(y)\big) \geq \tilde \mu \ca_p \big(\overline B_{2\rho}(y),B_{4\rho}(y)\big).
$$
\end{lemma}

\begin{lemma}
\label{lem:p-thick}
If a compact set $E$ is uniformly $p$-thick, then $E$ is uniformly $\vartheta$-thick for any $\vartheta \geq p$.
\end{lemma}

The next theorem shows that a uniformly $p$-thick set has a self-improving property, see \cite{Lewis}.

\begin{theorem}
\label{theo:p-thick}
Let $1<p\leq n$. If a set $E$ is uniformly $p$-thick, then there exists a $\gamma=\gamma(n,p,\mu)\in (1,p)$ for which $E$ is uniformly $\gamma$-thick.
\end{theorem}

Before we proceed, let us recall that $u \in W^{1,p}(\Omega)$ is called $p$-quasicontinuous if for each $\varepsilon >0$ there exists an open set $U \subset \Omega\subset B_{R}$ such that $\ca_p(U,B_{2R}) \leq \varepsilon$ and the restriction of $u$ to the set $\Omega\setminus U$ is finite valued and continuous. Note that every function $u\in W^{1,p}(\Omega)$ has a $p$-quasicontinuous representative. A proof of the next lemma can be found in \cite{Hedberg}.

\begin{lemma}
\label{lem:quasicont}
Let $B_\rho(x_o)$ be a ball in $\R^n$ and fix a $q$-quasicontinuous
representative of $u \in W^{1,q}(B_\rho(x_o))$. Denote 
$$
N_{B_{\rho/2}(x_o)}(u):=\{x\in \overline B_{\rho/2}(x_o):u(x)=0\}.
$$
Then there exists a constant $c=c(n,q)>0$ such that
$$
\mint_{B_\rho(x_o)} |u|^q \d x \leq \frac{c}{\ca_q(N_{B_{\rho/2}(x_o)}(u),B_\rho)} \int_{B_\rho(x_o)} |Du|^q \d x.
$$
\end{lemma}

The following Lemma can be found for instance in \cite[Lemma 3.13]{Parviainen}.

\begin{lemma}
\label{lem:quasicont_2}
Let $B_\rho(x_o)$ be a ball in $\R^n$ and suppose that $u \in W^{1,q}(B_\rho(x_o))$ is $q$-quasicontinuous. Denote 
$$
N_{B_{\rho/2}(x_o)}(u):=\{x\in \overline B_{\rho/2}(x_o):u(x)=0\}.
$$
Then, for $\tilde q \in[q,q^\ast]$ with $q^\ast=\frac{nq}{n-q}$ there exists a constant $c=c(n,q)>0$ such that
$$
\left(\mint_{B_\rho(x_o)} |u|^{\tilde q} \d x \right)^{\frac{1}{\tilde q}}\leq  \left( \frac{c}{\ca_q(N_{B_{\rho/2}(x_o)}(u),B_\rho)} \int_{B_\rho(x_o)} |Du|^q \d x\right)^{\frac 1 q}.
$$
\end{lemma}

\section{Energy estimates}
\label{sec:caccioppoli}

In this section, we will prove energy estimates that are required to prove a reverse H\"older inequality.

\subsection{Estimates near the lateral boundary}

We begin with a Caccioppoli type estimate at the lateral boundary. 

\begin{lemma} 
\label{lem:lateralcacciop}
Let $m>1$ and $u$ be a weak solution to \eqref{equation:PME} where the vector field $\A$ satisfies \eqref{assumption:A} and the Cauchy-Dirichlet datum $g$ fulfills \eqref{assumption:g}. Then there exists a constant $c=c(m,\nu,L)$ such that for any cylinder $Q_\rho^{(\theta)}(z_o)\subset \R^{n+1}$ with $0<\rho \leq 1$ and $\theta>0$ and for any $r\in [\rho/2,\rho)$ the following energy estimate
\begin{align*}
\sup_{t\in \Lambda_r^{(\theta)}(t_o) \cap(0,T)}& \int_{B_r(x_o)\cap \Omega}\b[ \u^m(t),\g^m(t)] \d x + \iint_{Q_r^{(\theta)}(z_o)\cap \Omega_T} |D\u^m|^2 \d x\d t \\
& \leq c \iint_{Q_\rho^{(\theta)}(z_o)\cap \Omega_T} \left[ \frac{\big|\u^m-\g^m\big|^2}{(\rho-r)^2}+ \theta^{m-1} \frac{\b[\u^m,\g^m]}{\rho^{\frac{m+1}{m}}-r^{\frac{m+1}{m}}} \right] \d x\d t \\
&\hspace{3mm} + c\iint_{Q_\rho^{(\theta)}(z_o)\cap \Omega_T}\left[ |F|^2 + |D\g^m|^2+|\partial_t \g^m|^{\frac{2m}{2m-1}} \right] \d x\d t
\end{align*}
holds true.
\end{lemma}

\begin{proof}[Proof]
The mollified version of the system \eqref{weak:PME} reads as
\begin{align}
\label{weak:mollified}
\begin{aligned}
\iint_{\Omega_T} &\Big[ \partial_t \llbracket u\rrbracket_h \cdot \varphi + \llbracket \A(x,t,u,D\u^m)\rrbracket_h \cdot D \varphi \Big]\d x \d t \\
& =\iint_{\Omega_T} -\llbracket F\rrbracket_h \cdot D \varphi \d x \d t+ \tfrac 1 h \int_\Omega u(0) \cdot \int_0^T e^{-\frac sh} \varphi \d s \d x
\end{aligned}
\end{align}
for any $\varphi\in L^2(0,T;W^{1,2}_0(\Omega,\R^N))$. 
For $t_1 \in \Lambda_{r}^{(\theta)}(t_o) \cap (0,T)$ approximate the characteristic function of the interval $(0,t_1)$ by 
\begin{align*}
\psi_\varepsilon(t) := \left\{
\begin{array}{cl}
\frac{t-\varepsilon}{\varepsilon},& \text{for } t\in (\varepsilon,2\varepsilon] \\
1, & \text{for } t\in (2\varepsilon,t_1-2\varepsilon]\\
\frac{t_1-\varepsilon-t}{\varepsilon},& \text{for } t \in (t_1-2\varepsilon,t_1-\varepsilon]\\
0,& \text{otherwise } 
\end{array}
\right.
\end{align*}
Furthermore, let $\eta \in C^\infty_0(B_\rho(x_o),[0,1])$ be the standard cut off function with $\eta \equiv 1$ in $B_r(x_o)$ and $|D\eta|\leq \frac{2}{\rho-r}$ and $\zeta \in W^{1,\infty}\left( \Lambda_\rho^{(\theta)}(t_o),[0,1] \right)$ be defined by
\begin{align*}
\zeta(t):= \left\{
\begin{array}{cl}
1,& \text{for } t \geq t_o-\theta^{1-m} r^{\frac{m+1}{m}} \\
\frac{(t-t_o)\theta^{m-1}+\rho^{\frac{m+1}{m}}}{\rho^{\frac{m+1}{m}}-r^{\frac{m+1}{m}}}, & \text{for } t \in (t_o-\theta^{1-m}\rho^{\frac{m+1}{m}},t_o-\theta^{1-m}r^{\frac{m+1}{m}})
\end{array}
\right.
\end{align*}
We choose 
$$
\varphi(x,t)=\eta^2(x)\zeta(t) \psi_\varepsilon(t)\left(\u^m(x,t)-\g^m(x,t)\right)
$$
as testing function in the mollified weak formulation \eqref{weak:mollified}. We start with the parabolic part of the equation and estimate
\begin{align*}
\iint_{\Omega_T}& \partial_t \llbracket u \rrbracket_h \cdot \varphi \d x \d t \\
&=\iint_{ \Omega_T} \eta^2 \zeta \psi_\varepsilon \partial_t \llbracket u \rrbracket_h \cdot \left( \boldsymbol{\llbracket u\rrbracket_h}^m-\g^m \right) \d x \d t\\
& \hspace{3mm}+\iint_{ \Omega_T} \eta^2 \zeta \psi_\varepsilon \partial_t \llbracket u \rrbracket_h \cdot \left(\u^m- \boldsymbol{\llbracket u\rrbracket_h}^m \right) \d x \d t\\
&\geq \iint_{ \Omega_T} \eta^2 \zeta \psi_\varepsilon \partial_t \left( \tfrac{1}{m+1} |\llbracket u\rrbracket_h|^{m+1} -\g^m \cdot \llbracket u \rrbracket_h+ \tfrac{m}{m+1}|g|^{m+1} \right) \d x \d t \\
& \hspace{3mm}+  \iint_{\Omega_T} \eta^2 \zeta \psi_\varepsilon \partial_t \g^m \cdot (\llbracket u \rrbracket_h-g) \d x\d t \\
&=\iint_{ \Omega_T} \left[ \eta^2 \zeta \psi_\varepsilon \partial_t \b\big[\boldsymbol{\llbracket u\rrbracket_h}^m,\g^m\big] +\eta^2 \zeta \psi_\varepsilon \partial_t \g^m \cdot (\llbracket u\rrbracket_h-g) \right] \d x\d t \\
&=\iint_{ \Omega_T} \left[ -\eta^2 (\zeta \partial_t \psi_ \varepsilon+\partial_t\zeta \psi_\varepsilon)  \b\big[\boldsymbol{\llbracket u\rrbracket_h}^m,\g^m\big] +\eta^2 \zeta \psi_\varepsilon \partial_t \g^m \cdot (\llbracket u\rrbracket_h-g) \right] \d x\d t,
\end{align*}
where we also used that $\partial_t \llbracket u \rrbracket_h=\frac 1h (u-\llbracket u \rrbracket_h)$. We are now able to pass to the limit $h\downarrow 0$ in the right-hand side of the previous estimate and obtain
\begin{align*}
\liminf_{h \downarrow 0} \iint_{\Omega_T} &\partial_t \llbracket u \rrbracket_h \cdot \varphi \d x \d t\\
& \geq \iint_{\Omega_T} \left[ -\eta^2(\zeta \partial_t \psi_\varepsilon +\psi_\varepsilon \partial_t \zeta) \b[\u^m,\g^m]+\eta^2 \zeta \psi_\varepsilon \partial_t \g^m \cdot ( u-g) \right] \d x\d t \\
&=: \mathrm{I}_\varepsilon +\mathrm{II}_\varepsilon +\mathrm{III}_\varepsilon.
\end{align*}
Now, we pass to the limit $\varepsilon \downarrow 0$  and obtain for the first term
\begin{align*}
\lim_{\varepsilon \downarrow 0} \mathrm{I}_\varepsilon = \int_{\Omega} \eta^2 \b[\u^m(t_1),\g^m(t_1)] \d x,
\end{align*}
for any $t_1 \in \Lambda_\rho^{(\theta)}(t_o)\cap (0,T)$, where we note that the integral at the time $t=0$ vanishes by assumption \eqref{assumption:initial_datum} in connection with Lemma \ref{lem:b}. The second term can be estimated as follows
\begin{align*}
|\mathrm{II}_\varepsilon| \leq \iint_{Q_\rho^{(\theta)}(z_o)\cap \Omega_T} \theta^{m-1} \frac{\b[\u^m,\g^m]}{\rho^{\frac{m+1}{m}}-r^{\frac{m+1}{m}}} \d x\d t,
\end{align*}
whereas the third term is estimated with the help of Young's inequality and Lemma \ref{lem:estimates}\,(i)
  \begin{align*}
|\mathrm{III}_\varepsilon|& \leq \iint_{Q_\rho^{(\theta)}(z_o)\cap
  \Omega_T} \left[
  (\rho-r)^{\frac2{2m-1}}|\partial_t \g^m|^{\frac{2m}{2m-1}}+ \frac{|u-g|^{2m}}{(\rho-r)^2} \right]  \d x\d t \\
&  \leq c\iint_{Q_\rho^{(\theta)}(z_o)\cap \Omega_T} \left[ |\partial_t \g^m|^{\frac{2m}{2m-1}} + \frac{|\u^m-\g^m|^2}{(\rho-r)^2} \right] \d x\d t,
\end{align*}
since $\rho\le1$.

Next we will treat the diffusion term. After passing to the limit
$h\downarrow 0$ we use the ellipticity and growth condition~\eqref{assumption:A} and Young's inequality and hence we arrive at
\begin{align*}
&\iint_{\Omega_T} \A(x,t,u,D\u^m) \cdot D\varphi \d x\d t \\
&=\iint_{\Omega_T} \A(x,t,u,D\u^m) \cdot \left[ \eta^2 \zeta \psi_\varepsilon (D\u^m-D\g^m) +2\eta \zeta \psi_\varepsilon (\u^m-\g^m)\otimes D\eta  \right] \d x \d t \\
&\geq  \iint_{\Omega_T} \nu \eta^2\zeta\psi_\varepsilon |D\u^m|^2 \d x\d t \\
&\hspace{3mm}- \iint_{\Omega_T}\left[ 2L \eta |D\eta|\zeta \psi_\varepsilon |D\u^m| |\u^m-\g^m|+ L|D\u^m|\eta^2\zeta \psi_\varepsilon |D\g^m| \right] \d x\d t \\
&\geq \tfrac \nu 2 \iint_{\Omega_T} \eta^2 \zeta\psi_\varepsilon |D\u^m|^2 \d x\d t -c \iint_{Q_\rho^{(\theta)}(z_o)\cap\Omega_T} \left[\frac{|\u^m-\g^m|^2}{(\rho-r)^2} +|D\g^m|^2 \right]\d x\d t
\end{align*}
for a constant depending on $m$, $\nu$ and $L$. Let us now consider
the right hand side in \eqref{weak:mollified}. Note that the
second term vanishes in the limit $h \downarrow 0$, since
  \begin{equation*}
    \lim_{h\downarrow0}\mint_0^h\int_\Omega |\power um-\power gm|^{\frac{m+1}{m}}\dx\dt=0,
  \end{equation*}
 which follows from \eqref{assumption:initial_datum}, Lemma   \ref{lem:estimates}(ii) and H\"older's inequality.
In the term containing $F$ we also pass to the limit $h \downarrow 0$ and use Young's inequality afterwards to obtain
\begin{align*}
\iint_{\Omega_T}& F\cdot D\varphi\, \d x\d t \\
&= \iint_{\Omega_T} \left[\eta^2 \zeta \psi_\varepsilon F\cdot (D\u^m - D\g^m) + 2 \eta \zeta \psi_\varepsilon F\cdot (\u^m-\g^m) \otimes D\eta \right] \d x\d t\\
&\leq \tfrac \nu 4\iint_{\Omega_T} \eta^2\zeta \psi_\varepsilon |D\u^m|^2 \d x\d t\\
&\quad+ c \iint_{Q_\rho^{(\theta)}(z_o)\cap\Omega_T} \left[ \frac{|\u^m-\g^m|^2}{(\rho-r)^2}+ |D\g^m|^2+ |F|^2 \right] \d x \d t.
\end{align*}
We combine all these estimates and pass to the limit $\varepsilon \downarrow 0$. This shows
\begin{align*}
\int_{B_r(x_o)\cap \Omega} &\b[\u^m(t_1),\g^m(t_1)] \d x + \int_{(t_o-\theta^{1-m}r^{\frac{m+1}{m}},\,t_1)\cap (0,T)}\int_{B_r(x_o)\cap\Omega} |D\u^m|^2 \d x\d t \\
& \leq c \iint_{Q_{\rho}^{(\theta)}(z_o)\cap \Omega_T} \left[ \frac{|\u^m-\g^m|^2}{(\rho-r)^2} +\theta^{m-1} \frac{\b[\u^m,\g^m]}{\rho^{\frac{m+1}{m}}-r^{\frac{m+1}{m}}} \right] \d x\d t \\
&\hspace{3mm} + c\iint_{Q_{\rho}^{(\theta)}(z_o)\cap \Omega_T} \left[|F|^2 +|D\g^m|^2 +|\partial_t \g^m|^{\frac{2m}{2m-1}} \right] \d x \d t
\end{align*}
for any $t_1 \in \Lambda_r^{(\theta)}(t_o)\cap (0,T)$ and a constant $c=c(m,\nu,L)$. Finally, we take the supremum over all $t_1\in \Lambda_r^{(\theta)}(t_o)\cap (0,T)$ in the first term on the left-hand side and then pass to the limit $t_1 \uparrow t_o+\theta^{1-m}r^{\frac{m+1}{m}}$ in the second term. This proves the lemma.
\end{proof}

\subsection{Estimates near the initial boundary and in the interior}

Up next we prove the corresponding Caccioppoli estimate near the
initial boundary $\Omega \times \{0\}$. For the initial datum we use the abbreviation 
\begin{align*}
g_0(x) := g(x,0) \quad \text{for } x\in \Omega.
\end{align*}
We do not impose an additional regularity assumption on the initial
datum except from $g_0\in L^{m+1}(\Omega,\R^N)$. However, we exploit the
fact that there is an extension $g:\Omega_T\to\R^N$ with
$g(\cdot,0)=g_0$ and $\power gm\in
L^{2+\eps}(0,T;W^{1,2+\eps}(\Omega,\R^N))$ as well as
$\partial_t\power gm\in L^{\frac{m(2+\eps)}{2m-1}}(\Omega_T,\R^N)$. 
At the initial boundary,
we begin with a Caccioppoli type estimate for the extended function $\hat u:\Omega\times(-T,T)\to\R^N$, defined by
\begin{equation}\label{def:hatu}
  \hat u(x,t):=
  \begin{cases}
    u(x,t),&t>0,\\
    g(x,-t),&t\le 0.
  \end{cases}
\end{equation}
We note that the following result also contains the interior case
$Q_\rho^{(\theta)}(z_o)\subset\Omega_T$.

\begin{lemma}\label{lem:initialcacciop}
Let $m>1$ and $u$ be a weak solution to \eqref{equation:PME} where
the vector field $\A$ satisfies \eqref{assumption:A} and the
Cauchy-Dirichlet datum $g$ fulfills \eqref{assumption:g}. Then there
exists a constant $c=c(n,m,\nu,L)$ such that for every cylinder
$Q_\rho^{(\theta)}(z_o)\subset \Omega\times(-T,T)$ with
$z_o\in\Omega\times[0,T)$, $0<\rho \leq 1$ and $\theta>0$,  the following
holds.
For every $r\in[\rho/2,\rho)$ and every $a\in\R^N$, the energy estimate
\begin{align*}
\sup_{t\in \Lambda_r^{(\theta)}(t_o)}&
\int_{B_r(x_o)}\b\big[\power{\hat
  u}{m}(t),\power{a}{m}\big] \d x
  +
  \iint_{Q_{r}^{(\theta)}(z_o)} |D\power{\hat u}{m}|^2 \d x\d t \\
& \leq c \iint_{Q_{\rho}^{(\theta)}(z_o)} \left[
  \frac{\big|\power{\hat u}{m}-\power am\big|^2}{(\rho-r)^2}+ \theta^{m-1}
  \frac{\b\big[\power{\hat u}{m},\power am\big]}{\rho^{\frac{m+1}{m}}-r^{\frac{m+1}{m}}}  \right] \d x\d t \\
&\hspace{3mm} + c \iint_{Q_{\rho,+}^{(\theta)}(z_o)}
\big(|F|^2 +|D\power{g}{m}|^2+|\partial_t\power gm|^{\frac{2m}{2m-1}}\big)\d x \d t
\end{align*}
holds true, where $\hat u$ is defined according to~\eqref{def:hatu}.
\end{lemma}

\begin{proof}[Proof]
We start with arguments similar to the proof of Lemma~\ref{lem:lateralcacciop}. We
consider the mollified version~(\ref{weak:mollified}) of the equation
and use now the test-function
\begin{align*}
\varphi = \eta^2 \zeta \psi_\eps \big( \u^m - \power am \big)
\end{align*}
with $\eta$, $\zeta$, and $\psi_{\eps}$ defined as in 
Lemma~\ref{lem:lateralcacciop} and $\g^m$ replaced by $\power am$. Observe that $\partial_t \power am
= 0$ and $D \power am = 0$. For the parabolic part we obtain
\begin{align*}
\iint_{\Omega_T}& \partial_t \llbracket u \rrbracket_h \cdot \varphi \d x \d t \\
&\geq \iint_{ \Omega_T} \eta^2 \zeta \psi_\varepsilon \partial_t \left( \tfrac{1}{m+1} |\llbracket u\rrbracket_h|^{m+1} -\power am \cdot \llbracket u \rrbracket_h\right) \d x \d t \\
&= - \iint_{ \Omega_T} \eta^2 (\zeta \partial_t \psi_ \varepsilon+\partial_t\zeta \psi_\varepsilon)  \b\big[\boldsymbol{\llbracket u\rrbracket_h}^m,\power am \big] \d x\d t.
\end{align*}
By first passing to the limit $h \downarrow 0$, then $\eps \downarrow
0$ and using the same estimates as in Lemma~\ref{lem:lateralcacciop} we arrive at
\begin{align*}
 \liminf_{\eps \downarrow 0} \bigg( \liminf_{h \downarrow 0} \iint_{\Omega_T} &\partial_t \llbracket u \rrbracket_h \cdot \varphi \d x \d t \bigg) \\
& \geq \int_{\Omega} \eta^2 \b[\u^m(t_1),\power am] \d x - \zeta(0)\int_{\Omega} \eta^2\b[\power{g_0}{m},\power am] \d x   \\
&\hspace{3mm} - c \iint_{Q_{\rho,+}^{(\theta)}(z_o)} \theta^{m-1} \frac{\b[\u^m,\power am]}{\rho^{\frac{m+1}{m}}-r^{\frac{m+1}{m}}} \d x\d t,
\end{align*}
for any $t_1 \in \Lambda_r^{(\theta)}(t_o)\cap (0,T)$. Here we also used the fact that
\begin{align*}
\frac{1}{h} \int_0^h \int_{\Omega} \b[ \u^m(t), \power am] \d x \d t \to \int_{\Omega} \b[ \power{g_0}{m}, \power am] \d x \quad \text{as } h\downarrow 0,
\end{align*}
which follows from Lemma~\ref{lem:estimates}\,(i) and
assumption~(\ref{assumption:initial_datum}).
The diffusion term and the term containing $F$ are treated exactly in
the same way as in Lemma~\ref{lem:lateralcacciop} with $\power am$
instead of $\g^m$ (with obvious simplifications as $D \power am = 0$). The
second integral on the right-hand side of the mollified
equation~(\ref{weak:mollified}) vanishes in the limit $h \downarrow 0$
because of assumption~\eqref{assumption:initial_datum}. By combining these estimates we
obtain the bound
\begin{align}\label{pre-energy}
\sup_{t\in \Lambda_r^{(\theta)}\cap(0,T)}& \int_{B_r}\b\big[
\u^m(t),\power am\big] \d x + \iint_{Q_{r,+}^{(\theta)}} |D\u^m|^2 \d x\d t \\\nonumber
& \leq c \iint_{Q_{\rho,+}^{(\theta)}} \left[
  \frac{\big|\u^m-\power am \big|^2}{(\rho-r)^2}+ \theta^{m-1}
  \frac{\b[\u^m,\power am]}{\rho^{\frac{m+1}{m}}-r^{\frac{m+1}{m}}}  \right] \d x\d t \\\nonumber
&\hspace{3mm} + c \iint_{Q_{\rho,+}^{(\theta)}} |F|^2 \d x \d
t
+
c\,\zeta(0) \int_{B_\rho}
\b[\power{g_0}{m},\power am] \d x.
\end{align}
It remains to estimate the last integral. We start with the
observation that two applications of Lemma~\ref{lem:b}\,(i) imply
$\b[\power{g_0}{m},\power am]\le c\b[\power am,\power{g_0}{m}]$ and
moreover, we have the identity
$$
\partial_t\b[\power am,\power{\hat u}{m}]=\partial_t\power {\hat
  u}m\cdot(\hat u-a)
  \qquad\mbox{on $\Omega\times(-T,0]$}.
$$
This enables us to estimate
\begin{align*}
  &\zeta(0) \int_{B_\rho}
  \b[\power{g_0}{m},\power am] \d x
  \le
  c\,\zeta(0) \int_{B_\rho}
  \b[\power am,\power{g_0}{m}] \d x\\
  &\qquad=
  c\int^0_{t_o-\theta^{1-m}\rho^{\frac{m+1}{m}}}\int_{B_\rho}\partial_t\Big(\zeta(t)\b[\power
  am,\power {\hat u}m]\Big)\d x\d t\\
  &\qquad\le
  c\iint_{Q_{\rho,-}^{(\theta)}}\Big(\big|\partial_t\b[\power
  am,\power {\hat u}m]\big|
  +
  |\partial_t\zeta|\,
  \b[\power am,\power {\hat u}m]\Big)\d x\d t\\
  &\qquad\le
  c\iint_{Q_{\rho,-}^{(\theta)}}\bigg(|\partial_t\power {\hat
    u}m|\,|\hat u-a|
  +
  \theta^{m-1}\frac{\b[\power am,\power
    {\hat u}m]}{\rho^{\frac{m+1}m}-r^{\frac{m+1}m}}\bigg)\d x\d t,
\end{align*}
where we have abbreviated ${Q_{\rho,-}^{(\theta)}}:=
Q_\rho^{(\theta)}\cap\{t<0\}$.
Next, we use Young's inequality, the facts $\rho\le1$ and $\hat
u(t)=g(-t)$ for $t<0$, as well as Lemmas
\ref{lem:estimates} and \ref{lem:b}, with the result 
\begin{align*}
  &\zeta(0) \int_{B_\rho}
  \b[\power{g_0}{m},\power am] \d x\\
  &\qquad\le
  c\iint_{Q_{\rho,-}^{(\theta)}}\bigg(|\partial_t\power {\hat
    u}m|^{\frac{2m}{2m-1}}+\frac{|{\hat u}-a|^{2m}}{(\rho-r)^2}
  +
  \theta^{m-1}\frac{\b[\power am,\power {\hat u}m]}{\rho^{\frac{m+1}m}-r^{\frac{m+1}m}}\bigg)\d x\d t\\
  &\qquad\le
  c\iint_{Q_{\rho,+}^{(\theta)}}|\partial_t\power
  gm|^{\frac{2m}{2m-1}}\d x \dt\\
  &\qquad\qquad+
  c\iint_{Q_{\rho,-}^{(\theta)}}\bigg(\frac{|\power {\hat u}m-\power am|^{2}}{(\rho-r)^2}
  +\theta^{m-1}\frac{\b[\power {\hat u}m,\power am]}{\rho^{\frac{m+1}m}-r^{\frac{m+1}m}}\bigg)\d x\d t.
\end{align*}
Plugging this estimate into \eqref{pre-energy}, we arrive at 
\begin{align}\label{pre-energy-2}
\sup_{t\in \Lambda_r^{(\theta)}\cap(0,T)}& \int_{B_r}\b\big[
\u^m(t),\power am\big] \d x + \iint_{Q_{r,+}^{(\theta)}} |D\u^m|^2 \d x\d t \\\nonumber
& \leq c \iint_{Q_{\rho}^{(\theta)}} \left[
  \frac{\big|\power{\hat u}{m}-\power am \big|^2}{(\rho-r)^2}+ \theta^{m-1}
  \frac{\b[\power{\hat u}{m},\power am]}{\rho^{\frac{m+1}{m}}-r^{\frac{m+1}{m}}}  \right] \d x\d t \\\nonumber
&\hspace{3mm} + c \iint_{Q_{\rho,+}^{(\theta)}} \Big(|F|^2
+|\partial_t\power gm|^{\frac{2m}{2m-1}}\Big)\d x \d t.
\end{align}
It remains to estimate the terms on the left-hand side for negative
times $t\in\Lambda_r^{(\theta)}\cap(-T,0)$. Note that this case only
occurs if $t_o<\theta^{1-m}r^{\frac{m+1}m}$. In this situation, we estimate
\begin{align*}
  &\int_{B_r}\b\big[\power{\hat u}{m}(t),\power am\big] \d x
  \le
  c\int_{B_\rho}\b\big[\power am,\power{\hat u}{m}(t)\big] \d x\\
  &\qquad\le
  c\,\bint_{\Lambda_\rho^{(\theta)}\cap(-T,0)}\int_{B_\rho}
  \bigg[
  \b\big[\power am,\power{\hat u}{m}(\tau)\big]
  +\int_\tau^{t}\partial_t\b\big[\power am,\power{\hat u}{m}(s)\big]\d s
  \bigg]\d x\d\tau\\
  &\qquad\le
  c\,\bint_{\Lambda_\rho^{(\theta)}\cap(-T,0)}\int_{B_\rho}
  \bigg[
  \b\big[\power{\hat u}{m}(\tau),\power am\big]
  +\int_\tau^{t}|\partial_t\power{\hat u}{m}(s)|\,|\hat u(s)-a|\d s
  \bigg]\d x\d\tau.
\end{align*}
For the estimate of the first term, we observe that
 $|\Lambda_\rho^{(\theta)}\cap(-T,0)|\ge
 \theta^{1-m}(\rho^{\frac{m+1}{m}}-r^{\frac{m+1}{m}})$, which is a
consequence of $t_o-\theta^{1-m}r^{\frac{m+1}{m}}<0$. To the remaining
term, we apply Young's inequality and Fubini's theorem, which leads to
the estimate 
\begin{align}\label{est-sup-g}
   &\sup_{t\in\Lambda_r^{(\theta)}\cap(-T,0)}\,\int_{B_r}\b\big[\power{\hat u}{m}(t),\power am\big] \d x  \\\nonumber
   &\qquad\le
   \iint_{Q_{\rho,-}^{(\theta)}}\bigg(\theta^{m-1}\frac{\b\big[\power{\hat u}{m}(\tau),\power
     am\big]}{\rho^{\frac{m+1}{m}}-r^{\frac{m+1}{m}}}
   +
   |\partial_t\power{\hat u}{m}|^{\frac{2m}{2m-1}}
   +
   |\hat u-a|^{2m}\bigg)\d x\d\tau\\\nonumber
   &\qquad\le
   \iint_{Q_{\rho,-}^{(\theta)}}\bigg(\theta^{m-1}\frac{\b\big[\power{\hat u}{m}(\tau),\power
     am\big]}{\rho^{\frac{m+1}{m}}-r^{\frac{m+1}{m}}}
   +
   \frac{|\power{\hat u}{m}-\power am|^{2}}{(\rho-r)^2}\bigg)\d x\d\tau\\\nonumber
   &\qquad\qquad+
   \iint_{Q_{\rho,+}^{(\theta)}}|\partial_t\power{g}{m}|^{\frac{2m}{2m-1}}\d x\d\tau.
\end{align}
In the last step, we used Lemma~\ref{lem:estimates}\,(i), the fact
$\rho\le1$ and the definition of $\hat u$. Moreover, from the
definition of $\hat u$, we immediately obtain the estimate
\begin{equation}
  \label{est-Dg}
  \iint_{Q_{r,-}^{(\theta)}}|D\power{\hat u}{m}|^2\d x\d t
  \le
  \iint_{Q_{r,+}^{(\theta)}}|D\power{g}{m}|^2\d x\d t.
\end{equation}
Combining the estimates~\eqref{est-sup-g} and \eqref{est-Dg} with \eqref{pre-energy-2}, we deduce the claim. 
\end{proof}

Next we prove a lemma that allows us to compare slice-wise values of
the solution between the initial time and any given point of
time. This type of lemma is termed {\it gluing lemma}, and we will use
it later in the proof of a Sobolev-type inequality near the initial boundary.

We start by recalling the gluing lemma
from the interior case, see~\cite[Lemma
3.2]{Boegelein-Duzaar-Korte-Scheven}. By applying this result to the cylinder
$Q_{\rho,+}^{(\theta)}(z_o)$ and using initial
condition~\eqref{assumption:initial_datum} in the case $t=0$, we infer the following lemma.

\begin{lemma} \label{lem:gluing-interior}
Let $m > 1$ and $u$ be a global weak solution to~\eqref{weak:PME} in the
sense of Definition~\ref{global_solution}. We consider a cylinder
$Q_{\rho}^{(\theta)}(z_o) \subset \Omega\times(-T,T)$ with
$z_o\in\Omega\times[0,T)$, $0<\rho \leq 1$ and $\theta > 0$.
Then, there exists $\hat{\rho} \in [\frac{\rho}{2}, \rho]$ such that
for any $t,\tau \in \Lambda_{\rho}^{(\theta)}(t_o)$ with $t,\tau\ge0$ we have
\begin{align*}
| (u)_{x_o;\hat{\rho}}(\tau) - (u)_{x_o;\hat{\rho}}(t)|
&\leq
\frac{c \rho^{\frac{1}{m}}}{\theta^{m-1}}
\biint_{Q_{\rho,+}^{(\theta)}(z_o)}
\Big[ \babs{D\u^m} + |F|\Big] \, \d x \d t,
\end{align*}
with  a constant $c=c(L)$. 
\end{lemma}

We extend this result to a version adapted to the initial boundary. 

\begin{lemma} \label{lem:gluing-new}
Let $m > 1$ and $u$ be a global weak solution to~\eqref{weak:PME} in the
sense of Definition~\ref{global_solution}. We consider a cylinder
$Q_{\rho}^{(\theta)}(z_o) \subset \Omega\times(-T,T)$ with
$z_o\in\Omega\times[0,T)$, $0<\rho\le 1$ and $\theta > 0$.
Then, there exists $\hat{\rho} \in [\frac{\rho}{2}, \rho]$ such that
for any $t,\tau \in \Lambda_{\rho}^{(\theta)}(t_o)$ we have
\begin{align*}
&\frac1\rho\big| \power{(\hat u)_{x_o;\hat{\rho}}}{m}(\tau) - \power{(\hat u)_{x_o;\hat{\rho}}}{m}(t)\big|\\
&\qquad\leq
c\frac{\Theta_{\tau,t}^{m-1}}{\theta^{m-1}}
\biint_{Q_{\rho,+}^{(\theta)}(z_o)}
\Big[ \babs{D\u^m} + |F|\Big] \, \d x \d t\\
&\qquad\qquad+
c\biint_{Q_{\rho,+}^{(\theta)}(z_o)}|D\power gm|\d x\d t
+
\bigg(\frac{c\,\Theta_{\tau,t}^{m-1}}{\theta^{m-1}}\biint_{Q_{\rho,+}^{(\theta)}(z_o)}|\partial_t \power gm|\d x\d t\bigg)^{\frac{m}{2m-1}},
\end{align*}
with a constant $c=c(m,L)$, where we abbreviated
\begin{equation*}
  \Theta_{\tau,t}
  :=
  \bigg(\mint_{B_\rho(x_o)}\frac{|\hat u(\tau)|^m+|\hat u(t)|^m}{\rho}\d x\bigg)^{\frac1m}.
\end{equation*}

\end{lemma}

\begin{proof}
  Throughout the proof, we omit the reference to the center $z_o$ in
  the notation. We choose the radius $\hat{\rho} \in [\frac{\rho}{2}, \rho]$ that is
  provided by Lemma \ref{lem:gluing-interior}. We follow different
  strategies depending on whether the considered times are positive or negative. 
  In the case $t,\tau\ge0$, we combine Lemma~\ref{lem:gluing-interior}
  with Lemma~\ref{lem:estimates}\,(ii) to obtain
  \begin{align}\label{gluing-positive-times}
    &\frac1\rho\big| \power{(\hat u)_{\hat{\rho}}}{m}(\tau) -
    \power{(\hat u)_{\hat{\rho}}}{m}(t)\big|\\\nonumber
    &\qquad\le
    \frac c\rho (\big|(\hat u)_{\hat{\rho}}(\tau)|+|(\hat
    u)_{\hat{\rho}}(t)|\big)^{m-1}\big|(\hat
    u)_{\hat{\rho}}(\tau)-(\hat u)_{\hat{\rho}}(t)\big|\\\nonumber
    &\qquad\le \frac{c\,\Theta_{\tau,t}^{m-1}}{\theta^{m-1}}
    \biint_{Q_{\rho,+}^{(\theta)}}\Big[ \babs{D\u^m} + |F|\Big] \, \d x \d t.
  \end{align}
  Next, we consider the case
  $t,\tau\le0$, in which we can estimate
  \begin{align}\label{gluing-g-1}
    \frac1\rho\mint_{B_\rho}\big|\power{\hat u}{m}(t)-\power{\hat
      u}{m}(\tau)\big|\d x
    &=
    \frac1\rho
    \mint_{B_\rho}\big|\power{g}{m}(-t)-\power{g}{m}(-\tau)\big|\d x\\\nonumber
    &\le
    \frac1\rho\mint_{B_\rho}\bigg|\int_{-t}^{-\tau}\partial_t\power gm(s)\d
    s\bigg|\d x\\\nonumber
    &\le
    \frac{\rho^{\frac{1}{m}}}{\theta^{m-1}}
    \biint_{Q_{\rho,+}^{(\theta)}}|\partial_t\power
      gm|\d x\d s.
  \end{align}
  We apply Lemma \ref{lem:uavetoa}, the definition of $\hat u$ and Poincar\'e's inequality
  in order to estimate
  \begin{align*}
    &\frac1\rho
    \big|(\power{\hat u}{m})_{\hat{\rho}}(t)-\power{(\hat u)_{\hat{\rho}}}{m}(t)\big|\\
    &\qquad\le
    \frac c\rho\mint_{B_\rho}\big|\power{\hat u}{m}(t)-(\power{\hat u}{m})_{\rho}(t)\big|\d x
    =
    \frac c\rho\mint_{B_\rho}\big|\power{g}{m}(-t)-(\power{g}{m})_{\rho}(-t)\big|\d x\\
    &\qquad\le
    \frac c\rho\bint_{\Lambda_{\rho,+}^{(\theta)}}
    \Bigg[\mint_{B_\rho}\big|\power{g}{m}(s)-(\power{g}{m})_{\rho}(s)\big|\d
      x
      +
      \int_{s}^{-t}\mint_{B_\rho}|\partial_t\power gm|(\sigma)\d x\d 
     \sigma\Bigg]\d s\\
    &\qquad\le
    c\,\biint_{Q_{\rho,+}^{(\theta)}}|D\power gm|\d x\d t
    +
    c\frac{\rho^{\frac1m}}{\theta^{m-1}}\biint_{Q_{\rho,+}^{(\theta)}}|\partial_t\power
    gm|\d x\d t.
  \end{align*}
  Obviously, the same estimate holds true for $\tau$ in place of
  $t$. From the two preceding estimates, we deduce 
  \begin{align}\label{gluing-negative-times}
    &\frac1\rho\big| \power{(\hat u)_{\hat{\rho}}}{m}(\tau) -
    \power{(\hat u)_{\hat{\rho}}}{m}(t)\big|\\\nonumber
    &\qquad\le
    \frac1\rho\big| (\power{\hat u}{m})_{\hat{\rho}}(\tau) -
    (\power{\hat u}{m})_{\hat{\rho}}(t)\big|\\\nonumber
    &\qquad\qquad
    +\frac1\rho\big| (\power{\hat u}{m})_{\hat{\rho}}(\tau) -
    \power{(\hat u)_{\hat{\rho}}}{m}(\tau)\big|
    +\frac1\rho\big| (\power{\hat u}{m})_{\hat{\rho}}(t) -
    \power{(\hat u)_{\hat{\rho}}}{m}(t)\big|\\\nonumber
    &\qquad\le
    c\,\biint_{Q_{\rho,+}^{(\theta)}}|D\power gm|\d x\d t
    +
    c\frac{\rho^{\frac1m}}{\theta^{m-1}}\biint_{Q_{\rho,+}^{(\theta)}}|\partial_t\power
    gm|\d x\d t
  \end{align}
  for any $\tau,t\in\Lambda_\rho^{(\theta)}$ with $\tau,t\le0$. 
  It remains to consider the case $t<0<\tau$. In this case we combine
  the estimates \eqref{gluing-positive-times} with $t=0$ and
  \eqref{gluing-negative-times} with $\tau=0$ and deduce
  \begin{align*}
    &\frac1\rho\big| \power{(\hat u)_{\hat{\rho}}}{m}(\tau) -
    \power{(\hat u)_{\hat{\rho}}}{m}(t)\big|\\
    &\qquad\le
    \frac1\rho\big| \power{(\hat u)_{\hat{\rho}}}{m}(\tau) -
    \power{(g_0)_{\hat{\rho}}}{m}\big|
    +
    \frac1\rho\big| \power{(g_0)_{\hat{\rho}}}{m} -
    \power{(\hat u)_{\hat{\rho}}}{m}(t)\big|\\
    &\qquad\le
    \frac{c\,\Theta_{\tau,0}^{m-1}}{\theta^{m-1}}
    \biint_{Q_{\rho,+}^{(\theta)}}
    \Big[ \babs{D\u^m} + |F|\Big] \, \d x \d t\\
    &\qquad\qquad+
    c\,\biint_{Q_{\rho,+}^{(\theta)}}|D\power gm|\d x\d t
    +
    c\frac{\rho^{\frac1m}}{\theta^{m-1}}\biint_{Q_{\rho,+}^{(\theta)}}|\partial_t\power
    gm|\d x\d t.
  \end{align*}
  Using \eqref{gluing-g-1} with $\tau=0$, the term $\Theta_{\tau,0}^{m-1}$ can be bounded as follows.
  \begin{align*}
    \Theta_{\tau,0}^{m-1}
    &\le
    c\Theta_{\tau,t}^{m-1}+
    \bigg(\frac c\rho\mint_{B_\rho}|\power{g_0}{m}-\power{\hat u}{m}(t)|\d x\bigg)^{\frac{m-1}{m}}\\
    &\le
    c\Theta_{\tau,t}^{m-1}+
    c\bigg(\frac{\rho^{\frac{1}{m}}}{\theta^{m-1}}\biint_{Q_{\rho,+}^{(\theta)}}|\partial_t\power{g}{m}|\d x\d t\bigg)^{\frac{m-1}{m}}.
  \end{align*}
  Plugging this into the preceding estimate and applying Young's
  inequality with exponents $\frac{m}{m-1}$ and $m$, we arrive at
    \begin{align*}
    &\frac1\rho\big| \power{(\hat u)_{\hat{\rho}}}{m}(\tau) -
    \power{(\hat u)_{\hat{\rho}}}{m}(t)\big|\\
    &\le
    \frac{c\,\Theta_{\tau,t}^{m-1}}{\theta^{m-1}}
    \biint_{Q_{\rho,+}^{(\theta)}}
    \Big[ \babs{D\u^m} + |F|\Big] \, \d x \d t
    +
    \bigg(\frac{c}{\theta^{m-1}}
    \biint_{Q_{\rho,+}^{(\theta)}}
    \Big[ \babs{D\u^m} + |F|\Big] \, \d x \d t\bigg)^m\\
    &\qquad+
    c\,\biint_{Q_{\rho,+}^{(\theta)}}|D\power gm|\d x\d t
    +
    c\frac{\rho^{\frac1m}}{\theta^{m-1}}\biint_{Q_{\rho,+}^{(\theta)}}|\partial_t\power
    gm|\d x\d t.
  \end{align*}
  In view of Estimates~\eqref{gluing-positive-times} and
  \eqref{gluing-negative-times}, this estimates holds in any case,
  i.e. for arbitrary times $t,\tau\in\Lambda_\rho^{(\theta)}$.
  We multiply the preceding estimate with
  \begin{equation*}
    \bigg(\frac1\rho\big| \power{(\hat u)_{\hat{\rho}}}{m}(\tau) -
    \power{(\hat u)_{\hat{\rho}}}{m}(t)\big|\bigg)^{m-1}
  \end{equation*}
  and use the estimates
  $\frac1\rho| \power{(\hat u)_{\hat{\rho}}}{m}(\tau) -
  \power{(\hat u)_{\hat{\rho}}}{m}(t)|\le\Theta_{\tau,t}^m$ and $\rho\le1$. This
  leads to the bound
  \begin{align*}
    &\bigg(\frac1\rho\big| \power{(\hat u)_{\hat{\rho}}}{m}(\tau) -
    \power{(\hat u)_{\hat{\rho}}}{m}(t)\big|\bigg)^m\\
    &\le
    c\bigg(\frac1\rho\big| \power{(\hat u)_{\hat{\rho}}}{m}(\tau) -
    \power{(\hat u)_{\hat{\rho}}}{m}(t)\big|\bigg)^{m-1}\\
    &\qquad\qquad\qquad\cdot\Bigg[\frac{\Theta_{\tau,t}^{m-1}}{\theta^{m-1}}
    \biint_{Q_{\rho,+}^{(\theta)}}
    \Big[ \babs{D\u^m} + |F|\Big] \, \d x \d t
    +
    \biint_{Q_{\rho,+}^{(\theta)}}|D\power gm|\d x\d t\Bigg]\\
    &\qquad+
    c\bigg(\frac{\Theta_{\tau,t}^{m-1}}{\theta^{m-1}}
    \biint_{Q_{\rho,+}^{(\theta)}}
    \Big[ \babs{D\u^m} + |F|\Big] \, \d x \d t\bigg)^m\\
    &\qquad +
    c\bigg(\frac1\rho\big| \power{(\hat u)_{\hat{\rho}}}{m}(\tau) -
    \power{(\hat u)_{\hat{\rho}}}{m}(t)\big|\bigg)^{\frac{(m-1)^2}{m}}
    \frac{\Theta_{\tau,t}^{m-1}}{\theta^{m-1}}\biint_{Q_{\rho,+}^{(\theta)}}|\partial_t\power
    gm|\d x\d t\\
    &\le \frac12\bigg(\frac1\rho\big| \power{(\hat u)_{\hat{\rho}}}{m}(\tau) -
    \power{(\hat u)_{\hat{\rho}}}{m}(t)\big|\bigg)^{m}\\
    &\qquad+
    c\Bigg[\frac{\Theta_{\tau,t}^{m-1}}{\theta^{m-1}}
    \biint_{Q_{\rho,+}^{(\theta)}}
    \Big[ \babs{D\u^m} + |F|\Big] \, \d x \d t
    +
    \biint_{Q_{\rho,+}^{(\theta)}}|D\power gm|\d x\d t\Bigg]^m\\
    &\qquad+
    c\bigg(\frac{\Theta_{\tau,t}^{m-1}}{\theta^{m-1}}\biint_{Q_{\rho,+}^{(\theta)}}|\partial_t\power
    gm|\d x\d t\bigg)^{\frac{m^2}{2m-1}},
  \end{align*}
  where in the last step we applied Young's inequality, once with
  exponents $\frac{m}{m-1}$ and $m$ and a second time with
  $\frac{m^2}{(m-1)^2}$ and $\frac{m^2}{2m-1}$. We re-absorb the first
  term of the right-hand side into the left and take the $m$th root of
  both sides. This yields the asserted estimate. 
\end{proof}

\section{Sobolev Poincar\'e type inequalities}
\label{sec:poincare}

\subsection{Estimates near the lateral boundary}

The next lemma is an adoption of Lemma 4.2 of \cite{Boegelein-Parviainen}. However, for the sake of completeness we will state a proof.
	
\begin{lemma}
\label{lem:Poincare}
Let $u$ be a global weak solution in the sense of Definition \ref{global_solution} and assume that $\R^n\setminus \Omega$ is uniformly $2$-thick. Moreover, consider a cylinder $Q_{\rho,s}(z_o) \subset \R^{n+1}$ with $B_{\rho/3}(x_o)\setminus \Omega \neq \emptyset$. Then there exists  $\gamma=\gamma(n,\mu) \in (1,2)$ such that for any $\gamma \leq \vartheta \leq 2$ we have
$$
\iint_{Q_{\rho,s}(z_o)\cap \Omega_T} |\u^m-\g^m|^\vartheta \d x \d t \leq c \rho^\vartheta \iint_{Q_{\rho,s}(z_o)\cap \Omega_T} |D(\u^m-\g^m)|^\vartheta \d x\d t,
$$
where $c=c(n,N,\mu,\rho_o,\vartheta)$.
\end{lemma}	

\begin{proof}[Proof]
Let $\gamma=\gamma(n,\mu)\in (1,2)$ be the constant from Theorem \ref{theo:p-thick}. Then, by Lemma \ref{lem:p-thick} we know that $\R^n\setminus \Omega$ is uniformly $\vartheta$-thick for any $\gamma\leq \vartheta \leq 2$. 

We can extend $\u^m-\g^m$ outside of $\Omega_T$ by zero (still denoted in the same way) and define for fixed $t \in (t_o-s,t_o+s)\cap (0,T)$ the set
$$
N_{B_{\rho/2}(x_o)} :=\{x \in \overline B_{\rho/2}(x_o): (\u^m-\g^m)(x,t)=0\} .
$$
Using Lemma \ref{lem:quasicont} shows
\begin{align*}
\int_{B_\rho(x_o)\cap \Omega}& \big|(\u^m-\g^m)(\cdot,t)\big|^\vartheta \d x \\
 &=\int_{B_\rho(x_o)} \big|(\u^m-\g^m)(\cdot,t)\big|^\vartheta \d x \\
& \leq \frac{c\rho^n}{\ca_\vartheta(N_{B_{\rho/2}(x_o)},B_\rho(x_o))} \int_{B_\rho(x_o)} \big|D(\u^m-\g^m)\big|^\vartheta \d x
\end{align*}
for a.e. $t\in (t_o-s,t_o+s)\cap (0,T)$, with a constant $c$ depending only on $n,N,\vartheta$. Since $\R^n\setminus \Omega$ is uniformly $\vartheta$-thick, Lemma \ref{lem:est_cap} and \eqref{cap:ball} imply
$$
\ca_\vartheta (N_{B_{\rho/2}(x_o)},B_\rho(x_o)) \geq \tilde \mu \ca_\vartheta (\overline B_{\rho/2}(x_o),B_\rho(x_o))=c\rho^{n-\vartheta},
$$
where $\tilde \mu =\tilde \mu (n,\mu,\rho_o,\vartheta)$. Combining the previous estimates leads to 
$$
\int_{B_\rho(x_o)\cap\Omega} \big|(\u^m-\g^m)(\cdot,t) \big|^\vartheta \d x \leq c \rho^\vartheta \int_{B_\rho(x_o)\cap \Omega} \big|D(\u^m-\g^m)(\cdot,t)\big|^\vartheta \d x.
$$
Finally, integrating this inequality with respect to $t$ over $(t_o-s,t_o+s)\cap(0,T)$ finishes the proof of the Lemma.
\end{proof}

Next, we are going to prove a different version of a Sobolev-type
inequality.
To this end, we assume that the boundary values are extended to a
function $g\in L^{2+\eps}(0,T;W^{1,2+\eps}(\Omega,\R^N))$,
which is possible since $\Omega$ is an extension domain.
Moreover, we extend the solution and the boundary values across the initial
boundary by letting 
\begin{equation}\label{def:uhat}
\hat u = \left\{
\begin{array}{cl}
u, & t\geq 0 \\
g(-t), & t<0
\end{array}
\right.
\qquad\mbox{and}\qquad
\hat g = \left\{
\begin{array}{cl}
g, & t\geq 0 \\
g(-t), & t<0
\end{array}
\right.
\end{equation}
Note that $\hat \u^m -\hat \g^m=0$ outside of $\Omega_T$. For the
proof of the Sobolev-type inequality, we assume that the cylinders $Q_\rho^{(\theta)}$ satisfy the sub-intrinsic scaling
\begin{align}
\label{sub-intrinsic}
\biint_{Q_{\rho}^{(\theta)}(z_o)} 2\frac{|\hat\u^m-\hat\g^m|^2+|\hat g|^{2m}}{\rho^2} \d x\d t\leq 2^{d+2} \theta^{2m}.
\end{align}
 We observe that \eqref{sub-intrinsic} implies that
\begin{align}
\label{sub-intrinsic_2}
\frac{1}{|Q_{\rho}^{(\theta)}(z_o)|}\iint_{Q_{\rho}^{(\theta)}(z_o)\cap \Omega_T} 2\frac{|\u^m-\g^m|^2+|g|^{2m}}{\rho^2} \d x\d t\leq 2^{d+2} \theta^{2m}
\end{align}
holds true.

\begin{lemma}
\label{lem:Poincare-Sobolev}
Let $u$ be a global weak solution in the sense of Definition
\ref{global_solution} and assume that $\R^n\setminus \Omega$ is
uniformly $2$-thick. Moreover, consider a cylinder
$Q_{\rho}^{(\theta)}(z_o)\subset \R^{n+1}$ with
$\dist(B_{\rho}(x_o), \partial\Omega)=0$ that satisfies the sub-intrinsic scaling \eqref{sub-intrinsic}. Then there exists  $q=q(n,\mu) \in (1,2)$ such that for every $\varepsilon \in (0,1)$
\begin{align*}
\frac{1}{|Q_\rho^{(\theta)}(z_o)|}&\iint_{Q_{\rho}^{(\theta)}(z_o)\cap \Omega_T} \frac{|\u^m-\g^m|^2}{\rho^2} \d x \d t \\
& \leq   \varepsilon \sup_{t\in \Lambda_\rho^{(\theta)} \cap (0,T)} \frac{1}{|B_\rho(x_o)|} \int_{B_\rho(x_o)\cap \Omega} \theta^{m-1} \frac{\b[\u^m(t),\g^m(t)]}{\rho^{\frac{m+1}{m}}} \d x \\
&\quad+  c\varepsilon^{-\frac{4m-2mq}{mq+q}}\left[\frac{1}{|Q_{4\rho}^{(\theta)}(z_o)|} \iint_{Q_{4\rho}^{(\theta)}(z_o)\cap \Omega_T} |D(\u^m-\g^m)|^q \d x\d t\right]^{\frac 2 q} ,
\end{align*}
where $c=c(n,m,N,\mu,\rho_o)$.
\end{lemma}	

\begin{proof}[Proof]
To shorten the notation, we will omit $z_o$ as the reference point for the cylinder. Note that the condition $\dist(B_{\rho}(x_o), \partial\Omega)=0$ implies that $B_{\frac 43 \rho}(x_o)\setminus \Omega \neq \emptyset$.
With a similar argument as in the proof of Lemma \ref{lem:Poincare},
where we use Lemma \ref{lem:quasicont_2} instead of Lemma
\ref{lem:quasicont}, we obtain an exponent
$\vartheta=\vartheta(n,\mu)\in(1,2)$ so that for every
$q\in[\vartheta,2)$ we have 
\begin{align}\label{pre-Sobolev}
  \begin{aligned}
    \frac{1}{| B_{4\rho}|}&\int_{B_{4\rho}\cap\Omega} \big|(\u^m-\g^m)(\cdot,t) \big|^{\frac{nq}{n-q}} \d x \\
    & \leq c \rho^{\frac{nq}{n-q}} \left( \frac{1}{| B_{4\rho}|}
      \int_{B_{4\rho}\cap \Omega} \big|D(\u^m-\g^m)(\cdot,t)\big|^q \d x
    \right)^{\frac{n}{n-q}},
  \end{aligned}
\end{align}
for a constant $c=c(n,N,\mu,\rho_o,q)$. For $\alpha = \alpha(m,q) \in (0,2)$ to be chosen later we estimate with the help of Lemma \ref{lem:b} (iii), H\"older's inequality and the sub-intrinsic scaling \eqref{sub-intrinsic_2} 
\begin{align*}
 &\frac{1}{|Q_\rho^{(\theta)}|}\iint_{Q_\rho^{(\theta)}\cap\Omega_T} \frac{|\u^m-\g^m|^2}{\rho^2} \d x\d t \\
& =\frac{1}{\rho^2|Q_\rho^{(\theta)}|}\iint_{Q_\rho^{(\theta)}\cap \Omega_T} |\u^m-\g^m|^\alpha|\u^m-\g^m|^{2-\alpha} \d x\d t \\ 
&\leq \frac{c}{\rho^2|Q_\rho^{(\theta)}|}  \iint_{Q_\rho^{(\theta)}\cap \Omega_T} \b[\u^m,\g^m]^{\frac\alpha 2} [|u|^{m-1}+|g|^{m-1}]^{\frac\alpha 2} |\u^m-\g^m|^{2-\alpha} \d x\d t \\
&\leq \frac{c}{\rho^2|Q_\rho^{(\theta)}|} \left( \iint_{Q_\rho^{(\theta)}\cap \Omega_T}[|u|^{m-1}+|g|^{m-1}]^{\frac{2m}{m-1}} \d x \d t \right)^{\frac{\alpha(m-1)}{4m}} \\
& \qquad \cdot \left(\iint_{Q_\rho^{(\theta)}\cap \Omega_T} \left[ \b[\u^m,\g^m]^{\frac \alpha 2} |\u^m-\g^m|^{2-\alpha} \right]^{\frac{4m}{4m-\alpha(m-1)}} \d x\d t\right)^{\frac{4m-\alpha(m-1)}{4m}} \\
&\leq \frac{c}{\rho^2|Q_\rho^{(\theta)}|} \big( |Q_\rho^{(\theta)}|  \theta^{2m}\rho^2 \big)^{\frac{\alpha(m-1)}{4m}} \\
& \qquad \cdot \left(\iint_{Q_\rho^{(\theta)}\cap \Omega_T} \left[ \b[\u^m,\g^m]^{\frac \alpha 2} |\u^m-\g^m|^{2-\alpha} \right]^{\frac{4m}{4m-\alpha(m-1)}} \d x\d t\right)^{\frac{4m-\alpha(m-1)}{4m}} \\
&=c \left( \frac{1}{ |Q_\rho^{(\theta)}|}  \iint_{Q_\rho^{(\theta)}\cap \Omega_T} \left[ \left(\theta^{m-1}\frac{\b[\u^m,\g^m]}{\rho^{\frac{m+1}{m}}}\right)^{\frac \alpha 2} \frac{|\u^m-\g^m|^{2-\alpha}}{\rho^{2-\alpha}} \right]^{p^\prime} \d x\d t\right)^{\frac{1}{p^\prime}} ,
\end{align*}
where we used the short-hand notation for exponents 
$$
p:= \frac{4m}{\alpha(m-1)}  \qquad\text{ and } \qquad r:= \frac{nq}{n-q} \frac{1}{p^\prime(2-\alpha)},
$$
and $p^\prime,r^\prime$ are the H\"older conjugates of $p$ and
$r$. Let us note that $p,r>1$ holds true when we choose $\alpha$
suitably, as we do below. Next, we apply H\"older's inequality and then 
estimate \eqref{pre-Sobolev}. In this way, we deduce 
\begin{align*}
\frac{1}{|Q_\rho^{(\theta)}|} &\iint_{Q_\rho^{(\theta)}\cap \Omega_T}
\frac{|\u^m-\g^m|^2}{\rho^2} \d x\d t  \\
&\leq c  \left(\frac{1}{|\Lambda_\rho^{(\theta)}|} \int_{\Lambda_\rho^{(\theta)}\cap(0,T)} \left[\frac{1}{|B_\rho|} \int_{B_\rho \cap \Omega} \left(\theta^{m-1}\frac{\b[\u^m,\g^m]}{\rho^{\frac{m+1}{m}}}\right)^{\frac\alpha2 p^\prime r^\prime} \d x \right]^{\frac{1}{r^\prime}} \right.\\
&\qquad\left.\quad\cdot\left[\frac{1}{|B_{4\rho}|} \int_{B_{4\rho} \cap \Omega} \left(\frac{|\u^m-\g^m|}{\rho}\right)^{\frac{nq}{n-q}}\d x \right]^{\frac 1 r}\d t\right)^{\frac{1}{p^\prime}} \\
& \leq c  \left(\frac{1}{|\Lambda_\rho^{(\theta)}|} \int_{\Lambda_\rho^{(\theta)}\cap(0,T)} \left[\frac{1}{|B_\rho|} \int_{B_\rho \cap \Omega} \left(\theta^{m-1}\frac{\b[\u^m,\g^m]}{\rho^{\frac{m+1}{m}}}\right)^{\frac\alpha2 p^\prime r^\prime} \d x \right]^{\frac{1}{r^\prime}} \right.\\
&\qquad\left.\quad\cdot\left[\frac{1}{|B_{4\rho}|} \int_{B_{4\rho }\cap \Omega} |D(\u^m-\g^m)|^q\d x \right]^{\frac 1 r\frac{n}{n-q}}\d t\right)^{\frac{1}{p^\prime}}
\end{align*}
At this point we are choosing $\alpha$ such that
\begin{align}
\label{exponents}
\frac 1 r\cdot \frac{n}{n-q} =1\quad \Leftrightarrow \quad \frac{p^\prime(2-\alpha)}{q}=1  \quad \Leftrightarrow \quad \alpha = \frac{8m-4mq}{4m-q(m-1)} \in (0,2),
\end{align}
what implies 
\begin{align*}
\frac{1}{|Q_\rho^{(\theta)}|} &\iint_{Q_\rho^{(\theta)}\cap\Omega_T} \frac{|\u^m-\g^m|^2}{\rho^2} \d x\d t  \\
&\leq c \left(\sup_{t\in \Lambda_\rho^{(\theta)}\cap(0,T)}\frac{1}{|B_\rho|} \int_{B_\rho \cap \Omega} \left(\theta^{m-1}\frac{\b[\u^m,\g^m]}{\rho^{\frac{m+1}{m}}}\right)^{\frac\alpha2 p^\prime r^\prime} \d x \right)^{\frac{1}{r^\prime p^\prime}} \\
&\quad \cdot \left(\frac{1}{|Q_{4\rho}^{(\theta)}|} \iint_{Q_{4\rho}^{(\theta)}\cap\Omega_T} |D(\u^m-\g^m)|^q \d x\d t \right)^{\frac{1}{p^\prime}}.
\end{align*}
Next, we observe that we obtain in the limit $q\uparrow 2$ that
$\alpha \to 0$, $p^\prime \to 1$  and $r^\prime\to \frac
n2$. Therefore, we can choose $q$ close to $2$ such that $p>1$, $r>1$,
and $\frac\alpha2 p^\prime r^\prime <1$ hold true and we are able to
use H\"older's and Young's inequalities to deduce 
\begin{align*}
\frac{1}{|Q_\rho^{(\theta)}|} &\iint_{Q_\rho^{(\theta)}\cap\Omega_T} \frac{|\u^m-\g^m|^2}{\rho^2} \d x\d t  \\
&\leq c \left(\sup_{t\in \Lambda_\rho^{(\theta)}\cap(0,T)}\frac{1}{|B_\rho|} \int_{B_\rho \cap \Omega} \theta^{m-1}\frac{\b[\u^m,\g^m]}{\rho^{\frac{m+1}{m}}} \d x \right)^{\frac \alpha 2} \\
&\quad \cdot \left(\frac{1}{|Q_{4\rho}^{(\theta)}|} \iint_{Q_{4\rho}^{(\theta)}\cap\Omega_T} |D(\u^m-\g^m)|^q \d x\d t \right)^{\frac{1}{p^\prime}} \\
&\leq \varepsilon \sup_{t\in \Lambda_\rho^{(\theta)}\cap(0,T)}\frac{1}{|B_\rho|} \int_{B_\rho \cap \Omega} \theta^{m-1}\frac{\b[\u^m,\g^m]}{\rho^{\frac{m+1}{m}}} \d x  \\
&\quad +c\varepsilon^{-\frac{4m-2mq}{mq+q}} \left(\frac{1}{|Q_{4\rho}^{(\theta)}|} \iint_{Q_{4\rho}^{(\theta)}\cap \Omega_T} |D(\u^m-\g^m)|^q \d x\d t \right)^{\frac{2}{(2-\alpha)p^\prime}}.
\end{align*}
Noting that $\frac{2}{(2-\alpha)p^\prime}=\frac 2q$ finishes the proof.
\end{proof}

Next, we are going to prove a Poincar\'e inequality for the boundary
function $g$ that will be very useful in the course of the paper.
This is the point in the proof at which the Sobolev extension property
of the domain is crucial. We recall that the extension property in
particular implies the measure density
condition~\eqref{measure_density}, which in turn implies the lower
bound 
\begin{align}
\label{measure_density_cylinders}
|\Omega_T \cap Q_\rho^{(\theta)}(z_o)| \geq \alpha |Q_\rho^{(\theta)}(z_o)|
\end{align}
for any cylinder $Q_\rho^{(\theta)}(z_o)$ with center
$z_o\in\Omega_T\cup\partial_{\mathrm{par}}\Omega_T$, where 
$\alpha>0$ is a constant depending only on $\Omega$.

\begin{lemma}
\label{lem:poincare_g}
Let $m>1$, $z_o \in \Omega_T \cup \partial_{\mathrm{par}} \Omega_T$, $0<\rho\leq 1$ and $g$ satisfy \eqref{assumption:g}. Then for every
sub-intrinsic cylinder $Q^{(\theta)}_\rho(z_o)\subset\R^n\times(-T,T)$, i.e.\
\eqref{sub-intrinsic} holds true, we have
\begin{align}
\label{poincare_sub_intrinsic}
\biint_{Q_{\rho}^{(\theta)}(z_o)}& \frac{ |\hat\g^m-(\hat\g^m)_{Q_\rho^{(\theta)}(z_o)}|^2}{\rho^2} \d x \d t  \notag\\
&\leq c \biint_{Q_{\rho,+}^{(\theta)}(z_o)} \left[ |D\g^m|^2 + |\partial_t \g^m|^{\frac{2m}{2m-1}}\chi_{\Omega_T} \right] \d x\d t  
\end{align}
for a constant $c$ depending only on $m$, $n$ and $\alpha$.
\end{lemma}

\begin{proof}[Proof.]
For simplicity we omit the center of the cylinder in the notation.  Adding and subtracting the slice-wise mean value integral leads to 
\begin{align*}
 \biint_{Q_{\rho}^{(\theta)}}& |\hat \g^m-(\hat\g^m)_{Q_{\rho}^{(\theta)}}|^2 \d x \d t \\
 &\leq \biint_{Q_{\rho}^{(\theta)}} |\hat\g^m-(\hat\g^m)_{Q_\rho^{(\theta)}\cap \Omega_T}|^2 \d x \d t  \\
&\leq 2 \biint_{Q_{\rho}^{(\theta)}} \left[ |\hat\g^m-(\hat\g^m(t))_{B_\rho \cap \Omega}|^2 +|(\hat\g^m (t))_{B_\rho \cap \Omega}-(\hat\g^m)_{Q_\rho^{(\theta)}\cap \Omega_T}|^2 \right] \d x\d t\\
& =: \mathrm{I} +\mathrm{II}.
\end{align*}
Using Lemma \ref{lem:uavetoa}, the measure density condition \eqref{measure_density} and the Sobolev-inequality shows for the first term
\begin{align*}
\mathrm{I} &\leq  \frac{|B_\rho|}{|\Omega \cap B_\rho|} \biint_{Q_{\rho}^{(\theta)}}|\hat\g^m -(\hat\g^m(t))_{B_\rho}|^2 \d x\d t \\
&\leq c \rho^2 \biint_{Q_{\rho}^{(\theta)}} |D\hat\g^m|^2 \d x\d t \leq c \rho^2 \biint_{Q_{\rho,+}^{(\theta)}} |D\g^m|^2 \d x\d t 
\end{align*}
for a constant $c=c(\alpha,m,n)$.
In order to treat the second term we may assume that $t>0$ because otherwise we use the identity $\hat g(t)=\hat g(-t)$. This allows us to estimate
\begin{align*}
|(\hat\g^m(t))_{B_\rho \cap \Omega}&-(\hat\g^m)_{Q_\rho^{(\theta)}\cap\Omega_T}|^2 \\
&= \left| \biint_{Q_\rho^{(\theta)}\cap \Omega_T}   \g^m(t)-\g^m(\tau) \d x\d \tau \right|^2 \\
  &\leq\left( \biint_{Q_\rho^{(\theta)}\cap \Omega_T} \left| \int_\tau^t \partial_t\g^{m}(s)\d s \right| \d x\d t \right)^2\\
& \leq  \left( 2\biint_{Q_\rho^{(\theta)}\cap \Omega_T} \rho^{\frac{m+1}{m}} \theta^{1-m}|\partial_t \g^m| \d x\d \tau \right)^2.
\end{align*}
This proves the following estimate 
\begin{align*}
 \biint_{Q_{\rho}^{(\theta)}}& |\hat\g^m-(\hat\g^m)_{Q_{\rho}^{(\theta)}}|^2 \d x \d t \\ 
&\leq c\biint_{Q_{\rho,+}^{(\theta)}} \rho^2 |D\g^m|^2 \d x\d t +c \left( \biint_{Q_\rho^{(\theta)}\cap \Omega_T} \rho^{\frac{m+1}{m}} \theta^{1-m}|\partial_t \g^m| \d x\d t\right)^2.
\end{align*}
Using the sub-intrinsic scaling of the cylinders and $m>1$, we obtain
\begin{align*}
(\rho^{\frac{m+1}{m}}\theta^{1-m})^2& \leq c \rho^{2\frac{m+1}{m}} \left( \biint_{Q_{\rho}^{(\theta)}} \frac{|\hat\u^m-\hat\g^m|^2+|\hat g|^{2m}}{\rho^2} \d x\d t \right)^{\frac{1-m}{m}} \\
& \leq \rho^4 \left( \biint_{Q_{\rho}^{(\theta)}} | \hat g|^{2m} \d x\d t \right)^{\frac{1-m}{m}}\\
&\leq c\rho^4 \left( \biint_{Q_{\rho}^{(\theta)}}
  |\hat\g^m-(\hat\g^m)_{Q_\rho^{(\theta)}}|^2\d x\d t
\right)^{\frac{1-m}{m}}.
\end{align*}
This in connection with the last estimate shows
\begin{align*}
\biint_{Q_{\rho}^{(\theta)}}& |\hat\g^m-(\hat\g^m)_{Q_{\rho}^{(\theta)}}|^2 \d x \d t \leq c\biint_{Q_{\rho,+}^{(\theta)}} \rho^2 |D\g^m|^2 \d x\d t\\
&\quad +c \rho^4 \left( \biint_{Q_{\rho}^{(\theta)}} |\hat\g^m-( \hat\g^m)_{Q_{\rho}^{(\theta)}}|^2\d x\d t \right)^{\frac{1-m}{m}}  \left( \biint_{Q_\rho^{(\theta)}\cap \Omega_T} |\partial_t \g^m| \d x\d t\right)^2.
\end{align*}
We multiply this estimate with
$$
  \left[\biint_{Q_{\rho}^{(\theta)}}|\hat\g^m-(\hat\g^m)_{Q_{\rho}^{(\theta)}}|^2 \d x \d t\right]^{\frac{m-1}{m}},
$$
and take both sides to the power $\frac{m}{2m-1}$, which leads to the
bound 
\begin{align*}
\biint_{Q_{\rho}^{(\theta)}}& |\hat\g^m-(\hat\g^m)_{Q_{\rho}^{(\theta)}}|^2 \d x \d t\\
& \leq c\left(\biint_{Q_{\rho,+}^{(\theta)}} \rho^2 |D\g^m|^2 \d x\d t\right)^{\frac{m}{2m-1}} \left(  \biint_{Q_{\rho}^{(\theta)}} |\hat\g^m-(\hat\g^m)_{Q_{\rho}^{(\theta)}}|^2 \d x \d t\right)^{\frac{m-1}{2m-1}}\\
&\quad+\rho^{4\frac{m}{2m-1}}\left( \biint_{Q_\rho^{(\theta)}\cap \Omega_T} |\partial_t \g^m| \d x\d t\right)^{\frac{2m}{2m-1}} \\
&\leq \tfrac 12  \biint_{Q_{\rho}^{(\theta)}} |\hat\g^m-(\hat\g^m)_{Q_{\rho}^{(\theta)}}|^2 \d x \d t+ c\biint_{Q_{\rho,+}^{(\theta)}} \rho^2 |D\g^m|^2 \d x\d t \\
&\quad+\rho^{\frac{4m}{2m-1}}\left( \biint_{Q_\rho^{(\theta)}\cap \Omega_T} |\partial_t \g^m| \d x\d t\right)^{\frac{2m}{2m-1}}.
\end{align*}
Absorbing the first term on the right-hand side, using the measure density condition \eqref{measure_density_cylinders} and applying H\"older's inequality finishes the proof of the lemma.
\end{proof}

\subsection{Estimates near the initial boundary}

Here we prove a Sobolev-type inequality near the initial
boundary. First we prove one auxiliary lemma, since we cannot use the
Sobolev inequality in both space and time directions
simultaneously. That is due to the fact that less regularity is
assumed in the time direction. That is why we first use the gluing
lemma to treat the time direction and then use the Sobolev inequality
slice-wise in space. In this section we assume that the considered
cylinders $Q_{\rho}^{(\theta)}(z_o) \subset \Omega\times(-T,T)$ with
$z_o\in\Omega_T$, $0 < \rho \leq
1$ and $\theta > 0$  satisfy a sub-intrinsic coupling of the form 
\begin{align}
\label{initial-sub-intrinsic-2}
\biint_{Q_{\rho}^{(\theta)}(z_o)} \frac{|\hat u|^{2m}}{\rho^2} \d x\d t\leq 2^{d+2} \theta^{2m},
\end{align}
where $\hat u:\Omega\times(-T,T)\to\R^N$ is defined as in \eqref{def:hatu}.
\begin{lemma}
\label{lem:initial_presobolev}
Let $m>1$ and $u$ be a weak solution to \eqref{equation:PME} where the vector field $\A$ satisfies \eqref{assumption:A}. 
Then there exists a constant $c=c(n,m,\nu,L)$ such that for any
sub-cylinder $Q_\rho^{(\theta)}(z_o)\subset \Omega\times(-T,T)$ with
$z_o\in\Omega\times[0,T)$, $0<\rho \leq 1$ and $\theta>0$, which is
sub-intrinsic in the sense of \eqref{initial-sub-intrinsic-2}, the
inequality
\begin{align*}
&\biint_{Q_{\rho}^{(\theta)}(z_o)} \frac{\big|\power{\hat
    u}{m} - (\power{\hat u}{m})_{z_o;\rho}^{(\theta)}\big|^2}{\rho^2}\, \dx \dt \\
&\leq c \biint_{Q_{\rho}^{(\theta)}(z_o)} \frac{\big|\power{\hat u}{m} -
  (\power{\hat u}{m})_{x_o;\rho}(t) \big|^2}{\rho^2}\, \dx \dt \\
&\quad +  c\left( \biint_{Q_{\rho,+}^{(\theta)}(z_o)}
  \Big[ \babs{D\u^m} + |F| +\babs{D\g^m} \Big] \, \dx \dt \right)^2 \\
&\quad + c \left( \biint_{Q_{\rho,+}^{(\theta)}(z_o)} \babs{\partial_t \g^m} \, \dx \dt \right)^\frac{2m}{2m-1}
\end{align*}
holds true.
\end{lemma}

\begin{proof}
Let $\hat{\rho} \in [\frac{\rho}{2}, \rho]$ be the radius in Lemma~\ref{lem:gluing-new}. For simplicity we omit the reference point $z_o$ in the notation. We start by decomposing
\begin{align*}
 \biint_{Q_{\rho}^{(\theta)}} \frac{\big|\power{\hat u}{m} - (
   \power{\hat u}{m})^{(\theta)}_{\rho}\big|^2}{\rho^2}\, \dx \dt  
  &\le
	3
	\Bigg[\biint_{Q_\rho^{(\theta)}}
	\frac{\big|\power{\hat u}{m} - \power{(\hat u)_{\hat\rho}}{m}(t)\big|^{2}}{\rho^2} \dx\dt\nonumber\\
	&\phantom{\le 3\Bigg[\,}
  	+
	\frac{1}{\rho^2}\mint_{\Lambda_\rho^{(\theta)}}
    	\bigg| \mint_{\Lambda_\rho^{(\theta)}}
	\Big[ \power{(\hat u)_{\hat\rho}}{m}(t)- \power{(\hat u)_{\hat\rho}}{m}(\tau) \Big] \dtau
	\bigg|^{2}\! \dt \nonumber\\
    &\phantom{\le 3\Bigg[\,}
	+
    \frac{1}{\rho^2}\bigg| \mint_{\Lambda_\rho^{(\theta)}}
	\power{(\hat u)_{\hat\rho}}{m}(\tau) \dtau -
	(\power{\hat u}{m})_{\rho}^{(\theta)} \bigg|^{2}\Bigg] \nonumber\\
	&=:
	3 \big[\mbox{I} + \mbox{II} + \mbox{III}\big].
\end{align*}
The first integral we can estimate by using Lemma~\ref{lem:uavetoa}
slice-wise and the fact $\hat{\rho}
\in[\frac\rho2,\rho]$ to obtain
\begin{align*}
  \mathrm{I}
  \le
   c\,\biint_{Q_{\rho}^{(\theta)}} \frac{\babs{\power{\hat
        u}{m} - ( \power{\hat u}{m} )_{\rho} (t)}^2}{\rho^2}\, \dx \dt,
\end{align*}
in which $c = c(n,m)$. For the second term $\mathrm{II}$ we use Gluing
Lemma~\ref{lem:gluing-new}, H\"older's inequality and the
sub-intrinsic scaling~\eqref{initial-sub-intrinsic-2} such that we have
%

\begin{align*}
  \mathrm{II}
  &\leq \frac{c}{ \theta^{2(m-1)} }
  \Bigg(\biint_{Q_\rho} \frac{ |\hat u|^{2m}}{\rho^2}  \, \dx\d t \Bigg)^{\frac{m-1}{m}} \left(
\biint_{Q_{\rho,+}^{(\theta)}}
\Big[ \babs{D\u^m} + |F|\Big] \, \d x \d t \right)^2\\
&\quad+ \frac{c}{ \theta^\frac{2m(m-1)}{2m-1} }\Bigg(\biint_{Q_\rho} \frac{ |\hat u|^{2m}}{\rho^2}  \, \dx\d t \Bigg)^{\frac{m-1}{2m-1}} \left(
\biint_{Q_{\rho,+}^{(\theta)}}
\babs{\partial_t \g^m} \, \d x \d t \right)^\frac{2m}{2m-1}\\
&\quad + c \left( \biint_{Q_{\rho,+}^{(\theta)}}
\babs{D\g^m} \, \d x \d t \right)^2\\
&\leq c \left(
\biint_{Q_{\rho,+}^{(\theta)}}
\Big[ \babs{D\u^m} + |F| + \babs{D\g^m} \Big] \, \d x \d t \right)^2\\
&\quad + c \left(
\biint_{Q_{\rho,+}^{(\theta)}}
\babs{\partial_t \g^m} \, \d x \d t \right)^\frac{2m}{2m-1}.
\end{align*}
For the third term, H\"older's inequality and the estimate for
$\mathrm{I}$ imply
\begin{align*}
  \mathrm{III}\le\mathrm{I}
  \le
   c\,\biint_{Q_{\rho}^{(\theta)}} \frac{\babs{\power{\hat
        u}{m} - ( \power{\hat u}{m} )_{\rho} (t)}^2}{\rho^2}\, \dx \dt,
\end{align*}
which completes the proof.
\end{proof}
Now we are able to prove a suitable Sobolev-type inequality near the initial boundary.
\begin{lemma}
\label{lem:initialSobolev}
Let $m>1$ and $u$ be a global weak solution to \eqref{equation:PME} in
the sense of Definition~\ref{global_solution}, where the vector field
$\A$ satisfies \eqref{assumption:A} and the Cauchy-Dirichlet datum $g$
fulfills \eqref{assumption:g}. Then there exists a constant
$c=c(n,m,\nu,L)$ such that for any sub-cylinder
$Q_\rho^{(\theta)}(z_o)\subset \Omega\times(-T,T)$ with
$z_o\in\Omega\times[0,T)$, $0<\rho \leq 1$ and $\theta>0$ the following
inequalities hold true.
We have the Poincar\'e type estimate
\begin{align*}
 &\biint_{Q_{\rho}^{(\theta)}(z_o)} \frac{ \babs{\power{\hat u}{m} - (\power{\hat u}{m})_\rho^{(\theta)}}^2}{\rho^2}\, \dx \dt \\
&\qquad\leq c\,\biint_{Q_{\rho,+}^{(\theta)}(z_o) }\left[ \babs{D\power{u}{m}}^{2}+ |F|^2 + \babs{D\g^m}^2 + \babs{\partial_t \g^m}^\frac{2m}{2m-1} \right] \, \dx \dt
\end{align*}
as well as the Sobolev-Poincar\'e inequality
\begin{align*}
&\biint_{Q_{\rho}^{(\theta)}(z_o)} \frac{ \babs{\power{\hat u}{m} - (\power{\hat u}{m})_\rho^{(\theta)}}^2}{\rho^2}\, \dx \dt \\
&\leq \eps \sup_{t\in \Lambda_{\rho}^{(\theta)}(t_o)} \bint_{B_{\rho}(x_o)}
\theta^{m-1} \frac{\b[\power{\hat u}{m}(\cdot,t), (\power{\hat u}{m})_{x_o;\rho}^{(\theta)}]}{\rho^{\frac{m+1}{m}}}\, \dx \\
&\hspace{0,5cm}+ \frac{c}{\eps^{\frac{2}{n}}} \Bigg[
\biint_{Q_{\rho}^{(\theta)}(z_o)} \babs{D\power{\hat u}{m}}^{q}\, \dx \dt \Bigg]^{\frac{2}{q}} \\
&\hspace{0,5cm}+ c \biint_{Q_{\rho,+}^{(\theta)}(z_o) } \left[ |F|^2 + \babs{D\g^m}^2 + \babs{\partial_t \g^m}^\frac{2m}{2m-1} \right]  \, \dx \dt
\end{align*}
for any $\eps \in (0,1]$, where $q:=\frac{2n}d<2$.
\end{lemma}
\begin{proof}
  We take Lemma~\ref{lem:initial_presobolev} as a starting point.
  The
first estimate simply follows by an application of Poincar\'e's
inequality on the time slices. For the second claim, we 
proceed as in~\cite[Lemma 4.3]{Boegelein-Duzaar-Korte-Scheven}.
\end{proof}

\section{Reverse H\"older inequalities}
\label{sec:reverse-holder}

In this section we will prove reverse H\"older inequalities. Since the construction of our cylinders does not ensure that we always have intrinsic coupling, we have to distinguish between two cases here. Additionally, we have to treat the lateral boundary in a different way than the initial boundary.

\subsection{The lateral boundary}

The preceding results bring us into position to prove the following reverse H\"older inequality.

\begin{lemma}
\label{lem:reverse_lateral_intrinsic}
Let $m>1$, $z_o \in \Omega_T \cup \partial_{\mathrm{par}} \Omega_T$ and $u$ be a weak solution to \eqref{equation:PME} where the vector field $\A$ satisfies \eqref{assumption:A} and the Cauchy-Dirichlet datum $g$ fulfills \eqref{assumption:g}. Then on any cylinder $Q_\rho^{(\theta)}(z_o)\subset\R^n\times(-T,T)$ with $\dist(B_\rho(x_o),\partial \Omega)=0$ which satisfies the intrinsic coupling
\begin{align}\label{intrinsic_coupling}
\biint_{Q_{2\rho}^{(\theta)}(z_o)}& 2\frac{|\hat\u^m-\hat\g^m|^2+|\hat g|^{2m}}{(2\rho)^2} \d x\d t\leq \theta^{2m}\notag \\
&\leq K\biint_{Q_{\rho}^{(\theta)}(z_o)}2\frac{ |\hat\u^m-\hat\g^m|^2+|\hat g|^{2m}}{\rho^2} \d x\d t 
\end{align}
for some $0< \rho\leq 1$, $\theta\geq 0$ and
  $K\ge1$, we have the following reverse H\"older inequality
\begin{align*}
&\frac{1}{|Q_\rho^{(\theta)}(z_o)|}\iint_{Q_\rho^{(\theta)}(z_o)\cap \Omega_T} |D\u^m|^2 \d x\d t \\
&\leq c\left(\frac{1}{|Q_{8\rho}^{(\theta)}(z_o)|}\iint_{Q_{8\rho}^{(\theta)}(z_o)\cap \Omega_T} |D\u^m|^q  \d x \d t \right)^{\frac2q}\\
&\quad
+c\biint_{Q_{8\rho,+}^{(\theta)}(z_o)} \left[  \left( |F|^2 +|\partial_t
  \g^m|^{\frac{2m}{2m-1}} \right) \chi_{\Omega_T} +|D\g^m|^2 \right] \d x\d
t,
\end{align*}
for a constant $c  = c(m,n,N,\alpha,\mu, \rho_o,\nu,L,K)$ and some $q = q(n,\mu)\in (1,2)$.
\end{lemma}

\begin{proof}[Proof]
Let $0<\rho \leq r<s\leq 2\rho$. To shorten the notation, we will again omit the reference point $z_o$. Utilizing Lemma \ref{lem:lateralcacciop} shows
\begin{align}\label{caccioppoli-boundary}
&\sup_{t\in \Lambda_r^{(\theta)}\cap(0,T)} \frac{1}{|B_r|} \int_{B_r\cap \Omega} \theta^{m-1}\frac{\b[ \u^m(t),\g^m(t)]}{r^{\frac{m+1}{m}}} \d x \\\notag
&\quad + \frac{1}{|Q_r^{(\theta)}|}\iint_{Q_r^{(\theta)}\cap \Omega_T} |D\u^m|^2 \d x\d t \\\notag
& \quad\leq  \frac{c}{|Q_s^{(\theta)}|}\iint_{Q_s^{(\theta)}\cap \Omega_T}  \frac{\big|\u^m-\g^m\big|^2}{(s-r)^2} \d x \d t\\\notag
&\qquad+  \frac{c}{|Q_s^{(\theta)}|}\iint_{Q_s^{(\theta)}\cap \Omega_T}\theta^{m-1} \frac{\b[\u^m,\g^m]}{s^{\frac{m+1}{m}}-r^{\frac{m+1}{m}}}  \d x\d t \\\notag
&\qquad +  c\biint_{Q_{s,+}^{(\theta)}} \left[ \left( |F|^2 +|\partial_t \g^m|^{\frac{2m}{2m-1}} \right) \chi_{\Omega_T}+ |D\g^m|^2 \right] \d x\d t \\\notag
&=: \mathrm{I} +\mathrm{II}+ \mathrm{III}
\end{align}
with the obvious meaning of $\mathrm{I}-\mathrm{III}$. Using the abbreviation
$$
\mathcal{R}_{r,s}:= \frac{s^{\frac{m+1}{2m}}}{s^{\frac{m+1}{2m}}-r^{\frac{m+1}{2m}}}
$$
as well as the estimate $(s^{\frac{m+1}{2m}}-r^{\frac{m+1}{2m}}) \leq (s-r)^{\frac{m+1}{2m}}$ implies
\begin{align*}
\mathrm{I} \leq  \frac{c\mathcal{R}_{r,s}^{\frac{4m}{m+1}}}{|Q_s^{(\theta)}|}\iint_{Q_s^{(\theta)}\cap \Omega_T}  \frac{\big|\u^m-\g^m\big|^2}{s^2} \d x \d t.
\end{align*} 
For the second term we use the intrinsic coupling \eqref{intrinsic_coupling}, noting that $(\hat\u^m -\hat \g^m)\chi_{Q_\rho^{(\theta)}}= (\u^m-\g^m)\chi_{Q_\rho^{(\theta)} \cap \Omega_T}$, and Lemma \ref{lem:b} to obtain
\begin{align*}
\mathrm{II} &\leq \frac{c\mathcal{R}_{r,s}^2}{|Q_s^{(\theta)}|}\iint_{Q_s^{(\theta)}\cap \Omega_T}\theta^{m-1} \frac{\b[\u^m,\g^m]}{s^{\frac{m+1}{m}}}  \d x\d t \\
&\leq\frac{c\mathcal{R}_{r,s}^2}{|Q_s^{(\theta)}|}\iint_{Q_s^{(\theta)}\cap \Omega_T}  \frac{\big|\u^m-\g^m\big|^2}{s^2} \d x \d t \\
&\quad+  \left(\biint_{Q_{s}^{(\theta)}}\frac{|\hat  g|^{2m}}{s^2} \d x\d t \right)^{\frac{m-1}{2m}}\frac{c\mathcal{R}_{r,s}^2}{|Q_s^{(\theta)}|}\iint_{Q_s^{(\theta)}\cap \Omega_T} \frac{\b[\u^m,\g^m]}{s^{\frac{m+1}{m}}}  \d x\d t\\
&=: \mathrm{II}_1+ \mathrm{II}_2.
\end{align*}
In order to estimate the second term on the right-hand side, we first use the  Poincar\'e inequality \eqref{poincare_sub_intrinsic} to obtain
\begin{align*}
\biint_{Q_{s}^{(\theta)}}& \frac{ |\hat g|^{2m}}{s^2}\d x\d t \\
&\leq \frac{2}{s^2}\biint_{Q_{s}^{(\theta)}} \left[ |\hat\g^m-(\hat\g^m)_{Q_{s}^{(\theta)}}|^2+|(\hat\g^m)_{Q_{s}^{(\theta)}}|^2 \right] \d x\d t \\
&\leq c \biint_{Q_{s,+}^{(\theta)}} \left[ |D\g^m|^2 +|\partial_t \g^m|^{\frac{2m}{2m-1}} \chi_{\Omega_T} \right] \d x\d t +c \frac{|(\hat\g^m)_{Q_{s}^{(\theta)}}|^2}{s^2} \\
 &=: \mathcal{G}_s +c \frac{|(\hat\g^m)_{Q_{s}^{(\theta)}}|^2}{s^2},
\end{align*}
where we abbreviated
\begin{equation*}
  \mathcal G_s:=\biint_{Q_{s,+}^{(\theta)}} \left[ |D\g^m|^2 +\Big(|\partial_t \g^m|^{\frac{2m}{2m-1}}+|F|^2\Big) \chi_{\Omega_T} \right] \d x\d t.
\end{equation*}
Using Lemma \ref{lem:b} and Young's inequality, we further estimate
\begin{align*}
\mathrm{II}_2 &\leq c  \mathcal{G}_s^{\frac{m-1}{2m}}\frac{c\mathcal{R}_{r,s}^2}{|Q_s^{(\theta)}|}\iint_{Q_s^{(\theta)}\cap \Omega_T} \frac{\b[\u^m,\g^m]}{s^{\frac{m+1}{m}}}  \d x\d t \\
&\quad +\frac{c\mathcal{R}_{r,s}^2}{|Q_s^{(\theta)}|}\iint_{Q_s^{(\theta)}\cap \Omega_T} |(\hat\g^m)_{Q_{s}^{(\theta)}}|^{\frac{m-1}{m}} \frac{\b[\u^m,(\hat\g^m)_{Q_{s}^{(\theta)}}]+\b[(\hat\g^m)_{Q_{s}^{(\theta)}},\g^m]}{s^2} \d x\d t \\
&\leq c \mathcal{R}_{r,s}^2 \left[ \mathcal G_s +\left(\biint_{Q_s^{(\theta)}\cap\Omega_T} \frac{|\u^m-\g^m|^{\frac{m+1}{m}}}{s^{\frac{m+1}{m}}}  \d x\d t \right)^{\frac{2m}{m+1}} \right] \\
&\quad+ \frac{c\mathcal{R}_{r,s}^2}{|Q_s^{(\theta)}|}\iint_{Q_s^{(\theta)}\cap \Omega_T} \frac{|\u^m-\g^m|^2+|\g^m-(\hat\g^m)_{Q_{s}^{(\theta)}}|^2}{s^2} \d x\d t \\
&\leq c \mathcal{R}_{r,s}^2 \left[ \mathcal G_s + \frac{1}{|Q_s^{(\theta)}|}\iint_{Q_s^{(\theta)}\cap \Omega_T} \frac{|\u^m-\g^m|^2}{s^2} \d x\d t \right.\\
&\hspace{35mm} \left.+\biint_{Q_s^{(\theta)}} \frac{|\hat \g^m-(\hat \g^m)_{Q_s^{(\theta)}}|^2}{s^2} \d x\d t \right]\\
&\leq c \mathcal{R}_{r,s}^2 \left[ \mathcal G_s + \frac{1}{|Q_s^{(\theta)}|}\iint_{Q_s^{(\theta)}\cap \Omega_T} \frac{|\u^m-\g^m|^2}{s^2} \d x\d t \right].
\end{align*}
Inserting the estimates for the terms $\mathrm{I}$ and $\mathrm{II}$, and using Lemma \ref{lem:Poincare-Sobolev} shows for every $\varepsilon\in(0,1)$ that
\begin{align*}
&\sup_{t\in \Lambda_r^{(\theta)} \cap(0,T)} \frac{1}{|B_r|} \int_{B_r\cap \Omega} \theta^{m-1}\frac{\b[ \u^m(t),\g^m(t)]}{r^{\frac{m+1}{m}}} \d x \\
&\quad + \frac{1}{|Q_r^{(\theta)}|}\iint_{Q_r^{(\theta)}\cap \Omega_T} |D\u^m|^2 \d x\d t \\
&\quad \leq c\mathcal{R}_{r,s}^{\frac{4m}{m+1}} \left[ \varepsilon \sup_{t\in \Lambda_s^{(\theta)} \cap(0,T)} \frac{1}{|B_s|} \int_{B_s\cap \Omega} \theta^{m-1}\frac{\b[ \u^m(t),\g^m(t)]}{s^{\frac{m+1}{m}}} \d x \right.\\
&\quad \left.+\varepsilon^{-\frac{4m-2mq}{mq+q}}\left(\frac{1}{|Q_{4s}^{(\theta)}|}\iint_{Q_{4s}^{(\theta)}\cap \Omega_T} |D(\u^m-\g^m)|^q  \d x \d t \right)^{\frac2q}+\mathcal G_s\right]
\end{align*}
holds true. Choosing $\varepsilon= \frac{1}{2c\mathcal{R}_{r,s}^{\frac{4m}{m+1}}}$ yields
\begin{align*}
&\sup_{t\in \Lambda_r^{(\theta)} \cap(0,T)} \frac{1}{|B_r|} \int_{B_r\cap \Omega} \theta^{m-1}\frac{\b[ \u^m(t),\g^m(t)]}{r^{\frac{m+1}{m}}} \d x \\
&\quad + \frac{1}{|Q_r^{(\theta)}|}\iint_{Q_r^{(\theta)}\cap \Omega_T} |D\u^m|^2 \d x\d t \\
&\quad \leq \frac 12  \sup_{t\in \Lambda_s^{(\theta)} \cap(0,T)} \frac{1}{|B_s|} \int_{B_s\cap \Omega} \theta^{m-1}\frac{\b[ \u^m(t),\g^m(t)]}{s^{\frac{m+1}{m}}} \d x \\
&\quad +c \mathcal{R}_{r,s}^{\frac{4m}{m+1}\big(\frac{4m-2mq}{(m+1)q}+1\big)}\left(\frac{1}{|Q_{8\rho}^{(\theta)}|}\iint_{Q_{8\rho}^{(\theta)}\cap \Omega_T} |D(\u^m-\g^m)|^q  \d x \d t \right)^{\frac2q}+c\mathcal{R}_{r,s}^{\frac{4m}{m+1}}\mathcal G_{2\rho}.
\end{align*}
We are now in position to apply Lemma \ref{lem:iteration} and obtain
\begin{align*}
&\sup_{t\in \Lambda_\rho^{(\theta)}\cap(0,T)} \frac{1}{|B_\rho|} \int_{B_\rho\cap \Omega} \theta^{m-1}\frac{\b[ \u^m(t),\g^m(t)]}{\rho^{\frac{m+1}{m}}} \d x \\
&\quad + \frac{1}{|Q_\rho^{(\theta)}|}\iint_{Q_\rho^{(\theta)}\cap \Omega_T} |D\u^m|^2 \d x\d t \\
&\leq c\left(\frac{1}{|Q_{8\rho}^{(\theta)}|}\iint_{Q_{8\rho}^{(\theta)}\cap \Omega_T} |D(\u^m-\g^m)|^q  \d x \d t \right)^{\frac2q} +\mathcal G_{2\rho}.
\end{align*}
This finishes the proof of the Lemma.
\end{proof}

\begin{lemma}
\label{lem:reverse_lateral_sub-intrinsic}
Let $m>1$, $z_o \in \Omega_T \cup \partial_{\mathrm{par}}\Omega_T$ and $u$ be a weak solution to \eqref{equation:PME} where the vector field $\A$ satisfies \eqref{assumption:A} and the Cauchy-Dirichlet datum $g$ fulfills \eqref{assumption:g}. Then on any cylinder $Q_\rho^{(\theta)}(z_o)\subset\R^n\times(-T,T)$ with  $\dist(B_\rho(x_o),\partial \Omega)=0$, which satisfies the intrinsic coupling
\begin{align}
\label{sub-intrinsic_coupling}
\biint_{Q_{2\rho}^{(\theta)}(z_o)}&2\frac{|\hat\u^m-\hat \g^m|^2+|\hat g|^{2m}}{(2\rho)^2} \d x\d t\leq \theta^{2m} \notag \\
&\leq  K\biint_{Q_{\rho,+}^{(\theta)}(z_o)} \left[ \left(|D\u^m|^2+|F|^2+|\partial_t \g^m|^{\frac{2m}{2m-1}} \right) \chi_{\Omega_T}+|D\g^m|^2 \right] \d x\d t 
\end{align}
for some $0< \rho\leq 1$, $\theta>0$ and $K \geq 1$ we have the following reverse H\"older inequality
\begin{align*}
&\frac{1}{|Q_\rho^{(\theta)}(z_o)|}\iint_{Q_\rho^{(\theta)}(z_o)\cap \Omega_T} |D\u^m|^2 \d x\d t \\
&\leq c\left(\frac{1}{|Q_{8\rho}^{(\theta)}(z_o)|}\iint_{Q_{8\rho}^{(\theta)}(z_o)\cap \Omega_T} |D\u^m|^q  \d x \d t \right)^{\frac2q}\\
&\quad +\frac{c}{|Q_{8\rho}^{(\theta)}(z_o)|}\iint_{Q_{8\rho,+}^{(\theta)}(z_o)} \left[ |D\g^m|^2+\left(|\partial_t\g^m|^{\frac{2m}{2m-1}}+|F|^2 \right) \chi_{\Omega_T} \right] \d x\d t.
\end{align*}
for a constant $c  = c(m,n,N,\alpha,\mu, \rho_o,\nu,L, K)$ and some $q = q(n,\mu)\in (1,2)$.
\end{lemma}

\begin{proof}[Proof]
We consider again radii $r,s>0$ with $\rho \leq r<s \leq 2\rho$ and
take the Caccioppoli inequality from Lemma \ref{lem:lateralcacciop} as
starting point. We use the same short-hand notation as in the proof of
Lemma \ref{lem:reverse_lateral_intrinsic}. The first term in~\eqref{caccioppoli-boundary} can be estimated in the same way as before, whereas the second term will be treated in a different way. By using Young's inequality and Lemma~\ref{lem:b} (ii) we obtain
\begin{align*}
\mathrm{II} &\leq\frac{c\mathcal{R}_{r,s}^2}{|Q_s^{(\theta)}(z_o)|}\iint_{Q_s^{(\theta)}(z_o)\cap \Omega_T}\theta^{m-1} \frac{\b[\u^m,\g^m]}{s^{\frac{m+1}{m}}}  \d x\d t \\
&\leq \delta \theta^{2m} +\frac{c_\delta\mathcal{R}_{r,s}^{\frac{4m}{m+1}}}{|Q_s^{(\theta)}(z_o)|}\iint_{Q_s^{(\theta)}(z_o)\cap \Omega_T}\frac{\b[\u^m,\g^m]^{\frac{2m}{m+1}}}{s^2}  \d x\d t \\
&\leq \delta \theta^{2m} +\frac{c_\delta\mathcal{R}_{r,s}^{\frac{4m}{m+1}}}{|Q_s^{(\theta)}(z_o)|}\iint_{Q_s^{(\theta)}(z_o)\cap \Omega_T}\frac{|\u^m-\g^m|^{2}}{s^2}  \d x\d t.
\end{align*}
Using the intrinsic coupling \eqref{sub-intrinsic_coupling} allows us to absorb the term involving $D\u^m$ and moreover, exploiting Lemma \ref{lem:Poincare-Sobolev} leads to 
\begin{align*}
&\sup_{t\in \Lambda_r^{(\theta)}(t_o) \cap(0,T)} \frac{1}{|B_r(x_o)|} \int_{B_r(x_o)\cap \Omega} \theta^{m-1}\frac{\b[ \u^m(t),\g^m(t)]}{r^{\frac{m+1}{m}}} \d x \\
&\quad + \frac{1}{|Q_r^{(\theta)}(z_o)|}\iint_{Q_r^{(\theta)}(z_o)\cap \Omega_T} |D\u^m|^2 \d x\d t \\
&\leq \frac{c\mathcal{R}_{r,s}^{\frac{4m}{m+1}}}{|Q_s^{(\theta)}(z_o)|}\iint_{Q_{s,+}^{(\theta)}(z_o)} \Bigg[ \frac{|\u^m-\g^m|^{2}}{s^2} +|D\g^m|^2 +\left(|F|^2  + |\partial_t \g^m|^{\frac{2m}{2m-1}}\right)\chi_{\Omega_T} \Bigg] \d x\d t\\
&\leq  c\mathcal{R}_{r,s}^{\frac{4m}{m+1}} \left[ \varepsilon \sup_{t\in \Lambda_s^{(\theta)}(t_o) \cap(0,T)} \frac{1}{|B_s(x_o)|} \int_{B_s(x_o)\cap \Omega} \theta^{m-1}\frac{\b[ \u^m(t),\g^m(t)]}{s^{\frac{m+1}{m}}} \d x \right.\\
&\qquad +\varepsilon^{-\frac{4m-2mq}{mq+q}}\left(\frac{1}{|Q_{4s}^{(\theta)}(z_o)|}\iint_{Q_{4s}^{(\theta)}(z_o)\cap \Omega_T} |D(\u^m-\g^m)|^q  \d x \d t \right)^{\frac2q} \\
&\left. \qquad+\frac{1}{|Q_s^{(\theta)}(z_o)|}\iint_{Q_{s,+}^{(\theta)}(z_o)} \Big[ |D\g^m|^2+ \left(|F|^2  + |\partial_t \g^m|^{\frac{2m}{2m-1}} \right) \chi_{\Omega_T} \Big] \d x\d t\right].
\end{align*}
Proceeding as in the proof of Lemma \ref{lem:reverse_lateral_intrinsic} completes the proof.
\end{proof}

\subsection{The initial boundary}

Our next goal is the proof of reverse H\"older inequalities at the
initial boundary. Again, we have to distinguish between two cases. 

\begin{lemma}
\label{lem:reverse_initial_intrinsic}
Let $m>1$ and $u$ be a global weak solution to \eqref{equation:PME} in
the sense of Definition~\ref{global_solution}, where the vector field
$\A$ satisfies \eqref{assumption:A} and the Cauchy-Dirichlet datum $g$
fulfills \eqref{assumption:g}. Then on any cylinder
$Q_{2\rho}^{(\theta)}(z_o) \subset \Omega\times(-T,T)$ with 
$z_o\in\Omega\times[0,T)$, which satisfies the intrinsic coupling
\begin{align}\label{initial_intrinsic_coupling}
  \biint_{Q_{2\rho}^{(\theta)}(z_o)}
  \frac{|\hat u|^{2m}}{(2\rho)^2} \d x\d t\leq \theta^{2m}  
  \leq
  \biint_{Q_{\rho}^{(\theta)}(z_o)}\frac{|\hat u|^{2m}}{\rho^2} \d x\d t 
\end{align}
for some $0< \rho\leq 1$ and $\theta\geq 1$ we have the following reverse H\"older inequality
\begin{align*}
&\biint_{Q_\rho^{(\theta)}(z_o)} |D\power{\hat u}{m}|^2 \d x\d t \\
&\leq c\left(\biint_{Q_{2\rho}^{(\theta)}(z_o)} |D\power{\hat u}{m}|^{q}  \d x \d t \right)^{\frac2q}\\
&\qquad+
c\,\biint_{Q_{2\rho,+}^{(\theta)}(z_o)} \Big[
|F|^2+|\partial_t\power gm|^{\frac{2m}{2m-1}}+|D\power{g}{m}|^2 \Big] \d x\d t
\end{align*}
for a constant $c = c(n,m,\nu,L)$ and for $q:= \frac{2n}{d} < 2$.
\end{lemma}
\begin{proof}
We omit the reference point $z_o$ in notation, and consider radii $\rho \leq r < s \leq 2\rho$. From the Caccioppoli estimate in Lemma~\ref{lem:initialcacciop} we obtain
\begin{align}\label{caccio-initial}
\sup_{t\in \Lambda_r^{(\theta)}}& \bint_{B_r} \theta^{m-1}
\frac{\b\big[\power{\hat u}{m}(t),(\power{\hat
    u}{m})_r^{(\theta)}\big]}{r^{\frac{m+1}{m}}} \d x +
\biint_{Q_r^{(\theta)} } |D\power{\hat u}{m}|^2 \d x\d t \\\nonumber
& \leq c\biint_{Q_s^{(\theta)}} \left[ \frac{\big|\power{\hat u}{m}-(\power{\hat u}{m})^{(\theta)}_{r} \big|^2}{(s-r)^2}+ \theta^{m-1} \frac{\b\big[\power{\hat u}{m},(\power{\hat u}{m})^{(\theta)}_{r}]}{s^{\frac{m+1}{m}}-r^{\frac{m+1}{m}}}  \right] \d x\d t \\\nonumber
&\hspace{3mm} +c\biint_{Q_{s,+}^{(\theta)}} \Big[
|F|^2+|D\power{g}{m}|^2+|\partial_t\power gm|^{\frac{2m}{2m-1}} \Big] \d x
\d t \\\nonumber
&=: \mathrm{I} + \mathrm{II} + \mathrm{III} .
\end{align}
In the same way as in Lemma~\ref{lem:reverse_lateral_intrinsic} we can estimate the first term as 
\begin{align*}
  \mathrm{I}
  \leq
    c\mathcal{R}_{r,s}^{\frac{4m}{m+1}}
    \biint_{Q_s^{(\theta)}} \frac{\babs{\power{\hat u}{m} - (\power{\hat
    u}{m})^{(\theta)}_{r}}^2}{s^2} \, \dx\dt 
  \leq c
    \mathcal{R}_{r,s}^{\frac{4m}{m+1}}
    \biint_{Q_s^{(\theta)}} \frac{\babs{\power{\hat u}{m} - (\power{\hat u}{m})^{(\theta)}_{s}}^2}{s^2} \, \dx\dt,
\end{align*}
where in the last step we applied Lemma~\ref{lem:uavetoa}.
For the second term we use the intrinsic coupling~\eqref{initial_intrinsic_coupling} and end up in having
\begin{align*}
  \mathrm{II}
  &\leq
  c \mathcal{R}_{r,s}^2
  \biint_{Q_s^{(\theta)}} \theta^{m-1}
  \frac{\b[\power{\hat u}{m},(\power{\hat u}{m})^{(\theta)}_{r}]}{s^{\frac{m+1}{m}}}
      \, \dx \dt \\
  &\leq
  c \mathcal{R}_{r,s}^2
  \Bigg(\biint_{Q_\rho^{(\theta)}}\frac{|\hat u|^{2m}}{\rho^2}\d x\d t\Bigg)^{\frac{m-1}{2m}}
  \biint_{Q_s^{(\theta)}} 
  \frac{\b[\power{\hat u}{m},(\power{\hat u}{m})^{(\theta)}_{r}]}{s^{\frac{m+1}{m}}}
      \, \dx \dt \\
&\leq c\mathcal{R}_{r,s}^2\Bigg(\biint_{Q_\rho^{(\theta)}}\frac{\big|\power{\hat
    u}{m}-(\power{\hat u}{m})^{(\theta)}_{r}\big|^2}{\rho^2}\d x\d
t\Bigg)^{\frac{m-1}{2m}}
  \biint_{Q_s^{(\theta)}} 
  \frac{\b[\power{\hat u}{m},(\power{\hat u}{m})^{(\theta)}_{r}]}{s^{\frac{m+1}{m}}}
  \, \dx \dt\\
  &\qquad+
  c\mathcal{R}_{r,s}^2
  \bigg(\frac{\big|(\power{\hat u}{m})^{(\theta)}_{r}\big|^2}{\rho^2}\bigg)^{\frac{m-1}{2m}}\biint_{Q_s^{(\theta)}} 
  \frac{\b[\power{\hat u}{m},(\power{\hat u}{m})^{(\theta)}_{r}]}{s^{\frac{m+1}{m}}}
  \, \dx \dt\\
&=: c \mathcal{R}_{r,s}^2 \mathrm{II}_1 + c\mathcal{R}_{r,s}^2 \mathrm{II}_2.
\end{align*}
Now we can estimate
\begin{align*}
\mathrm{II}_1 &\leq \Bigg(\biint_{Q_\rho^{(\theta)}}\frac{\big|\power{\hat
    u}{m}-(\power{\hat u}{m})^{(\theta)}_{r}\big|^2}{\rho^2}\d x\d
t\Bigg)^{\frac{m-1}{2m}}
  \Bigg(\biint_{Q_s^{(\theta)}} 
  \frac{\big|\power{\hat u}{m}-(\power{\hat u}{m})^{(\theta)}_{r} \big|^2}{s^2}
  \, \dx \dt\Bigg)^\frac{m+1}{2m}\\
&\leq c \biint_{Q_s^{(\theta)}} 
  \frac{\big|\power{\hat u}{m}-(\power{\hat u}{m})^{(\theta)}_{r} \big|^2}{s^2}
  \, \dx \dt\\
&\leq c \biint_{Q_s^{(\theta)}} 
  \frac{\big|\power{\hat u}{m}-(\power{\hat u}{m})^{(\theta)}_{s} \big|^2}{s^2}
  \, \dx \dt
\end{align*}
by first using Lemma~\ref{lem:b}\,(ii) and H\"older's inequality, and then Lemma~\ref{lem:uavetoa}. On the other hand we obtain
\begin{align*}
\mathrm{II}_2 &\leq c  \biint_{Q_s^{(\theta)}} \big|(\power{\hat u}{m})^{(\theta)}_{r}\big|^\frac{m-1}{m} 
  \frac{\b[\power{\hat u}{m},(\power{\hat u}{m})^{(\theta)}_{r}]}{s^2}
  \, \dx \dt\\
&\leq c \biint_{Q_s^{(\theta)}} \frac{\big| \power{\hat u}{m} - (\power{\hat u}{m})^{(\theta)}_{r}\big|^2}{s^2}
  \, \dx \dt\\
&\leq c \biint_{Q_s^{(\theta)}} \frac{\big| \power{\hat u}{m} - (\power{\hat u}{m})^{(\theta)}_{s}\big|^2}{s^2}
  \, \dx \dt
\end{align*}
by Lemmas~\ref{lem:b}\,(iii) and \ref{lem:uavetoa}.
%
Collecting the estimates and applying Lemma~\ref{lem:initialSobolev}, we arrive at 
\begin{align*}
\sup_{t\in \Lambda_r^{(\theta)}}& \bint_{B_r} \theta^{m-1}
\frac{\b\big[\power{\hat u}{m}(t),(\power{\hat
    u}{m})_r^{(\theta)}\big]}{r^{\frac{m+1}{m}}} \d x +
\biint_{Q_r^{(\theta)} } |D\power{\hat u}{m}|^2 \d x\d t \\
& \leq
  c\mathcal{R}_{r,s}^{\frac{4m}{m+1}}
    \biint_{Q_s^{(\theta)}} \frac{\babs{\power{\hat u}{m} -
        (\power{\hat u}{m})^{(\theta)}_{s}}^2}{s^2} \, \dx\dt\\
    &\qquad
    +c\,\biint_{Q_{s,+}^{(\theta)}} \Big[ |F|^2+|\partial_t\power gm|^{\frac{2m}{2m-1}}+|D\power{g}{m}|^2 \Big] \d x \d
   t \\
   &\leq c\eps\mathcal{R}_{r,s}^{\frac{4m}{m+1}} \sup_{t\in \Lambda_{s}^{(\theta)}} \bint_{B_{s}}
\theta^{m-1} \frac{\b[\power{\hat u}{m}(\cdot,t), (\power{\hat u}{m})_s^{(\theta)}]}{s^{\frac{m+1}{m}}}\, \dx \\
&\hspace{0,5cm}+ \frac{c\mathcal{R}_{r,s}^{\frac{4m}{m+1}}}{\eps^{\frac{2}{n}}} \Bigg[
\biint_{Q_{2\rho}^{(\theta)}(z_o)} \babs{D\power{\hat u}{m}}^{q}\, \dx \dt \Bigg]^{\frac{2}{q}} \\
&\hspace{0,5cm}+ c\mathcal{R}_{r,s}^{\frac{4m}{m+1}} \biint_{Q_{2\rho,+}^{(\theta)} } \Big[ \babs{F}^{2}+|\partial_t\power gm|^{\frac{2m}{2m-1}}+|D\power{g}{m}|^2 \Big] \, \dx \dt.
\end{align*}
Choosing $\eps=\frac1{2c\mathcal{R}_{r,s}^{\frac{4m}{m+1}}}$ and using
the Iteration Lemma \ref{lem:iteration} in order to reabsorb the
$\sup$-term into the left-hand side, we deduce the asserted estimate. 
\end{proof}

Next we prove the reverse H\"older inequality in the degenerate case. 

\begin{lemma}
\label{lem:reverse_initial_intrinsic_2}
Let $m>1$ and $u$ be a weak solution to \eqref{equation:PME} where
the vector field $\A$ satisfies \eqref{assumption:A} and the
Cauchy-Dirichlet datum $g$ fulfills \eqref{assumption:g}. Then on any cylinder
$Q_{2\rho}^{(\theta)}(z_o) \subset \Omega\times(-T,T)$ with
$z_o\in\Omega\times[0,T)$
which satisfies the coupling
\begin{align}\label{initial_intrinsic_coupling_2} \notag
\biint_{Q_{2\rho}^{(\theta)}(z_o)} &\frac{|\hat u|^{2m}}{(2\rho)^2} \d x\d t\leq \theta^{2m} \\ 
&\leq  K\biint_{Q_{\rho}^{(\theta)}(z_o)} \Big[ \babs{D\power{\hat u}{m}}^2 + \Big( |F|^2+|\partial_t\power gm|^{\frac{2m}{2m-1}}+|D\power{g}{m}|^2 \Big) \chi_{\{t>0\}} \Big] \d x\d t
\end{align}
for some $0< \rho\leq 1$ and $\theta\geq 1$ we have the following reverse H\"older inequality
\begin{align*}
 &\biint_{Q_\rho^{(\theta)}(z_o)} |D\power{\hat u}{m}|^2 \d x\d t \\
 &\qquad\leq c\left(\biint_{Q_{2\rho}^{(\theta)}(z_o)}
 |D\power{\hat u}{m}|^{q}  \d x \d t \right)^{\frac2q}\\
 &\qquad\qquad+
 c\biint_{Q_{2\rho,+}^{(\theta)}(z_o)} \Big[ |F|^2 +
 |\partial_t\power gm|^{\frac{2m}{2m-1}}+|D\power gm|^2 \Big] \d x\d t 
\end{align*}
for a constant $c = c(n,m,\nu,L,K)$ and for $q:= \frac{2n}{d} < 2$.
\end{lemma}
\begin{proof}
  Similarly as in the proof of the preceding lemma, we start with
  estimate \eqref{caccio-initial}. The term $\mathrm{I}$ is estimated
  in the same way as before, but now we estimate $\mathrm{II}$ by
  \begin{align*}
    \mathrm{II}
    &\leq
    c \mathcal{R}_{r,s}^2
    \biint_{Q_s^{(\theta)}} \theta^{m-1}
     \frac{\b[\power{\hat u}{m},(\power{\hat
         u}{m})^{(\theta)}_{r}]}{s^{\frac{m+1}{m}}} \, \dx \dt \\
     &\le
     \delta\theta^{2m}
     +
     c_\delta\mathcal{R}_{r,s}^{\frac{4m}{m+1}}\biint_{Q_s^{(\theta)}(z_o)}\frac{\b[\power{\hat u}{m},(\power{\hat
         u}{m})^{(\theta)}_{r}]^{\frac{2m}{m+1}}}{s^2}\d x\d t\\
          &\le
     \delta\theta^{2m}
     +
     c_\delta\mathcal{R}_{r,s}^{\frac{4m}{m+1}}\biint_{Q_s^{(\theta)}(z_o)}\frac{\big|\power{\hat u}{m}-(\power{\hat u}{m})^{(\theta)}_{s}\big|^2}{s^2}\d x\d t.
  \end{align*}
 Using assumption~\eqref{initial_intrinsic_coupling_2}$_2$ to bound
 the first term and Lemma~\ref{lem:initialSobolev} for the estimate of
 the second, we deduce
 \begin{align*}
   \mathrm{II}
   &\le
   K\delta\biint_{Q_{\rho}^{(\theta)}(z_o)} \Big[ \babs{D\power{\hat
       u}{m}}^2 + \Big( |F|^2+|\partial_t\power gm|^{\frac{2m}{2m-1}}+|D\power{g}{m}|^2 \Big) \chi_{\{t>0\}} \Big] \, \d x\d t\\
   &\qquad+
   c_\delta\mathcal{R}_{r,s}^{\frac{4m}{m+1}}\eps \sup_{t\in \Lambda_{s}^{(\theta)}} \bint_{B_{s}}
\theta^{m-1} \frac{\b[\power{\hat u}{m}(\cdot,t), (\power{\hat u}{m})_s^{(\theta)}]}{s^{\frac{m+1}{m}}}\, \dx \\
&\qquad+ \frac{c_\delta\mathcal{R}_{r,s}^{\frac{4m}{m+1}}}{\eps^{\frac{2}{n}}} \Bigg[
\biint_{Q_{2\rho}^{(\theta)}(z_o)} \babs{D\power{\hat u}{m}}^{q}\, \dx \dt \Bigg]^{\frac{2}{q}} \\
&\qquad+ c\biint_{Q_{2\rho,+}^{(\theta)}(z_o)} \Big[ |F|^2 +
 |\partial_t\power gm|^{\frac{2m}{2m-1}}+|D\power gm|^2 \Big] \d x\d t.
 \end{align*}
 Choosing first $\delta$ and then $\eps$ small in the form
 $\delta=\frac1{4K}$ and
 $\eps=\frac1{2c_\delta\mathcal{R}_{r,s}^{\frac{4m}{m+1}}}$ allows to
 re-absorb the $\sup$-term and the term with $|D\power{\hat u}{m}|^2$
 with the help of Lemma~\ref{lem:iteration}. Therefore, we arrive at
 the claim similarly as in the proof of Lemma~\ref{lem:reverse_initial_intrinsic}.
\end{proof}

\section{Proof of higher integrability}
\label{sec:hi-int}
\subsection{Extension of the boundary values}
We consider the cylinder
$Q_{8R}(y_o,\tau_o)
\subset\R^n\times(-T,T)$ with $R\in
(0,1]$ and $(y_o,\tau_o)\in\Omega_T\cup\partial_{\mathrm{par}}\Omega_T$. Since the center will be fixed
throughout this section,
we will simply write $Q_\rho:=Q_\rho(y_o,\tau_o)$ for $\rho>0$.
We fix a specific extension of the boundary values in order to derive
an estimate on the cylinder $Q_{R}$. To this end, we choose a
standard cut-off function $\eta\in C_0^\infty(B_{8R},[0,1])$ with
$\eta\equiv1$ in $B_{4R}$ and $|D\eta|\le\frac1R$ in $B_{8R}$. We
assume that the extension of the boundary values is given
by $\power gm=E(\eta\power gm)$ on $Q_{8R,+}\setminus\Omega_T$,
where the extension operator $E$ 
from Definition~\ref{def:sobolev-extension} is applied separately on
each time slice. Then, for each
fixed time $t\in\Lambda_{4R}(\tau_o)\cap(0,T)$
we have the estimates
\begin{align}
  \label{extension1}
  \int_{B_{4R}\times\{t\}}|D\power gm|^{2+\eps}\d x
  &\le
  \|E(\eta\power gm)\|_{W^{1,2+\eps}(\R^n\times\{t\})}^{2+\eps}\\\nonumber
  &\le
  c_E^{2+\eps} \|\eta\power gm\|_{W^{1,2+\eps}(\Omega\times\{t\})}^{2+\eps}\\\nonumber
  &\le
  c(c_E)\int_{\Omega\cap B_{8R}\times\{t\}}\Big(|D\power
  gm|^{2+\eps}+\frac{|\power gm|^{2+\eps}}{R^{2+\eps}}\Big)\d x.
\end{align}
In the case $n>2$, we use H\"older's inequality and Sobolev's
embedding to infer
\begin{align*}
  \int_{B_{4R}\times\{t\}}\frac{|\power gm|^{2+\eps}}{R^{2+\eps}}\d x
  &\le
  c(n)\|\power gm\|_{L^{(2+\eps)_*}(B_{4R}\times\{t\})}^{2+\eps}\\\nonumber
  &\le
  c(n)\|E(\eta\power gm)\|_{W^{1,2+\eps}(\R^n\times\{t\})}^{2+\eps}\\\nonumber
  &\le
  c(n,c_E)\int_{\Omega\cap B_{8 R}\times\{t\}}\Big(|D\power
  gm|^{2+\eps}+\frac{|\power gm|^{2+\eps}}{R^{2+\eps}}\Big)\d x,
\end{align*}
where the last estimate follows from~\eqref{extension1}.
In dimension $n=2$, we use the Sobolev embedding
$W^{1,2+\eps}(\R^n)\subset C^{0,\alpha}(\R^n)$ with
$\alpha=\frac{\eps}{2+\eps}$, which yields
\begin{align*}
  &\int_{B_{4R}\times\{t\}}\frac{|\power gm|^{2+\eps}}{R^{2+\eps}}\d x
  \le c(n)R^{-\eps}\|\power gm\|_{L^\infty(B_{4R}\times\{t\})}^{2+\eps}\\\nonumber
  &\qquad \le
  c(n)R^{-\eps}\bigg(\bint_{\Omega\cap B_{4R}\times\{t\}}|\power
gm|\,\d x+\mathrm{osc}_{B_{4R}}(\power gm)\bigg)^{2+\eps}\\\nonumber
  &\qquad \le
  c(n)R^{-\eps}\bigg(\bint_{\Omega\cap B_{4R}\times\{t\}}|\power
  gm|\d x
  +
  R^{\alpha}[E(\eta\power  gm)]_{C^{0,\alpha}(\R^n\times\{t\})}\bigg)^{2+\eps}\\\nonumber
  &\qquad\le
  c(n)R^{-\eps}\bint_{\Omega\cap B_{4R}\times\{t\}}|\power
  gm|^{2+\eps}\d x
  +
  c(n)\|E(\eta\power gm)\|_{W^{1,2+\eps}(\R^n\times\{t\})}^{2+\eps}\\\nonumber
  &\qquad\le
  c(n,\alpha,c_E)\int_{\Omega\cap B_{8R}\times\{t\}}\Big(|D\power
  gm|^{2+\eps}+\frac{|\power gm|^{2+\eps}}{R^{2+\eps}}\Big)\d x,
\end{align*}
where we used the measure density property~\eqref{measure_density} and
\eqref{extension1} in the last step. From the three preceding
estimates, we deduce the bound 
\begin{equation}\label{extension-bound}
  \biint_{Q_{4R}}\Big(|D\power gm|^{2+\eps}+\frac{|\power
    gm|^{2+\eps}}{R^{2+\eps}}\Big)\d x\,\d t
  \le
  c\,\biint_{Q_{8R}\cap\Omega_T}\Big(|D\power
  gm|^{2+\eps}+\frac{|\power gm|^{2+\eps}}{R^{2+\eps}}\Big)\d x\,\d t.
\end{equation}
Using the extension of the boundary values specified above, we now define
\begin{align*}
\lambda_o :=1+ \left(  \biint_{Q_{4R}} \left[ 2\frac{|\hat\u^m-\hat\g^m|^2+|\hat g|^{2m}}{(4R)^2}+|D\u^m|^2\chi_{\Omega_T}+ G^2 \right] \d x\d t \right)^{\frac{1}{m+1}}
\end{align*}
where $\hat u$ and $\hat g$ are defined in \eqref{def:uhat} and  
\begin{align*}
 G^2 := |D\g^m|^2\chi_{\{t>0\}}+\big(|F|^2+ |\partial_t \g^m|^{\frac{2m}{2m-1}}\big) \chi_{\Omega_T}
\end{align*}
We use~\eqref{extension-bound} and Lemma~\ref{lem:lateralcacciop} with
$\theta  = 1$ in order to estimate
 \begin{align}\label{bound-lambda0}
   \lambda_o^{m+1}
   &\leq
   c\Bigg(1+ 
   \,\biint_{Q_{8R}\cap\Omega_T}\frac{|\u^m-\g^m|^2}{R^2}\,\d x \d t\Bigg)\\\nonumber
   &\qquad
   +c\bigg(\biint_{Q_{8R}\cap\Omega_T} \left( G^{2+\eps}
   +\frac{|g |^{m(2+\eps)}}{R^{2+\eps}} \right) \, \d x \d t\bigg)^{\frac2{2+\eps}}.
 \end{align}
For the estimates we also used the measure density condition
\eqref{measure_density}, which implies $|Q_{8R}\cap \Omega_T|\ge
c|Q_{8R}|$.

\subsection{Construction of a non-uniform system of cylinders}
\label{construction_cylinder}

The following construction is inspired by the one in
\cite{Boegelein-Duzaar-Korte-Scheven,Gianazza-Schwarzacher,Schwarzacher}.
However, the boundary case becomes much more involved
due to the fact that the notion of intrinsic cylinder is different at
the lateral boundary compared to the interior of the domain. The
transition between both cases requires additional carefulness.

For $z_o \in Q_{2R}$, we write $d_o:= \frac 12 \dist(x_o,\partial
\Omega)$.
We observe that $Q_{\rho}^{(\theta)}(z_o) \subset Q_{4R}$ whenever $\rho \in (0,R]$ and $\theta\geq 1$.

For $\rho\in (0,R]$, we define the parameter $\tilde \theta_\rho \equiv \tilde
\theta_{z_o;\rho}$  by 
\begin{align*}
\tilde \theta_\rho 
:=
\inf\bigg\{\theta\in [\lambda_o,\infty)
\,:\,\frac{1}{|Q_\rho|} \iint_{Q_{\rho}^{(\theta)}(z_o)} \frac{|\hat
  u|^{2m}}{\rho^2} \d x\d t\leq \theta^{m+1}\bigg\},
\end{align*}
if $\rho<d_o$, while in the case $\rho\geq d_o$, we let
\begin{align*}
  \tilde \theta_\rho
  :=
  \inf\bigg\{\theta\in [\lambda_o,\infty)\,:\,
  \frac{1}{|Q_\rho|} \iint_{Q_{\rho}^{(\theta)}(z_o)}2 \frac{|\hat\u^m-\hat\g^m|^{2}+|\hat g|^{2m}}{\rho^2} \d x\d t\leq \theta^{m+1}\bigg\}.
\end{align*}
Observe that $\tilde \theta_\rho$ is well defined, since the integral
condition is satisfied for some $\theta \geq \lambda_o$. This follows
from the fact that in the limit $\theta \to \infty$, the integral on the left-hand side converges to zero, while the right-hand side blows up. Note that we can rewrite the condition for the integral in the definition of $\tilde \theta_\rho$ as 
\begin{align*}
\left\{
\begin{array}{ll}
 \biint_{Q_{\rho}^{(\theta)}(z_o)} \frac{|\hat u|^{2m}}{\rho^2} \d x\d t\leq \theta^{2m}, & \text{if } \rho< d_o \vspace{2mm}\\
 \biint_{Q_{\rho}^{(\theta)}(z_o)} 2\frac{|\hat\u^m-\hat\g^m|^{2}+|\hat g|^{2m}}{\rho^2} \d x\d t\leq \theta^{2m} , & \text{if } \rho\geq d_o.
\end{array}
\right.
\end{align*}
By the very definition of $\tilde \theta_\rho$ we either have 
\begin{align*}
\tilde \theta_\rho =\lambda_o  \quad \text{ and } \quad
\left\{
\begin{array}{ll}
 \biint_{Q_{\rho}^{(\tilde\theta\rho)}(z_o)} \frac{|\hat u|^{2m}}{\rho^2} \d x\d t\leq \tilde\theta_\rho^{2m} =\lambda_o^{2m},& \text{if } \rho< d_o \vspace{2mm} \\
\biint_{Q_{\rho}^{(\theta)}(z_o)} 2\frac{|\hat\u^m-\hat\g^m|^{2}+|\hat g|^{2m}}{\rho^2} \d x\d t\leq \tilde\theta_\rho^{2m} =\lambda_o^{2m}, & \text{if } \rho\geq d_o 
\end{array}
\right.
\end{align*}
or 
\begin{align}
\label{est:tilde_theta}
\tilde \theta_\rho >\lambda_o \quad \text{ and } \quad
\left\{
\begin{array}{ll}
  \biint_{Q_{\rho}^{(\tilde\theta_\rho)}(z_o)} \frac{|\hat u|^{2m}}{\rho^2} \d x\d t= \tilde \theta_\rho^{2m}, & \text{if } \rho< d_o \vspace{2mm} \\
 \biint_{Q_{\rho}^{(\tilde\theta_\rho)}(z_o)} 2\frac{|\hat\u^m-\hat\g^m|^{2}+|\hat g|^{2m}}{\rho^2} \d x\d t= \tilde \theta_\rho^{2m} , & \text{if } \rho\geq d_o. 
 \end{array}
\right.
\end{align}
In any case we have $\tilde\theta_R \geq \lambda_o\geq 1$. If $\lambda_o<\tilde \theta_R$ and $R\geq d_o$ then we obtain 
\begin{align*}
\tilde \theta^{m+1}_R &= \frac{1}{|Q_R|}\iint_{Q_{R}^{(\tilde\theta_R)}(z_o)} 2\frac{|\hat\u^m-\hat \g^m|^2+|\hat g|^{2m}}{R^2} \d x\d t\\
&\leq \frac{4^2}{|Q_R|}\iint_{Q_{R}^{(\tilde\theta_R)}(z_o)} 2\frac{|\hat\u^m-\hat\g^m|^2+|\hat g|^{2m}}{(4R)^2} \d x\d t \leq 4^{d+2}\lambda_o^{m+1},
\end{align*}
where we used the fact that $Q_R^{(\tilde\theta_R)}\subset Q_{4R}$. If
$R<d_o$ we argue similarly. In any case, we obtain
\begin{equation}\label{bound-theta-R}
\tilde \theta_R \leq 4^{\frac{d+2}{m+1}} \lambda_o.
\end{equation}

Next, we prove that the function $\tilde \theta$ is piecewise continuous: 

\begin{lemma}\label{lem:jump-up}
For fixed $z_o$ the function $\rho \mapsto\tilde \theta_\rho$ is continuous on $(0,d_o)$ and $[d_o,R)$, and we have
$$
\lim_{\rho\uparrow d_o}\tilde \theta_\rho \leq \lim_{\rho \downarrow d_o} \tilde \theta_\rho.
$$
\end{lemma}
\begin{proof}[Proof.]
Without loss of generality, we may assume that $d_o\in (0,R)$. If $\rho \in (0,d_o)$ the proof works as in \cite[Section 6.1]{Boegelein-Duzaar-Korte-Scheven}. If $\rho\in [d_o,R]$ the idea still remains the same, but we will present the proof for convenience.
Therefore, we consider $\rho\in [d_o,R]$ and $\varepsilon>0$ and define $ \theta_+:=\tilde \theta_\rho+\varepsilon$. Then there exists $\delta=\delta(\varepsilon,\rho)>0$ such that 
$$
\frac{1}{|Q_r|} \iint_{Q_r^{(\theta_+)}(z_o)} 2\frac{|\hat\u^m-\hat\g^m|^2+|\hat g|^{2m}}{r^2} \d x\d t < \theta_+^{m+1}
$$
for all $r\in [d_o,R]$ with $|r-\rho|<\delta$. This can be verified by
observing that the strict inequality holds true if $r=\rho$ and that
both sides are continuous with respect to the radius. This shows
$\tilde \theta_r \leq \theta_+=\tilde \theta_\rho+\varepsilon$ if
$|r-\rho|<\delta$. To prove the reverse inequality we set
$\theta_-:=\tilde \theta_\rho-\varepsilon$. If $\theta_-\leq
\lambda_o$ the desired estimate follows directly from the construction. In the other case we obtain
\begin{align*}
\frac{1}{|Q_r|} \iint_{Q_r^{(\theta_-)}(z_o)} 2\frac{|\hat\u^m-\hat\g^m|^2+|\hat g|^{2m}}{r^2} \d x\d t > \theta_-^{m+1}
\end{align*}
for all $r\in [d_o,R]$ with $|r-\rho|<\delta$, where $\delta=\delta(\varepsilon,\rho)$ was possibly diminished. For $r=\rho$, this follows again directly from the definition, since otherwise we would have $\tilde \theta_\rho \leq \theta_-$, which is a contradiction. For $r$ with $|r-\rho|<\delta$ the claim follows from the continuity of both sides as a function of $r$. This implies $\tilde \theta_r \geq \theta_-=\tilde \theta_\rho-\varepsilon$ and consequently the map $\rho\mapsto \tilde\theta_\rho$ is continuous on $[d_o,R]$.

The fact that $\tilde \theta_\rho$ jumps upwards at $d_o$ follows
directly from the definition of $\tilde \theta$, since
$|\hat u|^{2m}\le 2(|\hat\u^m-\hat\g^m|^2+|\hat g|^{2m})$.
\end{proof}

Unfortunately, the mapping $\rho\mapsto \tilde\theta_\rho$ might not be monotone or continuous at the point $d_o$. This forces us to modify $\tilde\theta_\rho$ in the following way
$$
\theta_\rho \equiv\theta_{z_o,\rho} := \max_{r\in [\rho,R]} \tilde \theta_{z_o,r}.
$$
Then, by Lemma \ref{lem:jump-up} and the construction, the map
$\rho\mapsto\theta_\rho$ is continuous and monotonically decreasing. This construction can be considered as a rising sun construction (see Figure~\ref{fig:sunrise}).

\begin{figure}[h] 
\centering
\includegraphics[width=280pt]{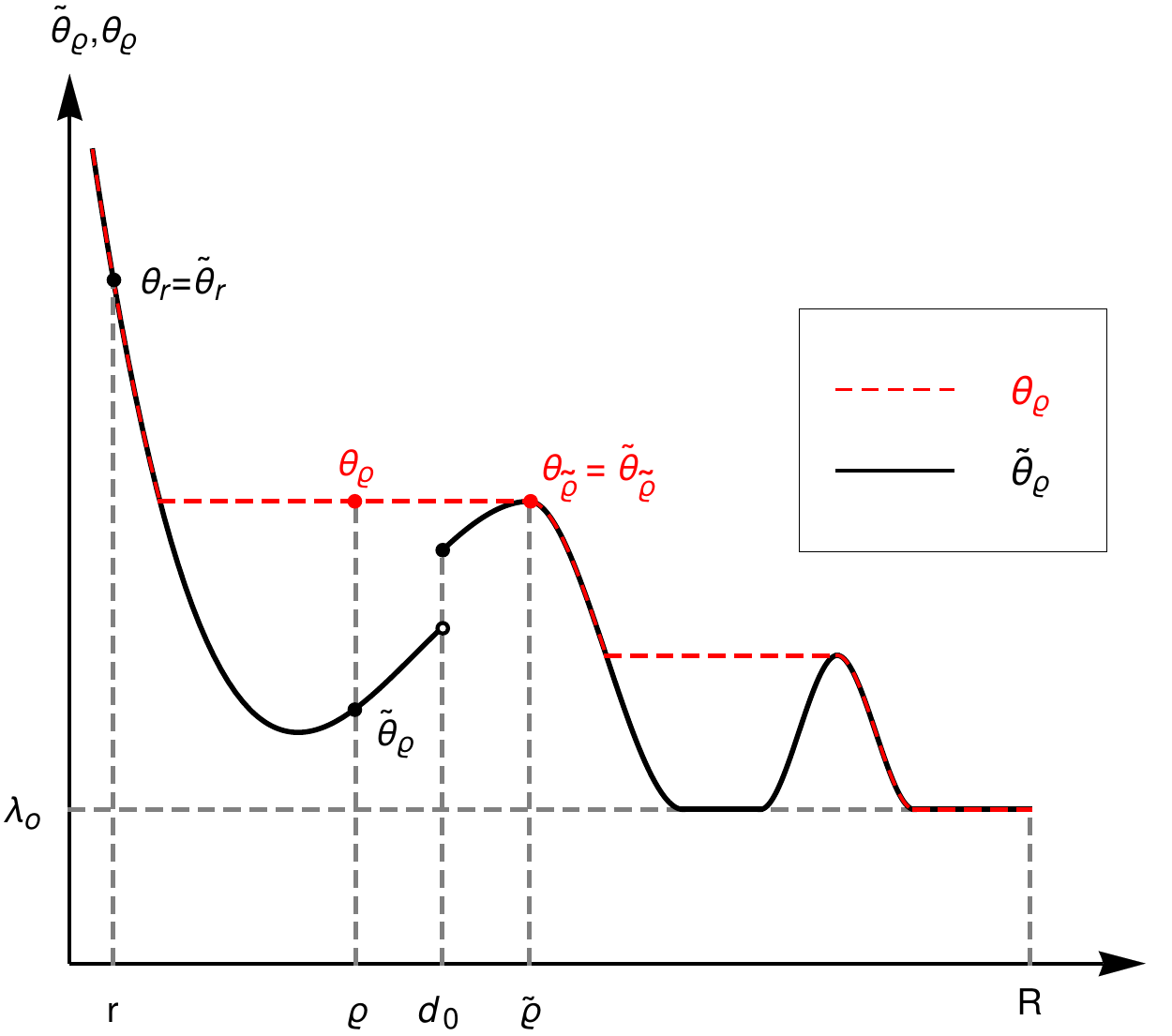}
\caption{Illustration of the rising sun construction.}\label{fig:sunrise}
\end{figure}

Next, we define 
\begin{align}
\label{change_radius}
\tilde \rho :=\left\{
\begin{array}{cl}
R& \text{ if } \theta_\rho =\lambda_o \vspace{2mm},\\
\min \{s\in [\rho,R]:\theta_s =\tilde\theta_s\} & \text{ if } \theta_\rho >\lambda_o,
\end{array}
\right.
\end{align}
	i.e.\ we have $\theta_r=\theta_{\tilde \rho}$ for any $r\in[\rho,\tilde \rho]$.

\begin{lemma}
\label{lem:properties_theta}
\begin{itemize}
\item[(i)] For any  $0<\rho\leq s\leq R$, the constructed cylinders
  $Q_s^{(\theta_\rho)}$ are sub-intrinsic in the sense 
\begin{align*}
\left\{
\begin{array}{ll}
\biint_{Q_s^{(\theta_\rho)}(z_o) } \frac{|\hat u|^{2m}}{s^2} \d x\d t \leq \theta_\rho^{2m}, & \text{if } s< d_o, \vspace{2mm}\\
\biint_{Q_s^{(\theta_\rho)}(z_o) } 2\frac{|\hat\u^m-\hat\g^m|^2+|\hat g|^{2m}}{s^2} \d x\d t \leq \theta_\rho^{2m}, & \text{if } s\geq d_o.
\end{array}
\right. 
\end{align*} 
\item[(ii)] For any $s\in (\rho,R]$ we have
\begin{align*}
\theta_\rho \leq \left(\frac{s}{\rho} \right)^{\frac{d+2}{m+1}} \theta_s  .
\end{align*}	
\item[(iii)] For  any $0<\rho\leq R$ there holds
\begin{align*}
\theta_\rho \leq \left(\frac{R}{\rho} \right)^{\frac{d+2}{m+1}} \theta_R \leq \left(\frac{4R}{\rho} \right)^{\frac{d+2}{m+1}} \lambda_o.
\end{align*}
\end{itemize}
\end{lemma}

\begin{proof}[Proof.]
(i) By definition we have $\tilde \theta_s \leq \theta_\rho$  so that $Q_s^{(\theta_\rho)}(z_o)\subset Q_s^{(\tilde \theta_s)}(z_o)$ and therefore if $s<d_o$
\begin{align*}
\biint_{Q_s^{(\theta_\rho)}(z_o) } \frac{|\hat u|^{2m}}{s^2} \d x\d t &\leq \left(\frac{\theta_\rho}{\tilde \theta_s}\right)^{m-1} \biint_{Q_s^{(\tilde\theta_s)}(z_o) } \frac{|\hat u|^{2m}}{s^2} \d x\d t \\
& \leq \left(\frac{\theta_\rho}{\tilde \theta_s}\right)^{m-1} \tilde \theta_s^{2m} =\theta_\rho^{m-1} \tilde \theta_s^{m+1} \leq \theta_\rho^{2m}.
\end{align*}
If $s \geq d_o$ we obtain 
\begin{align*}
&\biint_{Q_s^{(\theta_\rho)}(z_o) } 2\frac{|\hat \u^m -\hat
  \g^m|+|\hat g|^{2m}}{s^2} \d x\d t \\
&\qquad\leq \left(\frac{\theta_\rho}{\tilde \theta_s}\right)^{m-1} \biint_{Q_s^{(\tilde\theta_s)}(z_o) } 2\frac{|\hat \u^m -\hat \g^m|+|\hat g|^{2m}}{s^2} \d x\d t \\
&\qquad\leq \left(\frac{\theta_\rho}{\tilde \theta_s}\right)^{m-1} \tilde \theta_s^{2m} =\theta_\rho^{m-1} \tilde \theta_s^{m+1} \leq \theta_\rho^{2m}.
\end{align*}
(ii) If $\theta_\rho=\lambda_o$ we know that also $\theta_s=\lambda_o$, so that the claim holds true. In the case $\theta_\rho>\lambda_o$ we first consider radii $s$ with $s\in (\rho,\tilde \rho]$. Then, $\theta_\rho=\theta_s$ and there is nothing to prove. In contrary, if $s\in (\tilde \rho,R]$ the monotonicity of $\rho\mapsto \theta_\rho$, \eqref{est:tilde_theta} and the first part of the Lemma imply  in the case $\tilde \rho<d_o$
\begin{align*}
\theta_\rho&=\tilde \theta_{\tilde\rho} = \left[ \frac{1}{|Q_{\tilde \rho}|}\iint_{Q_{\tilde \rho}^{(\theta_{\tilde \rho})}(z_o)} \frac{|\hat u|^{2m}}{\tilde \rho^2} \d x\d t\right]^{\frac{1}{m+1}} \\
&\leq \left( \frac{s}{\tilde \rho}\right)^{\frac{d+2}{m+1}} \left[ \frac{1}{|Q_{s}|} \iint_{Q_{s}^{(\theta_{s})}(z_o)} \frac{|\hat u|^{2m}}{s^2} \d x\d t\right]^{\frac{1}{m+1}}\\
& \leq \left( \frac{s}{ \rho}\right)^{\frac{d+2}{m+1}} \theta_s.
\end{align*}
On the other hand, if $\tilde \rho \geq d_o$
\begin{align*}
\theta_\rho&=\tilde \theta_{\tilde\rho} = \left[ \frac{1}{|Q_{\tilde \rho}|}\iint_{Q_{\tilde \rho}^{(\theta_{\tilde \rho})}(z_o)} 2\frac{|\hat \u^m-\hat \g^m|^2+|\hat g|^{2m}}{\tilde \rho^2} \d x\d t\right]^{\frac{1}{m+1}} \\
&\leq \left( \frac{s}{\tilde \rho}\right)^{\frac{d+2}{m+1}} \left[ \frac{1}{|Q_{s}|} \iint_{Q_{s}^{(\theta_{s})}(z_o)} 2\frac{|\hat \u^m-\hat \g^m|^2+|\hat g|^{2m}}{s^2} \d x\d t\right]^{\frac{1}{m+1}}\\
& \leq \left( \frac{s}{ \rho}\right)^{\frac{d+2}{m+1}} \theta_s.
\end{align*}
(iii) Combining (ii) for $s=R$ with estimate~\eqref{bound-theta-R}
yields the claim, since $\theta_R=\tilde \theta_R$. 
\end{proof}

We note that due to the monotonicity of the map $\rho\mapsto \theta_{\rho,z_o}$ the above constructed cylinders are nested in the sense that
$$
Q_r^{(\theta_{z_o,r})}(z_o) \subset Q_s^{(\theta_{z_o,s})}(z_o) \quad \text{ whenever } 0<r<s\leq R.
$$
However, these cylinders in general only fulfill the sub-intrinsic coupling condition from Lemma \ref{lem:properties_theta} (i).

\subsection{Covering property} Next, we will present a Vitali type
covering property for the above constructed cylinders. Using the just
established bounds from Lemma \ref{lem:properties_theta}, this
result can be established by a slight adaptation of the arguments
in \cite[Lemma 6.1]{Boegelein-Duzaar-Korte-Scheven} which is based on the ideas of~\cite[Lemma 5.2]{Gianazza-Schwarzacher}. 

\begin{lemma} \label{lem:vitali}
There exists a constant $\hat c=\hat c(n,m)\geq 160$ such that the following holds true: Let $\mathcal F$ be any collection of cylinders $Q_{32r}^{(\theta_{z,r})}(z)$ where $Q_{r}^{(\theta_{z,r})}(z)$ is a cylinder of the form as constructed in section \ref{construction_cylinder} with radius $r\in (0,\frac{R}{\hat c})$. Then there exists a countable subfamily $\mathcal G$ of disjoint cylinders in $\mathcal F$ such that
\begin{align*}
\bigcup_{Q\in \mathcal F} Q \subset \bigcup_{Q\in \mathcal G} \hat Q,
\end{align*}
where $\hat Q$ denotes the $\frac{\hat c}{32}$-times enlarged cylinder $Q$, i.e.\ if $Q=Q_{32r}^{(\theta_{r,z})}(z)$, then $\hat Q= Q_{\hat c r}^{(\theta_{z,r})}(z)$.
\end{lemma}

\begin{proof}[Proof.]
For $j \in \N$ define a sub-collection of $\mathcal F$ as 
$$
	\mathcal F_j
	:=
	\big\{Q_{32r}^{(\theta_{z;r})}(z)\in \mathcal F: 
	\tfrac{R}{2^j\hat c}<r\le \tfrac{R}{2^{j-1}\hat c} \big\}.
$$
Next we construct a countable collection of disjoint cylinders
$\mathcal G \subset \mathcal F$ inductively as follows. Let $\mathcal
G_1$ be a maximal disjoint collection of cylinders in $\mathcal
F_1$. Observe that the measure of every cylinder in $\mathcal G_1$ is bounded from below by Lemma~\ref{lem:properties_theta} (iii), which implies that $\mathcal G_1$ is finite. For $k\geq 2$, define $\mathcal G_k$ to be any maximal sub-collection of 
$$
	\bigg\{Q\in \mathcal F_k: 
	Q\cap Q^\ast=\emptyset \mbox{ for any $ \displaystyle Q^\ast\in \bigcup_{j=1}^{k-1} \mathcal G_j $}
	\bigg\}.
$$
Collection $\mathcal G_k$ is again finite, and we can define 
$$
\mathcal G := \bigcup_{j\in \N} \mathcal G_j.
$$
Now $\mathcal G$ is a countable disjoint sub-collection of $\mathcal F$. To conclude the result we show that for any $Q \in \mathcal F$ there exists a cylinder $Q^* \in \mathcal G$ such that $Q \subset \hat Q^*$.

Fix $Q=Q_{32r}^{(\theta_{z;r})}(z) \in \mathcal F$. Then $Q \in \mathcal F_j$ for some $j \in \N$. Since $\mathcal G_{j}$ is maximal, there exists $Q^*=Q_{32r_*}^{(\theta_{z_*;r_*})}(z_*) \in \bigcup_{i=1}^j \mathcal G_j$ satisfying $Q \cap Q^* \neq \emptyset$. From the definitions it follows that $r \leq \frac{R}{2^{j-1}\hat c}$ and $r_* > \frac{R}{2^j \hat c}$, which implies $r \leq 2r_*$. Then clearly $B_{32r}(x) \subset B_{160r_*}(x_*)$. The main objective of the rest of the proof is to deduce the inclusion 
\begin{equation} \label{Lambda_inclusion}
\Lambda_{32r}^{(\theta_{z,r})} (t) \subset \Lambda_{\hat c r_*}^{(\theta_{z_*,r_*})} (t_*).
\end{equation}

Next we show that
\begin{align}
\label{est:theta} \theta_{z_*;r_*} \leq (4\mu)^{\frac{d+2}{m+1}} \theta_{z;r},
\end{align}
where $\mu = 128$. By $\tilde r_* \in [r_*,R]$ we denote the radius from \eqref{change_radius} associated to the cylinder $Q_{r_*}^{(\theta_{z_*;r_*})}(z_*)$. Recall that either $Q_{\tilde r_*}^{(\theta_{z_*;r_*})}(z_*)$ is intrinsic or $\tilde r_*=R$ and $\theta_{z_*;r_*}=\lambda_o$. In the latter case we have 
$$
\theta_{z_*;r_*} =\lambda_o \leq \theta_{z;r}.
$$ 
Therefore, we can assume that $Q_{\tilde r_*}^{(\theta_{z_*;r_*})}(z_*)$ is intrinsic, which means
\begin{align}
\label{intrinsic_cylinder}
\left\{
\begin{array}{ll}
\frac{1}{|Q_{\tilde r_*}|} \iint_{Q_{\tilde r_*}^{(\theta_{z_*;r_*})}(z_*)} \frac{|\hat u|^{2m}}{\tilde r_*^2} \d y\d \tau=\theta_{z_*;r_*}^{m+1} , & \text{if } \tilde r_*<\frac 12\dist(x_*,\partial \Omega) \vspace{2mm}\\
\frac{1}{|Q_{\tilde r_*}|} \iint_{Q_{\tilde r_*}^{(\theta_{z_*,r_*})}(z_*)}2 \frac{|\hat\u^m-\hat\g^m|^{2}+|\hat g|^{2m}}{\tilde r_*^2} \d y\d \tau=  \theta^{m+1}_{z_*;r_*}, & \text{if } \tilde r_*\geq \frac 12 \dist(x_*,\partial \Omega)
\end{array}
\right.
\end{align}

We first consider the case where $\tilde r_*>\frac R \mu$. Here we
obtain by triangle inequality in both of the cases $ \tilde
r_*<\frac12\dist(x_*,\partial \Omega) $ and $ \tilde
r_*\geq\frac12\dist(x_*,\partial \Omega) $ that 
\begin{align*}
  \theta^{m+1}_{z_*;r_*}
  &=
\frac{1}{|Q_{\tilde r_*}|} \iint_{Q_{\tilde r_*}^{(\theta_{z_*,r_*})}(z_*)}2 \frac{|\hat\u^m-\hat\g^m|^{2}+|\hat g|^{2m}}{\tilde r_*^2} \d y\d \tau \\
& \leq \left(\frac{4R}{\tilde r_*}\right)^2 \frac{1}{|Q_{\tilde r_*}|} \iint_{Q_{4R}}2 \frac{|\hat\u^m-\hat\g^m|^{2}+|\hat g|^{2m}}{(4R)^2} \d y\d \tau \\
&\leq  \left(\frac{4R}{\tilde r_*}\right)^{d+2} \lambda_o^{m+1} \leq (4\mu)^{d+2} \theta^{m+1}_{z;r}.
\end{align*}
This shows \eqref{est:theta} if $\tilde r_*>\frac R \mu$. Next, we assume that $\tilde r_*\leq\frac R \mu$. We can assume that $\theta_{z;r} \leq \theta_{z_*;r_*}$ since otherwise~\eqref{est:theta} follows directly. Since the cylinders $Q$ and $Q^*$ intersect, we have 
\begin{equation}\label{t-t_*}
	|t-t_*|\le \theta_{z;r}^{1-m}(32r)^{\frac{m+1}{m}}
	+
	\theta_{z_*;r_*}^{1-m}(32r_*)^{\frac{m+1}{m}}.
\end{equation}
Because $\rho \mapsto \theta_{z;\rho}$ is decreasing and $r \leq 2r_*\leq 2 \tilde r_* \leq \mu \tilde r_*$, we have that
$$
	\theta_{z_*;r_*}
	\ge 
	\theta_{z;r}
	\ge
	\theta_{z;\mu \widetilde r_*}.
$$
This implies that 
\begin{align*}
	\theta_{z_\ast;r_*}^{1-m}(32\widetilde r_*)^{\frac{m+1}{m}} + |t-t_*|
	&\le
    2\theta_{z_*;r_*}^{1-m}(32\widetilde r_*)^{\frac{m+1}{m}}+
    \theta_{z;r}^{1-m}(32r)^{\frac{m+1}{m}} \\
	&\le
	2\cdot 64^{\frac{m+1}{m}} 
	\theta_{z;\mu\widetilde r_*}^{1-m}\widetilde r_*^{\frac{m+1}{m}}
	\le
	\theta_{z;\mu \widetilde r_\ast}^{1-m}(\mu\widetilde r_\ast)^{\frac{m+1}{m}},
\end{align*}
which shows that 
$$
\Lambda_{32\tilde r_*}^{(\theta_{z_*; r_*})}(t_*) \subset \Lambda_{\mu\tilde r_*}^{(\theta_{z;\mu \tilde r_*})}(t)
$$
holds true. Since $| x - x_* | \leq 96 \tilde r_*$, we also have that $B_{32\tilde r_*}(x_*) \subset B_{\mu \tilde r_*}(x)$. Thus we have the inclusion 
\begin{equation} \label{Q_inclusion}
Q_{32 \tilde r_*}^{(\theta_{z_*;r_*})} (z_*) \subset Q_{\mu \tilde r_*}^{(\theta_{z;\mu \tilde r_*})} (z).
\end{equation}
If $\tilde r_*\geq \frac 12 \dist(x_*,\partial \Omega)$, then we also get
\begin{align*}
\frac 12 \dist(x,\partial \Omega) &\leq \frac 12 \dist(x_*,\partial \Omega)+\frac 12 |x-x_*|
 \leq \tilde r_* +\frac 12 (32 r+32 r_*)\leq 49\tilde r_* \leq \mu \tilde r_*.
\end{align*}
Therefore, using the intrinsic scaling and Lemma~\ref{lem:properties_theta}\,(i) leads to 
\begin{align*}
\theta^{m+1}_{z_*,r_*}& =\frac{1}{|Q_{\tilde r_*}|} \iint_{Q_{\tilde r_*}^{(\theta_{z_*,r_*})}(z_*)}2 \frac{|\hat\u^m-\hat\g^m|^{2}+|\hat g|^{2m}}{\tilde r_*^2} \d y\d \tau\\
&\leq \frac{\mu^{d+2}}{|Q_{\mu\tilde r_*}|} \iint_{Q_{\mu\tilde r_*}^{(\theta_{z,\mu \tilde r_*})}(z)}2 \frac{|\hat\u^m-\hat\g^m|^{2}+|\hat g|^{2m}}{(\mu\tilde r_*)^2} \d y\d \tau \\
&\leq
  \mu^{d+2}\theta^{m+1}_{z;\mu \tilde r_*}
  \le
  \mu^{d+2}\theta^{m+1}_{z;r}.
\end{align*}
On the other hand if $\tilde r_*<\frac 12 \dist(x_*,\partial \Omega)$ we obtain the same estimate. This can be seen as follows: If $\mu \tilde r_* <\frac 12 \dist(x,\partial \Omega)$ we have
\begin{align*}
\theta^{m+1}_{z_*,r_*}& =\frac{1}{|Q_{\tilde r_*}|} \iint_{Q_{\tilde r_*}^{(\theta_{z_*,r_*})}(z_*)} \frac{|\hat u|^{2m}}{\tilde r_*^2} \d y\d \tau\\
&\leq \frac{\mu^{d+2}}{|Q_{\mu\tilde r_*}|} \iint_{Q_{\mu\tilde r_*}^{(\theta_{z,\mu \tilde r_*})}(z)} \frac{|\hat u|^{2m}}{(\mu\tilde r_*)^2} \d y\d \tau \\
&\leq   \mu^{d+2}\theta^{m+1}_{z;\mu \tilde r_*}
  \le
\mu^{d+2}\theta^{m+1}_{z;r},
\end{align*}
and for $ \mu \tilde r_* \geq \frac 12 \dist(x,\partial \Omega)$ we
can use the triangle inequality to obtain
\begin{align*}
\theta^{m+1}_{z_*,r_*}& =\frac{1}{|Q_{\tilde r_*}|} \iint_{Q_{\tilde r_*}^{(\theta_{z_*,r_*})}(z_*)} \frac{|\hat u|^{2m}}{\tilde r_*^2} \d y\d \tau\\
&\leq \frac{\mu^{d+2}}{|Q_{\mu\tilde r_*}|} \iint_{Q_{\mu\tilde r_*}^{(\theta_{z,\mu \tilde r_*})}(z)}2 \frac{|\hat\u^m-\hat\g^m|^{2}+|\hat g|^{2m}}{(\mu\tilde r_*)^2} \d y\d \tau \\
&\leq  \mu^{d+2}\theta^{m+1}_{z;\mu \tilde r_*}
  \le
 \mu^{d+2}\theta^{m+1}_{z;r}.
\end{align*}
This finally shows \eqref{est:theta}. Now by using~\eqref{t-t_*}, $r \leq 2r_*$ and~\eqref{est:theta} we can estimate
\begin{align*}
	\theta_{z;r}^{1-m}(32r)^{\frac{m+1}{m}} + |t-t_*|
	&\le
	2\theta_{z;r}^{1-m}(32r)^{\frac{m+1}{m}} + 
	\theta_{z_*;r_*}^{1-m}(32r_*)^{\frac{m+1}{m}} \\
	&\le
	32^\frac{m+1}{m}\Big[ 1+2\cdot 2^\frac{m+1}{m}\cdot 512^\frac{(m-1)(d+2)}{m+1}\Big]
	\theta_{z_\ast;r_\ast}^{1-m}r_\ast^{\frac{m+1}{m}} \\
	&\le
	\theta_{z_\ast;r_\ast}^{1-m}(\hat c r_\ast)^{\frac{m+1}{m}},
\end{align*}
for a constant $\hat c = \hat c(n,m) > 32$, which shows the inclusion~\eqref{Lambda_inclusion}. After possibly enlarging $\hat c$ such that $\hat c \geq 160$ is satisfied, we have $Q \subset \hat Q^*$ which completes the proof.
\end{proof}

\subsection{Stopping time argument}
For $\lambda> \lambda_o$ and $r \in (0,2R]$ we define the super-level set 
$$
\mathbf{E}(r,\lambda):= \left\{z \in Q_r\cap \Omega_T: z \text{ is a Lebesgue point of } |D\u^m| \text{ and } |D\u^m|(z)>\lambda^m \right\},
$$
  where the notion of Lebesgue point has to be understood with respect
  to the system of cylinders constructed in
  Section~\ref{construction_cylinder}. Because of the Vitali type
  covering property from Lemma~\ref{lem:vitali}, almost every point is a Lebesgue
  point also in this sense,
  see \cite[2.9.1]{Federer}.
For fixed $0<R\leq R_1 <R_2 \leq 2R$ we consider the cylinders
$$
Q_R \subseteq Q_{R_1} \subset Q_{R_2} \subseteq Q_{2R}  .
$$
Let $z_o \in \mathbf{E}(R_1, \lambda)$. By definition of this set, we have
\begin{align*}
\liminf_{s\downarrow 0} \biint_{Q_s^{(\theta_s)}(z_o)} \Big[ |D\u^m|^2\chi_{\Omega_T}+G^2  \Big] \d x\d t \geq |D\u^m|^2(z_o)> \lambda^{2m}.
\end{align*} 
In the following, we consider values of $\lambda$ satisfying
\begin{align*}
 \lambda >B\lambda_o \quad \text{ where } \quad B:=\left(\frac{4\hat c R}{R_2-R_1} \right)^{\frac{n+2}{m+1}}>1,
\end{align*}
in which $\hat c$ is the constant from the Vitali covering lemma~\ref{lem:vitali}. For radii $s$ with
\begin{align*}
\frac{R_2-R_1}{\hat c} \leq s\leq R
\end{align*}
we obtain by using Lemma \ref{lem:properties_theta} (iii)
\begin{align*}
\biint_{Q_s^{(\theta_s)}(z_o)} \Big[ |D\u^m|^2\chi_{\Omega_T}+G^2 \Big]  \d x\d t &\leq \frac{|Q_{4R}|}{|Q_s^{(\theta_s)}|}\biint_{Q_{4R}} \Big[ |D\u^m|^2\chi_{\Omega_T}+G^2 \Big]  \d x\d t \\
&\leq  \frac{|Q_{4R}|}{|Q_s|} \theta_s^{m-1} \lambda_o^{m+1} \\
&\leq \left(\frac{4R}{s} \right)^{d+\frac{(d+2)(m-1)}{m+1}} \lambda_o^{2m} \\
&\leq   \left(\frac{4\hat cR}{R_2-R_1} \right)^{d+\frac{(d+2)(m-1)}{m+1}}\lambda_o^{2m}\\
&=B^{2m}\lambda_o^{2m}<\lambda^{2m}.
\end{align*}
By absolute continuity of the integral and the continuity of $s\mapsto
\theta_s$, there exists a maximal radius $\rho_{z_o}\in (0,\frac{R_2-R_1}{\hat c})$ such that
\begin{align} \label{lambda_rho}
\biint_{Q_{\rho_{z_o}}^{(\theta_{\rho_{z_o}})}(z_o)} \Big[ |D\u^m|^2\chi_{\Omega_T}+G^2 \Big]  \d x\d t =\lambda^{2m}.
\end{align}
The maximality of $\rho_{z_o}$ implies 
\begin{align} \label{lambda_s}
\biint_{Q_{s}^{(\theta_{s})}(z_o)} \Big[ |D\u^m|^2\chi_{\Omega_T}+G^2 \Big]  \d x\d t <\lambda^{2m}
\end{align}
for any $s \in (\rho_{z_o},R]$.

\subsection{Reverse H\"older inequalities}\label{sec:reverse-hoelder}
As before we consider $z_o\in \mathbf{E}(R_1,\lambda)$ and abbreviate
$\theta_{\rho_{z_o}}\equiv \theta_{z_o,\rho_{z_o}}$. From now on we
denote the exponent $q<2$ as the maximum of the Sobolev exponents $q$ in Lemma~\ref{lem:Poincare-Sobolev} and $ \frac{2n}{d}$ in Lemma~\ref{lem:initialSobolev}.
We distinguish between the non-degenerate and the degenerate case,
which correspond to the cases $\tilde \rho_{z_o}\leq 2 \rho_{z_o}$ and $\tilde \rho_{z_o}>2\rho_{z_o}$.

\subsubsection{The non-degenerate case $\tilde \rho_{z_o}\leq 2 \rho_{z_o}$.}

In this case, we note that the cylinder $Q_{\tilde
  \rho_{z_o}}^{(\theta_{\rho_{z_o}})}(z_o)$ is intrinsic since $\tilde
\rho_{z_o}<R$.
We first consider the boundary case $\tilde \rho_{z_o}\geq d_o$.
Lemma \ref{lem:properties_theta} (i) and the
fact that $Q_{\tilde \rho_{z_o}}^{(\theta_{\rho_{z_o}})}(z_o)$ is intrinsic imply
\begin{align*}
  \biint_{Q_{4 \tilde\rho_{z_o}}^{(\theta_{\rho_{z_o})}}(z_o)}
  &2\frac{|\hat\u^m-\hat\g^m|^2+|\hat g|^{2m}}{(4 \tilde\rho_{z_o})^2} \d x\d t
  \le
  \theta_{\rho_{z_o}}^{2m}  \\
  &\le 
  2^{d+2}\biint_{Q_{2 \tilde\rho_{z_o}}^{(\theta_{\rho_{z_o})} }(z_o)}
  2\frac{|\hat\u^m-\hat\g^m|^2+|\hat g|^{2m}}{(2\tilde\rho_{z_o})^2} \d x\d t.
\end{align*}
Since $\tilde \rho_{z_o} \geq d_o$, the
cylinder $Q_{2\tilde\rho_{z_o}}^{(\theta_{\rho_{z_o}})}(z_o)$
intersects or touches the lateral boundary. Hence, we are in position to use
Lemma \ref{lem:reverse_lateral_intrinsic} on this cylinder
 to obtain
\begin{align*}
\frac{1}{|Q_{\rho_{z_o}}^{(\theta_{\rho_{z_o}})}|}&\iint_{Q_{\rho_{z_o}}^{(\theta_{\rho_{z_o}})}(z_o)\cap \Omega_T} |D\u^m|^2 \d x\d t \\
&\leq \frac{c}{|Q_{2\tilde\rho_{z_o}}^{(\theta_{\rho_{z_o}})}|}\iint_{Q_{2\tilde\rho_{z_o}}^{(\theta_{\rho_{z_o}})}(z_o)\cap \Omega_T} |D\u^m|^2 \d x\d t \\
&\leq \left( \frac{c}{|Q_{16\tilde\rho_{z_o}}^{(\theta_{\rho_{z_o}})}|}\iint_{Q_{16\tilde\rho_{z_o}}^{(\theta_{\rho_{z_o}})}(z_o)\cap \Omega_T} |D\u^m|^q \d x\d t \right)^{\frac2q}+ c\biint_{Q_{16\tilde\rho_{z_o},+}^{(\theta_{\rho_{z_o}})}} G^2\d x\d t\\
&\leq \left( \frac{c}{|Q_{32\rho_{z_o}}^{(\theta_{\rho_{z_o}})}|}\iint_{Q_{32\rho_{z_o}}^{(\theta_{\rho_{z_o}})}(z_o)\cap \Omega_T} |D \u^m|^q \d x\d t \right)^{\frac2q}+ c\biint_{Q_{32\rho_{z_o},+}^{(\theta_{\rho_{z_o}})}} G^2\d x\d t,
\end{align*}
which is the desired reverse H\"older inequality. 
In the remaining case
$\tilde \rho_{z_o}<d_o$, we are either in the interior case
(if $Q_{2\tilde\rho_{z_o}}^{(\theta_{\rho_{z_o}})}\Subset\Omega_T$) or we
might intersect with the initial boundary.
Therefore, Lemma \ref{lem:properties_theta} (i) and the
fact that $Q_{\tilde \rho_{z_o}}^{(\theta_{\rho_{z_o}})}(z_o)$ is intrinsic imply
\begin{align*}
  \biint_{Q_{2\tilde\rho_{z_o}}^{(\theta_{\rho_{z_o}})}(z_o)} 
  \frac{|\hat u|^{2m}}{(2 \tilde\rho_{z_o})^2} \d x\d t
  \le
  \theta_{\rho_{z_o}}^{2m}  
  =
  \biint_{Q_{\tilde\rho_{z_o}}^{(\theta_{\rho_{z_o}})}(z_o) }
  \frac{|\hat u|^{2m}}{\tilde\rho_{z_o}^2} \d x\d t,
\end{align*}
 so that the conditions in
Lemma~\ref{lem:reverse_initial_intrinsic} are satisfied for the
cylinder $Q_{\tilde \rho_{z_o}}^{(\theta_{\rho_{z_o}})}(z_o)$. Hence, we obtain
\begin{align*}
\frac{1}{|Q_{\rho_{z_o}}^{(\theta_{\rho_{z_o}})}|}&\iint_{Q_{\rho_{z_o},+}^{(\theta_{\rho_{z_o}})}(z_o)} |D \u^m|^2 \d x\d t \\
&\leq \frac{c}{|Q_{\tilde\rho_{z_o}}^{(\theta_{\rho_{z_o}})}|}\iint_{Q_{\tilde\rho_{z_o}}^{(\theta_{\rho_{z_o}})}(z_o)} |D \hat \u^m|^2 \d x\d t \\
&\leq \left( \frac{c}{|Q_{2\tilde\rho_{z_o}}^{(\theta_{\rho_{z_o}})}|}\iint_{Q_{2\tilde\rho_{z_o}}^{(\theta_{\rho_{z_o}})}(z_o)} |D \hat \u^m|^q \d x\d t \right)^{\frac2q}+ c\biint_{Q_{2\tilde\rho_{z_o},+}^{(\theta_{\rho_{z_o}})}} G^2\d x\d t\\
&\leq \left( \frac{c}{|Q_{4\rho_{z_o}}^{(\theta_{\rho_{z_o}})}|}\iint_{Q_{4\rho_{z_o}}^{(\theta_{\rho_{z_o}})}(z_o)} |D\hat \u^m|^q \d x\d t \right)^{\frac2q}+ c\biint_{Q_{4\rho_{z_o},+}^{(\theta_{\rho_{z_o}})}} G^2\d x\d t.
\end{align*}
This completes the treatment of the case $\tilde \rho_{z_o}\le 2\rho_{z_o}$.

\subsubsection{The degenerate case $\tilde \rho_{z_o}>2\rho_{z_o}$}

The main objective in this case is the proof of the claim
\begin{equation}
  \label{claim-degenerate}
  \theta_{\rho_{z_o}}^{2m}
  \le
  c\,\biint_{Q_{\rho_{z_o}}^{(\theta_{\rho_{z_o}})}(z_o)} \Big[ |D\u^m|^2\chi_{\Omega_T}+G^2 \Big]  \d x\d t
\end{equation}
for a universal constant $c$. For the derivation of this property, we
distinguish between various cases. 
First, we observe that in the case $\theta_{\rho_{z_o}}=\lambda_o$,
the claim is immediate because 
\begin{align*}
\theta_{\rho_{z_o}}^{2m} =\lambda_o^{2m} <\lambda^{2m}=\biint_{Q_{\rho_{z_o}}^{(\theta_{\rho_{z_o}})}(z_o)} \Big[ |D\u^m|^2\chi_{\Omega_T}+G^2 \Big]  \d x\d t
\end{align*}
by \eqref{lambda_rho}.
Therefore, we may assume that $\theta_{\rho_{z_o}}>\lambda_o$, in which
case the cylinder $Q_{\tilde \rho_{z_o}}^{(\theta_{\rho_{z_o}})}(z_o)$
is intrinsic. We first consider the case $\tilde\rho_{z_o}<d_o$. We
use the Poincar\'e inequality from
Lemma~\ref{lem:initialSobolev}, inequality~\eqref{lambda_s} and
Lemma~\ref{lem:properties_theta}\,(i) with $s=\frac12
\tilde\rho_{z_o}>\rho_{z_o}$ to obtain 
\begin{align*}
  \theta_{\rho_{z_o}}^m
  &=
  \left[
  \biint_{Q_{\tilde\rho_{z_o}}^{(\theta_{\rho_{z_o}})}(z_o)}
  \frac{|\hat u|^{2m}}{\tilde \rho_{z_o}^2}\d x\d t\right]^{\frac{1}{2}}\\
&\leq\left[
  \biint_{Q_{\tilde\rho_{z_o}}^{(\theta_{\rho_{z_o}})}(z_o)}\frac{\big|\hat
    \u^m-(\hat \u^m)_{z_o;\frac 12 \tilde
      \rho_{z_o}}^{(\theta_{\rho_{z_o}})}\big|^{2}}{\tilde
    \rho_{z_o}^2} \d x\d t\right]^{\frac{1}{2}}
    +\frac{\big|(\hat \u^m)_{z_o;\frac 12 \tilde \rho_{z_o}}^{(\theta_{\rho_{z_o}})}\big|}{\tilde\rho_{z_o}}
     \\
&\leq c\left[  \biint_{Q_{\tilde\rho_{z_o}}^{(\theta_{\rho_{z_o}})}(z_o)} \Big[ |D\u^m|^2\chi_{\Omega_T}+G^2 \Big] \d x \d t \right]^{\frac{1}{2}} + \biint_{Q_{\frac 12 \tilde\rho_{z_o}}^{(\theta_{\rho_{z_o}})}(z_o)} \frac{|\hat u|^m}{\tilde\rho_{z_o}}\d x \d t\\
&\leq c \lambda^m + \tfrac12\left[ \biint_{Q_{\frac 12\tilde\rho_{z_o}}^{(\theta_{\rho_{z_o}})}(z_o)}\frac{|\hat u|^{2m}}{(\frac 12\tilde \rho_{z_o})^2} \d x\d t\right]^{\frac{1}{2}} \\
&\leq c \lambda^m +\tfrac12 \theta_{\rho_{z_o}}^m.
\end{align*}
This implies $\theta_{\rho_{z_o}}^{2m}\le c\lambda^{2m}$, which in
turn yields claim
\eqref{claim-degenerate} in view of property~\eqref{lambda_rho}.
Next, we turn our attention to the case $\tilde \rho_{z_o}\geq
d_o$. Now we use Lemma~\ref{lem:Poincare} on the cylinder
$Q_{8\tilde\rho_{z_o}}^{(\theta_{\tilde\rho_{z_o}})}(z_o)$, which is
possible since $\tilde\rho_{z_o}\ge d_o$ implies $B_{8\tilde\rho_{z_o}/3}\setminus\Omega\neq\emptyset$.
Moreover, we use Lemmas~\ref{lem:uavetoa} and \ref{lem:poincare_g} and
then inequality \eqref{lambda_s}. This leads us to the estimate
\begin{align}\label{degenerate-pre-claim}
\theta_{\rho_{z_o}}^m&= \left[ \biint_{Q_{\tilde\rho_{z_o}}^{(\theta_{\rho_{z_o}})}(z_o)} 2\frac{|\hat\u^m-\hat \g^m|^2+|\hat g|^{2m}}{\tilde \rho_{z_o}^2}\d x\d t\right]^{\frac{1}{2}}\\\nonumber
&\leq \left[ \biint_{Q_{\tilde\rho_{z_o}}^{(\theta_{\rho_{z_o}})}(z_o)}2\frac{|\hat\u^m-\hat \g^m|^2}{\tilde \rho_{z_o}^2} \d x\d t\right]^{\frac{1}{2}}\\\nonumber
&\quad+\left[
  \biint_{Q_{\tilde\rho_{z_o}}^{(\theta_{\rho_{z_o}})}(z_o)}2\frac{\big|\hat
    \g^m-(\hat \g^m)_{z_o;\frac 12 \tilde
      \rho_{z_o}}^{(\theta_{\rho_{z_o}})}\big|^{2}}{\tilde
    \rho_{z_o}^2} \d x\d t\right]^{\frac{1}{2}}+ \sqrt 2\,\frac{\babs{(\hat \g^m)_{z_o;\frac 12 \tilde \rho_{z_o}}^{(\theta_{\rho_{z_o}})} } }{\tilde\rho_{z_o}} \\\nonumber
&\leq c\left[  \biint_{Q_{8\tilde\rho_{z_o}}^{(\theta_{8\tilde\rho_{z_o}})}(z_o)} \Big[ |D\u^m|^2\chi_{\Omega_T}+G^2 \Big] \d x \d t \right]^{\frac{1}{2}} +\tfrac12\left[ \biint_{Q_{\frac 12\tilde\rho_{z_o}}^{(\theta_{\rho_{z_o}})}(z_o)}2\frac{|\hat g|^{2m}}{(\frac 12\tilde \rho_{z_o})^2} \d x\d t\right]^{\frac{1}{2}}\\\nonumber
&\leq c \lambda^m +\tfrac12\left[ \biint_{Q_{\frac 12\tilde\rho_{z_o}}^{(\theta_{\rho_{z_o}})}(z_o)}2\frac{|\hat g|^{2m}}{(\frac 12\tilde \rho_{z_o})^2} \d x\d t\right]^{\frac{1}{2}}.
\end{align}
For the estimate of the last integral, we distinguish further between
the cases $\frac12 \tilde\rho_{z_o}\ge d_o$ and $\frac12 \tilde\rho_{z_o}<d_o\leq \tilde \rho_{z_o}$. 
In the first case, Lemma~\ref{lem:properties_theta}\,(i) with $s=\frac12
\tilde\rho_{z_o}\ge\rho_{z_o}$, which satisfies
$s\ge d_o$, yields
\begin{equation*}
  \tfrac12\left[ \biint_{Q_{\frac
        12\tilde\rho_{z_o}}^{(\theta_{\rho_{z_o}})}(z_o)}2\frac{|\hat
      g|^{2m}}{(\frac 12\tilde \rho_{z_o})^2} \d x\d
    t\right]^{\frac{1}{2}}
   \le
   \tfrac12\theta_{\rho_{z_o}}^m.
\end{equation*}
In the second case $\frac12 \tilde\rho_{z_o}<d_o\leq \tilde
\rho_{z_o}$, we estimate 
\begin{align*}
  &\tfrac12\left[ \biint_{Q_{\frac
        12\tilde\rho_{z_o}}^{(\theta_{\rho_{z_o}})}(z_o)}2\frac{|\hat
      g|^{2m}}{(\frac 12\tilde \rho_{z_o})^2} \d x\d
    t\right]^{\frac{1}{2}}\\
  &\qquad\le
  \tfrac12\left[ \biint_{Q_{\frac
        12\tilde\rho_{z_o}}^{(\theta_{\rho_{z_o}})}(z_o)}2\frac{|\hat\g^m-\hat
     \u^m|^{2}}{(\frac 12\tilde \rho_{z_o})^2} \d x\d
    t\right]^{\frac{1}{2}}
  +
  \tfrac1{\sqrt2} \left[ \biint_{Q_{\frac 12\tilde\rho_{z_o}}^{(\theta_{\rho_{z_o}})}(z_o)}\frac{|\hat u|^{2m}}{(\frac 12\tilde \rho_{z_o})^2} \d x\d t \right]^{\frac{1}{2}}\\
&\qquad\leq c\left[  \biint_{Q_{8\tilde\rho_{z_o}}^{(\theta_{8\tilde\rho_{z_o}})}(z_o)}\Big[ |D\u^m|^2\chi_{\Omega_T}+G^2 \Big] \d x \d t \right]^{\frac{1}{2}} + \tfrac1{\sqrt2}\,\theta_{\rho_{z_o}}^m\\
  &\qquad\le c\lambda^m+\tfrac1{\sqrt2}\,\theta_{\rho_{z_o}}^m,
\end{align*}
where we applied Lemma~\ref{lem:Poincare} on 
$Q_{8\tilde\rho_{z_o}}^{(\theta_{\tilde\rho_{z_o}})}(z_o)$,
Lemma~\ref{lem:properties_theta}\,(i) with $s=\frac
12\tilde\rho_{z_o}<d_o$ and \eqref{lambda_s}. In view of the last two
estimates, we infer from~\eqref{degenerate-pre-claim} that in both
cases, we have
\begin{equation*}
  \theta_{\rho_{z_o}}^m\le c\lambda^m+\tfrac1{\sqrt2}\,\theta_{\rho_{z_o}}^m.
\end{equation*}
By absorbing $\tfrac1{\sqrt2}\,\theta_{\rho_{z_o}}^m$ into the left-hand
side and recalling~\eqref{lambda_rho}, we obtain the claim
\eqref{claim-degenerate} in the remaining case $\tilde\rho_{z_o}\ge
d_o$. Hence, we have established~\eqref{claim-degenerate} in any
case.

Now we are in position to derive the reverse H\"older inequality in
the degenerate case. If $\rho_{z_o} < d_o$, we observe that
Lemma~\ref{lem:properties_theta}\,(i) implies
\begin{equation*}
  \biint_{Q_{2\rho_{z_o}}^{(\theta_{\rho_{z_o}})}}
  \frac{|\hat u|^{2m}}{(2\rho_{z_o})^2}\d x\d t
  \le
  \theta_{\rho_{z_o}}^{2m}.
\end{equation*}
Combined with~\eqref{claim-degenerate} and the fact
$B_{2\rho_{z_o}}(x_o)\subset\Omega$, this shows that
the assumptions of Lemma~\ref{lem:reverse_initial_intrinsic_2} are
satisfied, which provides us with the reverse H\"older inequality 
\begin{align*}
\frac{1}{|Q_{\rho_{z_o}}^{(\theta_{\rho_{z_o}})}|}&\iint_{Q_{\rho_{z_o}}^{(\theta_{\rho_{z_o}})}(z_o)} |D\hat \u^m|^2 \d x\d t \\
&\leq \left( \frac{c}{|Q_{2\rho_{z_o}}^{(\theta_{\rho_{z_o}})}|}\iint_{Q_{2\rho_{z_o}}^{(\theta_{\rho_{z_o}})}(z_o)} |D\hat \u^m|^q \d x\d t \right)^{\frac2q}+ c\,\biint_{Q_{2\rho_{z_o},+}^{(\theta_{\rho_{z_o}})}} G^2\d x\d t.
\end{align*}
On the other hand if $\rho_{z_o} \geq d_o$, we infer from
Lemma~\ref{lem:properties_theta}\,(i) that 
\begin{equation*}
  \biint_{Q_{4\rho_{z_o}}^{(\theta_{\rho_{z_o}})}}
  2\frac{|\hat \u^m-\hat \g^m|^2+|\hat g|^{2m}}{(4\rho_{z_o})^2}\d x\d t
  \le
  \theta_{\rho_{z_o}}^{2m}.
\end{equation*}
Because of~\eqref{claim-degenerate} and $\rho_{z_o} \geq d_o$, we can
thus apply Lemma~\ref{lem:reverse_lateral_sub-intrinsic}
with $\rho = 2\rho_{z_o}$ and obtain
\begin{align*}
\frac{1}{|Q_{\rho_{z_o}}^{(\theta_{\rho_{z_o}})}|}&\iint_{Q_{\rho_{z_o}}^{(\theta_{\rho_{z_o}})}(z_o) \cap \Omega_T} |D \u^m|^2 \d x\d t \\
&\leq \frac{c}{|Q_{2\rho_{z_o}}^{(\theta_{\rho_{z_o}})}|}\iint_{Q_{2\rho_{z_o}}^{(\theta_{\rho_{z_o}})}(z_o) \cap \Omega_T} |D \u^m|^2 \d x\d t \\
&\leq \left(
  \frac{c}{|Q_{16\rho_{z_o}}^{(\theta_{\rho_{z_o}})}|}\iint_{Q_{16\rho_{z_o}}^{(\theta_{\rho_{z_o}})}(z_o)
    \cap \Omega_T} |D \u^m|^q \d x\d t \right)^{\frac2q}
+
c\,\biint_{Q_{16\rho_{z_o},+}^{(\theta_{\rho_{z_o}})}} G^2\d x\d t.
\end{align*}
This concludes the proof for the degenerate case.
In summary, in any case we have established the reverse H\"older inequality 
\begin{align} \label{reverse_holder} \notag
\frac{1}{\babs{Q_{\rho_{z_o}}^{(\theta_{\rho_{z_o}})}}}&\iint_{Q_{\rho_{z_o}}^{(\theta_{\rho_{z_o}})}(z_o) \cap \Omega_T} |D \u^m|^2 \d x\d t \\
&\leq \left( \frac{c}{\babs{Q_{32\rho_{z_o}}^{(\theta_{\rho_{z_o}})}}}\iint_{Q_{32\rho_{z_o}}^{(\theta_{\rho_{z_o}})}(z_o) \cap \Omega_T} |D \u^m|^q \d x\d t \right)^{\frac2q}+ c\biint_{Q_{32\rho_{z_o},+}^{(\theta_{\rho_{z_o}})}} G^2\d x\d t.
\end{align}

\subsection{Estimate on super-level sets}

We define the super-level set for function $G$ as 
$$
\G(r,\lambda):= \left\{z \in Q_r: z \text{ is a Lebesgue point of } G \text{ and } |G(z)|>\lambda^m \right\}.
$$

For $\eta \in (0,1]$ we have

\begin{align*}
\lambda^{2m} &= \biint_{Q_{\rho_{z_o}}^{(\theta_{\rho_{z_o}})}(z_o)} \Big[ |D\u^m|^2\chi_{\Omega_T}+G^2 \Big]  \d x\d t \\
&\leq c\left(\biint_{Q_{32\rho_{z_o}}^{(\theta_{\rho_{z_o}})}(z_o)} |D \u^m|^q \chi_{\Omega_T} \d x\d t \right)^{\frac2q}+ c\,\biint_{Q_{32\rho_{z_o},+}^{(\theta_{\rho_{z_o}})}} G^2\d x\d t \\
&\leq c \eta^{2m} \lambda^{2m} + \left(  \frac{c}{\babs{Q_{32\rho_{z_o}}^{(\theta_{\rho_{z_o}})}}} \iint_{Q_{32\rho_{z_o}}^{(\theta_{\rho_{z_o}})}(z_o) \cap \mathbf{E}(R_2,\eta \lambda)} |D \u^m|^q \d x\d t \right)^{\frac2q} \\
&\quad +  \frac{c}{\babs{Q_{32\rho_{z_o}}^{(\theta_{\rho_{z_o}})}}} \iint_{Q_{32\rho_{z_o},+}^{(\theta_{\rho_{z_o}})} \cap \G(R_2,\eta \lambda)} G^2\d x\d t,
\end{align*}
by using~\eqref{lambda_rho} and~\eqref{reverse_holder}. Now by choosing $\eta^{2m} = \frac{1}{2c}$ we can absorb the first term into the left-hand side. In order to treat the second term we estimate
\begin{align*}
\Bigg( \frac{c}{\babs{Q_{32\rho_{z_o}}^{(\theta_{\rho_{z_o}})}}} &\iint_{Q_{32\rho_{z_o}}^{(\theta_{\rho_{z_o}})}(z_o) \cap \mathbf{E}(R_2,\eta \lambda)} |D \u^m|^q \d x\d t \Bigg)^{\frac{2}{q} -1} \\
&\leq \left( \biint_{Q_{32\rho_{z_o}}^{(\theta_{\rho_{z_o}})}(z_o)} |D \u^m|^2 \chi_{\Omega_T} \d x\d t \right)^{1-\frac{q}{2}} \\
&\leq c \lambda^{m(2-q)},
\end{align*} 
where we used H\"older's inequality and inequality~\eqref{lambda_s}.
Collecting the estimates above we have
\begin{align*}
\lambda^{2m} \babs{Q_{32\rho_{z_o}}^{(\theta_{\rho_{z_o}})}} \leq c &\iint_{Q_{32\rho_{z_o}}^{(\theta_{\rho_{z_o}})}(z_o) \cap \mathbf{E}(R_2,\eta \lambda)} \lambda^{m(2-q)} \babs{D \u^m}^2 \d x\d t \\
&+ c \iint_{Q_{32\rho_{z_o},+}^{(\theta_{\rho_{z_o}})} \cap \G(R_2,\eta \lambda)} G^2\d x\d t.
\end{align*}
On the other hand, inequality~\eqref{lambda_s}, the monotonicity of the mapping $\rho \mapsto \theta_{\rho}$ and Lemma~\ref{lem:properties_theta}\,(ii) imply that
\begin{align*}
\lambda^{2m} > \biint_{Q_{\hat c \rho_{z_o}}^{(\theta_{\hat c \rho_{z_o}})}(z_o)} |D\u^m|^2 \chi_{\Omega_T}\,  \d x\d t \geq \hat{c}^\frac{(1-m)(d+2)}{m+1} \biint_{Q_{\hat c \rho_{z_o}}^{(\theta_{\rho_{z_o}})}(z_o)} |D\u^m|^2 \chi_{\Omega_T}\,  \d x\d t.
\end{align*}
The two previous estimates lead to
\begin{align} \label{covering_estimate}  \notag
\iint_{Q_{\hat c \rho_{z_o}}^{(\theta_{\rho_{z_o}})}(z_o)} |D\u^m|^2 \chi_{\Omega_T}\,  \d x\d t \leq c &\iint_{Q_{32\rho_{z_o}}^{(\theta_{\rho_{z_o}})}(z_o) \cap \mathbf{E}(R_2,\eta \lambda)} \lambda^{m(2-q)} \babs{D \u^m}^2 \d x\d t \\ 
&+ c \iint_{Q_{32\rho_{z_o},+}^{(\theta_{\rho_{z_o}})} \cap \G(R_2,\eta \lambda)} G^2\d x\d t
\end{align} 
for every $z_o\in\mathbf{E}(R_1,\lambda)$. 
Next we cover the set $\mathbf{E}(R_1,\lambda)$ by the collection of cylinders $\mathcal{F} := \{ Q_{32\rho_{z_o}}^{(\theta_{z_o;\rho_{z_o}})}(z_o) \}_{z_o \in \mathbf{E}(R_1,\lambda)}$. By Vitali-type covering lemma~\ref{lem:vitali} there exists a countable disjoint sub-collection
$$
\left\{ Q_{32\rho_{zi}}^{(\theta_{z_i;\rho_{z_i}})}(z_i) \right\}_{i \in \N} \subset \mathcal{F},
$$
such that 
$$
\mathbf{E}(R_1,\lambda) \subset \bigcup_{i\in \N} Q_{\hat c \rho_{z_i}}^{(\theta_{z_i;\rho_{z_i}})} (z_i) \subset Q_{R_2}
$$
holds true. This and~\eqref{covering_estimate} imply
\begin{align*}
\iint_{\mathbf{E}(R_1,\lambda)} \babs{D\u^m}^2\, \d x \d t &\leq \sum_{i=1}^\infty \iint_{Q_{\hat c \rho_{z_i}}^{(\theta_{z_i;\rho_{z_i}})} (z_i)} \babs{D\u^m}^2 \chi_{\Omega_T} \, \d x \d t \\
&\leq c \sum_{i=1}^\infty \iint_{Q_{32\rho_{z_i}}^{(\theta_{z_i;\rho_{z_i}})}(z_i) \cap \mathbf{E}(R_2,\eta \lambda)} \lambda^{m(2-q)} \babs{D \u^m}^2 \d x\d t \\
&\quad + c \sum_{i=1}^\infty \iint_{Q_{32\rho_{z_i},+}^{(\theta_{z_i;\rho_{z_i}})} (z_i) \cap \G(R_2,\eta \lambda)} G^2\d x\d t \\
&\leq c \iint_{\mathbf{E}(R_2,\eta \lambda)} \lambda^{m(2-q)} \babs{D \u^m}^2 \d x\d t + c \iint_{\G(R_2,\eta \lambda)} G^2\d x\d t.
\end{align*}
In the set $\mathbf{E}(R_1,\eta \lambda) \setminus \mathbf{E}(R_1,\lambda)$ by definition $\babs{D \u^m}^2 \leq \lambda^{2m}$ a.e.. Thus we can estimate
\begin{align*}
\iint_{\mathbf{E}(R_1,\eta \lambda) \setminus \mathbf{E}(R_1,\lambda)} \babs{D\u^m}^2\, \d x \d t \leq \iint_{\mathbf{E}(R_2,\eta \lambda)} \lambda^{m(2-q)} \babs{D\u^m}^q\, \d x \d t.
\end{align*}
Now by combining the previous two inequalities and replacing $\eta \lambda $ by $\lambda$, we obtain that 
\begin{align} \label{superlevel-reverse-holder} \notag
\iint_{\mathbf{E}(R_1,\lambda)} &\babs{D\u^m}^2 \, \d x \d t  \\
&\leq c \iint_{\mathbf{E}(R_2,\lambda)} \lambda^{m(2-q)} \babs{D\u^m}^q \, \d x \d t + c \iint_{\G(R_2,\lambda)} G^2\, \d x \d t
\end{align}
holds true for any $\lambda \geq \eta B \lambda_o =: \lambda_1$.

\subsection{Proof of the gradient estimate}

With estimate~\eqref{superlevel-reverse-holder} on super-level sets and using Fubini-type arguments we are finally able to prove the higher integrability for the gradient of the solution. In order to ensure that quantities we end up re-absorbing are finite we consider truncations. For $k > \lambda_1$ we define 
$$
\babs{D\u^m }_k := \min \Big\{ \babs{D\u^m}, k^m \Big\},
$$ 
and the corresponding super-level set as
$$
\mathbf{E}_k(r,\lambda) := \Big\{  z\in Q_r \cap \Omega_T: \babs{D\u^m}_k (z) > \lambda^m  \Big\}.
$$
With this notation and estimate~\eqref{superlevel-reverse-holder} we have 
\begin{align*}
\iint_{\mathbf{E}_k(R_1,\lambda)} &\babs{D\u^m}_k^{2-q}  \babs{D\u^m}^q \, \d x \d t  \\
&\leq c \iint_{\mathbf{E}_k(R_2,\lambda)} \lambda^{m(2-q)} \babs{D\u^m}^q \, \d x \d t + c \iint_{\G(R_2,\lambda)} G^2\, \d x \d t
\end{align*}
for $k > \lambda_1$. Here we exploited the facts that $\babs{D\u^m}_k \leq \babs{D\u^m}$ a.e., $\mathbf{E}_k(r,\lambda) = \mathbf{E}(r,\lambda)$ if $k > \lambda$ and $\mathbf{E}_k(r,\lambda) = \emptyset$ if $k \leq \lambda$.

Let $\eps \in (0,1]$. We multiply the inequality above by
$\lambda^{\eps m -1}$ and integrate over the interval $(\lambda_1, \infty)$. By using Fubini's theorem, on the left-hand side we have
\begin{align*}
\int_{\lambda_1}^{\infty} &\lambda^{\eps m -1} \Bigg(\iint_{\mathbf{E}_k(R_1,\lambda)} \babs{D\u^m}_k^{2-q}  \babs{D\u^m}^q \, \d x \d t  \Bigg)\, \d \lambda \\
&= \iint_{\mathbf{E}_k(R_1,\lambda_1)} \babs{D\u^m}_k^{2-q}  \babs{D\u^m}^q  \Bigg( \int_{\lambda_1}^{|D\u^m|_k^\frac1m} \lambda^{\eps m - 1}\, \d \lambda \Bigg)  \, \d x \d t \\
&= \frac{1}{\eps m} \iint_{\mathbf{E}_k(R_1,\lambda_1)} \Big( \babs{D\u^m}_k^{2-q + \eps}  \babs{D\u^m}^q - \lambda_1^{\eps m} \babs{D\u^m}_k^{2-q}  \babs{D\u^m}^q  \Big)  \, \d x \d t.
\end{align*}
For the first term on the right-hand side we obtain
\begin{align*}
\int_{\lambda_1}^{\infty} &\lambda^{m(2-q + \eps) - 1} \Bigg(\iint_{\mathbf{E}_k(R_2,\lambda)} \babs{D\u^m}^q \, \d x \d t  \Bigg)\, \d \lambda \\
&= \iint_{\mathbf{E}_k(R_2,\lambda_1)} \babs{D\u^m}^q  \Bigg( \int_{\lambda_1}^{|D\u^m|_k^\frac1m} \lambda^{m(2-q + \eps) - 1}   \, \d \lambda \Bigg)  \, \d x \d t \\
&\leq \frac{1}{m (2-q + \eps)} \iint_{\mathbf{E}_k(R_2,\lambda_1)} \babs{D\u^m}_k^{2-q+\eps} \babs{D\u^m}^q \, \d x \d t \\
&\leq \frac{1}{m (2-q)} \iint_{\mathbf{E}_k(R_2,\lambda_1)} \babs{D\u^m}_k^{2-q+\eps} \babs{D\u^m}^q \, \d x \d t,
\end{align*}
and for the last term
\begin{align*}
\int_{\lambda_1}^{\infty} \lambda^{\eps m -1} \Bigg(\iint_{\G(R_2,\lambda)} G^2 \, \d x \d t  \Bigg)\, \d \lambda &= \iint_{\G(R_2,\lambda_1)} G^2  \Bigg( \int_{\lambda_1}^{G^\frac1m} \lambda^{\eps m -1}\, \d \lambda \Bigg) \, \d x \d t \\
&\leq \frac{1}{\eps m} \iint_{\G(R_2,\lambda_1)} G^{2+\eps} \, \d x \d t \\
&\leq \frac{1}{\eps m} 	 \iint_{Q_{2R,+}} G^{2+\eps} \, \d x \d t.
\end{align*}
Combining the estimates and multiplying by $\eps m$ we obtain
\begin{align*}
\iint_{\mathbf{E}_k(R_1,\lambda_1)} &\babs{D\u^m}_k^{2-q + \eps}  \babs{D\u^m}^q \, \d x \d t \\
&\leq \lambda_1^{\eps m} \iint_{\mathbf{E}_k(R_1,\lambda_1)} \babs{D\u^m}_k^{2-q}  \babs{D\u^m}^q   \, \d x \d t \\
&\quad + \frac{c\, \eps}{2-q} \iint_{\mathbf{E}_k(R_2,\lambda_1)} \babs{D\u^m}_k^{2-q+\eps} \babs{D\u^m}^q \, \d x \d t \\
&\quad + c \iint_{Q_{2R,+}} G^{2+\eps} \, \d x \d t.
\end{align*}
For the complement $(Q_{R_1} \cap \Omega_T) \setminus \mathbf{E}_k(R_1,\lambda_1)$ we estimate
\begin{align*}
\iint_{Q_{R_1} \setminus \mathbf{E}_k(R_1,\lambda_1)} &\babs{D\u^m}_k^{2-q + \eps}  \babs{D\u^m}^q \chi_{\Omega_T} \, \d x \d t \\
&\leq \lambda_1^{\eps m} \iint_{Q_{R_1} \setminus \mathbf{E}_k(R_1,\lambda_1)} \babs{D\u^m}_k^{2-q }  \babs{D\u^m}^q \chi_{\Omega_T} \, \d x \d t.
\end{align*}
Adding the two previous estimates we deduce
\begin{align*}
\iint_{Q_{R_1}} &\babs{D\u^m}_k^{2-q + \eps}  \babs{D\u^m}^q \chi_{\Omega_T} \, \d x \d t \\
&\leq \frac{c_* \eps}{2-q} \iint_{Q_{R_2}} \babs{D\u^m}_k^{2-q+\eps} \babs{D\u^m}^q \chi_{\Omega_T} \, \d x \d t \\
&\quad + \lambda_1^{\eps m} \iint_{Q_{2R}}  \babs{D\u^m}^2  \chi_{\Omega_T} \, \d x \d t + c \iint_{Q_{2R,+}} G^{2+\eps} \, \d x \d t
\end{align*}
for $c_* = c_*(m,n,N,\nu,L,\alpha,\mu, \rho_o) \geq 1$. Next we choose 
\begin{align*}
\eps_o := \frac{2-q}{2c_*} < 1
\end{align*}
and assume that $\eps\le\eps_o$. 
Now $\lambda_1^\eps = (\eta B\lambda_o)^\eps \leq B \lambda_o^\eps$ since $\eta \leq 1$, $B > 1$ and $\eps < 1$. We obtain
\begin{align*}
\iint_{Q_{R_1}} &\babs{D\u^m}_k^{2-q + \eps}  \babs{D\u^m}^q \chi_{\Omega_T} \, \d x \d t \\
&\leq \frac12 \iint_{Q_{R_2}} \babs{D\u^m}_k^{2-q+\eps} \babs{D\u^m}^q \chi_{\Omega_T} \, \d x \d t \\
&+c \Bigg( \frac{R}{R_2-R_1}  \Bigg)^\frac{m(n+2)}{m+1} \lambda_o^{\eps m} \iint_{Q_{2R}}  \babs{D\u^m}^2  \chi_{\Omega_T} \, \d x \d t + c \iint_{Q_{2R,+}} G^{2+\eps} \, \d x \d t,
\end{align*}
for any $R_1,R_2$ satisfying $R \leq R_1 < R_2 \leq 2R$. By using Iteration Lemma~\ref{lem:iteration} we can re-absorb the first term into the left-hand side. Then by passing to the limit $k \to \infty$ and using Fatou's Lemma we can conclude
\begin{align*}
\iint_{Q_R \cap \Omega_T} &\babs{D\u^m}^{2+\eps} \, \d x \d t \\
&\leq c \lambda_o^{\eps m} \iint_{Q_{2R} \cap \Omega_T} \babs{D\u^m}^2 \, \d x\d t + c \iint_{Q_{2R,+}} G^{2+\eps} \, \d x \d t.
\end{align*}
%
%
Estimating $\lambda_o$ by means of \eqref{bound-lambda0} and the last
integral by \eqref{extension-bound} proves the estimate
\begin{align*}
&\iint_{Q_R \cap \Omega_T}\babs{D\u^m}^{2+\eps} \, \d x \d t \\
&\leq c  \Bigg( 1 + \biint_{Q_{8R}\cap\Omega_T} \frac{|\u^m -
  \g^m|^2}{R^2} \, \d x \d t \Bigg)^{\frac{\eps m}{m+1}}
\iint_{Q_{2R} \cap \Omega_T} \babs{D\u^m}^2 \, \d x\d t \\
&\quad
+c\bigg(\biint_{Q_{8R}\cap\Omega_T} \left[ G^{2+\eps}
   +\frac{|g |^{m(2+\eps)}}{R^{2+\eps}} \right] \, \d x \d
   t\bigg)^{\frac{2\eps m}{(2+\eps)(m+1)}}\iint_{Q_{2R} \cap \Omega_T} \babs{D\u^m}^2 \, \d x\d t \\
&\quad + c \iint_{Q_{2R}\cap\Omega_T} \left[ G^{2+\eps}+\frac{|g |^{m(2+\eps)}}{R^{2+\eps}} \right] \, \d x \d t,
\end{align*}
with $c=c(m,n,N,\nu,L,\alpha,\mu, \rho_o,c_E)$.
Finally, we note that we can replace the integrals over $Q_{8R}$ by
integrals over $Q_{2R}$ by a standard covering argument. More
precisely, we cover the cylinder $Q_R$ by smaller
cylinders $Q_{R/8}(z_i)$ with centers $z_i\in Q_R$, apply the preceding
estimate on each of the smaller cylinders and sum up the resulting
inequalities. This procedure leads to the asserted
estimate~\eqref{local-estimate}. The local estimate implies $|D\power um|\in
  L^{2+\eps}(\Omega_\tau)$ for every $\tau<T$. However, we can assume
  that the solution is given on the larger cylinder $\Omega_{2T}$ by
  reflecting the boundary values across the time slice
  $\Omega\times\{T\}$ and solving a Cauchy-Dirichlet problem on
  $\Omega\times[T,2T]$. Applying the preceding result on
  $\Omega_{2T}$, we deduce the remaining assertion $|D\power um|\in
  L^{2+\eps}(\Omega_T)$.
This completes the proof of Theorem
\ref{thm:main}.


\begin{thebibliography}{99}

\bibitem{AdimurthiByun:2018}
  K.~Adimurthi, S.~Byun.
  \newblock Boundary higher integrability for very weak solutions of
  quasilinear parabolic equations. Preprint 2018.
  
\bibitem{Boegelein:1}
\newblock V.~B\"ogelein.
\newblock Higher integrability for weak solutions of higher order degenerate parabolic systems. 
\newblock {\em Ann.~Acad.~Sci.~Fenn.~Math.} 33 (2008), no.~2, 387--412.

\bibitem{BDKS:2018}
\newblock V.~B\"ogelein, F.~Duzaar, J.~Kinnunen and C.~Scheven.
\newblock Higher integrability for doubly nonlinear parabolic systems.
Preprint.

\bibitem{Boegelein-Duzaar-Korte-Scheven}
\newblock V.~B\"ogelein, F.~Duzaar, R.~Korte and C.~Scheven.
\newblock The higher integrability of weak solutions of porous medium systems.
\newblock {\em Adv.~Nonlinear Anal.} 55 (2018), to appear.


\bibitem{BDM:pq}
\newblock V.~B\"ogelein, F.~Duzaar, P.~Marcellini.
\newblock Parabolic systems with $p,q$-growth: a variational approach.
\newblock  {\em Arch. Ration. Mech. Anal.} 210(1):219--267, 2013.

\bibitem{Boegelein-Duzaar-Scheven:2018}
\newblock V.~B\"ogelein, F.~Duzaar, and C.~Scheven.
\newblock Higher integrability for the singular porous medium system.
\newblock Preprint 2018.


\bibitem{Boegelein-Parviainen}
\newblock V.~B\"ogelein and M.~Parviainen.
\newblock Self-improving property of nonlinear higher order parabolic systems near the boundary. 
\newblock  {\em NoDEA Nonlinear Differential Equations Appl.} 17 (2010), no.~1, 21--54.

\bibitem{DiBe}
E.~DiBenedetto.
\newblock {\em Degenerate parabolic equations}.
\newblock Springer-Verlag, Universitytext xv, 387, New York, NY, 1993.

\bibitem{DbF2}
E.~DiBenedetto and A.~Friedman.
\newblock {H\"older} estimates for non-linear degenerate parabolic systems.
\newblock {\em J. Reine Angew. Math.}, 357:1--22, 1985.

\bibitem{DBGV-book} 
E. DiBenedetto, U. Gianazza, and V. Vespri. 
{\it Harnack's inequality for degenerate and singular parabolic equations}.
Springer Monographs in Mathematics, 2011.


\bibitem{Diening-Kaplicky-Schwarzacher}
\newblock L.~Diening, P.~Kaplick\'y, and S.~Schwarzacher.
\newblock BMO estimates for the $p$-Laplacian, 
\newblock {\em Nonlinear Anal.} 75 (2012), no. 2, 637--650.


\bibitem{Evans-Gariepy}
\newblock  L.~C.~Evans and R.~F.~Gariepy.
\newblock Measure theory and fine properties of functions.
\newblock {\em Studies in Advanced Mathematics. CRC Press, Boca Raton, FL} 1992.

\bibitem{Federer}
\newblock H.~Federer.
\newblock {\em Geometric measure theory.} 
\newblock Die Grundlehren der mathematischen Wissenschaften, Band 153, Springer-Verlag, New York, 1969.

\bibitem{Gehring}
\newblock F.~W.~Gehring.
\newblock The $L^p$-integrability of the partial derivatives of a quasiconformal mapping.
\newblock {\em Acta Math.} 130 (1973), 265--277.

\bibitem{Gianazza-Schwarzacher}
\newblock U.~Gianazza and S.~Schwarzacher.
\newblock Self-improving property of degenerate parabolic equations of porous medium-type.
\newblock Preprint 2016, to appear in Amer. J. Math.

\bibitem{Gianazza-Schwarzacher-singular}
\newblock U.~Gianazza and S.~Schwarzacher.
\newblock Self-improving property of the fast diffusion equation.
\newblock Preprint 2018.

\bibitem{Giaquinta:book}
\newblock M.~Giaquinta.
\newblock Multiple Integrals in the Calculus of Variations and Nonlinear Elliptic Systems.
\newblock Princeton University Press, Princeton, 1983.

\bibitem{Giaquinta-Modica:0}
\newblock M.~Giaquinta and G.~Modica.
\newblock Regularity results for some classes of higher order
nonlinear elliptic systems. 
\newblock {\em J.~Reine Ang.~Math.} 311/312 (1979) 145--169.


\bibitem{Giaquinta-Modica}
\newblock M.~Giaquinta and G.~Modica.
\newblock Partial regularity of minimizers of quasiconvex integrals
\newblock {\em Ann. Inst. H . Poincar\'e Anal. Non Lin\'eaire} 3 (1986) 185--208.

\bibitem{Giaquinta-Struwe}
\newblock M.~Giaquinta and M.~Struwe.
\newblock On the partial regularity of weak solutions of nonlinear parabolic systems.
\newblock {\em Math.~Z.} 179 (1982), no.~4, 437--451.

\bibitem{Giusti:book}
\newblock E.~Giusti,
\newblock {\em Direct Methods in the Calculus of Variations}.
\newblock World Scientific Publishing Company, Tuck Link, Singapore, 2003.

\bibitem{HaKoTu}
  P.~Haj\l asz, P.~Koskela, H.~Tuominen,
  \newblock Sobolev embeddings, extensions and measure density
  condition.
  \newblock \emph{J. Funct. Anal.} 254, no. 5, 1217--1234 (2008).

\bibitem{Hedberg}
\newblock L.I.~Hedberg.
\newblock Two approximation problems in function spaces.
\newblock {\em Ark.~Mat.} 16(1978), no.~1, 51--81. 

\bibitem{Heinonen-Kilpelainen-Martio}
\newblock J~Heinonen, T.~Kilpel\"ainen and O.~Martio.
\newblock Nonlinear potential theory of degenerate elliptic equations.
\newblock {\em Oxford Mathematical Monographs.} Oxford University Press, New York, 1993. 

\bibitem{Kilpelainen-Koskela}
\newblock T.~Kilpel{\"a}inen and P.~Koskela.
\newblock Global integrability of the gradients of solutions to partial differential equations.
\newblock {\em Nonlinear Anal.}, 23(7):899--909, 1994.

\bibitem{Kinnunen-Lewis:1}
\newblock J.~Kinnunen and J.~L.~Lewis.
\newblock Higher integrability for parabolic systems of $p$-Laplacian type.
\newblock {\em Duke Math.~J.} 102 (2000), no.~2, 253-271.

\bibitem{Kinnunen-Lewis:very-weak}
  J.~Kinnunen, J.~Lewis. 
  \newblock Very weak solutions of parabolic systems of p-Laplacian type.
  \newblock \emph{Ark. Mat.} 40(1):105--132, 2002.

\bibitem{Kinnunen-Lindqvist}
\newblock J.~Kinnunen and P.~Lindqvist.
\newblock Pointwise behaviour of semicontinuous supersolutions to a quasilinear parabolic equation.
\newblock {\em Ann. Mat. Pura Appl.} (4) 185(3):411--435, 2006.

\bibitem{Kinnunen-Parviainen}
\newblock J.~Kinnunen and M.~Parviainen.
\newblock Stability for degenerate parabolic equations.
\newblock {\em Adv.~Cal.~Var.} 3 (2010), no. 1, 29--48.


\bibitem{Lewis}
\newblock J.L.~Lewis.
\newblock Uniformly fat sets. 
\newblock {Trans. Am. Math. Soc.} 308 (1988), no.~1, 177--196. 

\bibitem{Meyers-Elcrat}
\newblock N.~G.~Meyers and A.~Elcrat.
\newblock Some results on regularity for solutions of non-linear elliptic
systems and quasi-regular functions.
\newblock {\em Duke Math. J.} 42 (1975), 121--136.

\bibitem{Parviainen}
\newblock M.~Parviainen.
\newblock Global gradient estimates for degenerate parabolic equations in nonsmooth domains. 
\newblock{\em Ann.~Mat.~Pura Appl.}  (4) 188 (2009), no.~2, 333--358.

\bibitem{Parviainen-singular}
  M.~Parviainen.
  \newblock Reverse H\"older inequalities for singular parabolic
  equations near the boundary.
  \newblock \emph{J. Differential Equations} 246(2):512--540, 2009.

\bibitem{Schwarzacher}
\newblock S.~Schwarzacher.
\newblock H\"older-Zygmund estimates for degenerate parabolic systems.
\newblock {\em J. Differential Equations}, 256 (2014), no. 7, 2423--2448.

\end{thebibliography}
\end{document}